\documentclass[11pt,reqno]{amsproc}

\usepackage{amsmath}
\usepackage{amssymb}
\usepackage{amssymb}
\usepackage{verbatim}

\def\a{\alpha}

\def\f{\varphi} 
\def\g{\gamma}
\def\s{\sigma}
\def\l{\lambda}
\def\vk{\varkappa}
\def\bold#1{\mbox{\bf #1}}

\def\N{\bold N}
\def\Z{\bold Z}

\def\Q{\bold Q}
\def\R{\bold R}
\def\C{\bold C}

\def\chr{ \operatorname{\rm char\,}}

\def\aut#1{\operatorname{Aut}(#1)}

\def\id{\mbox{\rm id}}

\def\codim{ \operatorname{codim}}

\def\Theory{\mbox{\rm Th}}

\def\inv{^{-1}} 
\def\av#1{\overline {#1}}

\def\str#1{\langle #1 \rangle}

\def\gl#1{\operatorname{ GL}(#1)}

\def\rl#1{\operatorname{ RL}(#1)}
\def\Gl#1{\mbox{\rm $\Gamma$L}(#1)}

\def\pgl#1{\operatorname{ PGL}(#1)}

\def\pGl#1{\mbox{\rm P$\Gamma$L}(#1)}

\def\End#1{\operatorname{End}(#1)}

\def\tspl#1{\mbox{\rm SL}}
\def\tGl#1{\mbox{\rm $\Gamma$L}}

\def\tpgl#1{\mbox{\rm PGL}}
\def\tpspl#1{\mbox{\rm PSL}}

\def\tEnd#1{\mbox{\rm End}}

\def\avx{\overline{x} }
\def\avy{\overline{y} }
\def\avz{\overline{z} }

\renewcommand{\le}{\leqslant}
\renewcommand{\ge}{\geqslant}

\newtheorem{Th}{\bf Theorem}[section]
\newtheorem{Cor}[Th]{\bf Corollary}
\newtheorem{Lem}[Th]{\bf Lemma}
\newtheorem{Prop}[Th]{\bf Proposition}
\newtheorem{Clm}[Th]{\bf Claim}

\newtheorem{clm}{\bf Claim}

\numberwithin{equation}{section}
\renewcommand{\theequation}{\thesection.\arabic{equation}}

\def\Cov{\mbox{\it Cov}}
\def\Fix{\operatorname{ Fix}}
\def\Rng{\operatorname{ Rng}}
\def\codim{\operatorname{ codim}\,}
\def\dcd{\operatorname{ dcd}\,}

\def\LII{\mbox{\bf L}_2}
\def\Mon{\operatorname{Mon}}
\def\cL{{\mathcal L}}
\def\TH{\operatorname{\rm Th}}

\def\avs{\overline{\sigma}}
\def\avpi{\overline{\pi}}

\def\avx{\overline{x}}
\def\avy{\overline{y}}
\def\avz{\overline{z}}

\def\avX{\overline{X}}
\def\avA{\overline{A}}

\def\rl#1{\text{\rm RL}(#1)}
\def\pEnd#1{\text{\rm PEnd}(#1)}

\def\ex#1{\text{\rm E}_{\mathfrak X}}

\def\Dvkom{D^{\vk}_{< \omega}}

\def\cP{{\mathcal P}}
\def\cE{{\mathcal E}}
\def\cV{{\mathcal V}}
\def\cD{{\mathcal D}}
\def\cPV{{\mathcal P}{\mathcal V}}
\def\cVD{{\mathcal V}{\mathcal D}}
\def\cEPV{{\mathcal E}{\mathcal P}{\mathcal V}}
\def\cPG{{\mathcal P}{\mathcal G}}

\def\cA{{\mathcal A}}
\def\cB{{\mathcal B}}

\def\cM{{\mathcal M}}
\def\cN{{\mathcal N}}

\def\fX{{\mathfrak X}}

\def\fL{{\mathfrak L}}

\renewcommand{\le}{\leqslant}
\renewcommand{\ge}{\geqslant}
\renewcommand{\wedge}{\, \& \,}

\def\Cov{\mbox{\it Cov}}
\def\Fix{\operatorname{ Fix}}
\def\Rng{\operatorname{ Rng}}
\def\codim{\operatorname{ codim}\,}
\def\dcd{\operatorname{ dcd}\,}

\def\LII{\mbox{\bf L}_2}
\def\Mon{\operatorname{Mon}}
\def\cL{{\mathcal L}}
\def\TH{\operatorname{\rm Th}}

\def\avs{\overline{\sigma}}
\def\avpi{\overline{\pi}}

\def\avx{\overline{x}}
\def\avy{\overline{y}}
\def\avz{\overline{z}}

\def\avX{\overline{X}}
\def\avA{\overline{A}}

\def\rl#1{\text{\rm RL}(#1)}
\def\pEnd#1{\text{\rm PEnd}(#1)}

\def\ex#1{\text{\rm E}_{\mathfrak X}}

\def\Dvkom{D^{\vk}_{< \omega}}

\def\cP{{\mathcal P}}
\def\cE{{\mathcal E}}
\def\cV{{\mathcal V}}
\def\cD{{\mathcal D}}
\def\cPV{{\mathcal P}{\mathcal V}}
\def\cVD{{\mathcal V}{\mathcal D}}
\def\cEPV{{\mathcal E}{\mathcal P}{\mathcal V}}
\def\cPG{{\mathcal P}{\mathcal G}}

\def\cA{{\mathcal A}}
\def\cB{{\mathcal B}}

\def\cM{{\mathcal M}}
\def\cN{{\mathcal N}}

\def\fX{{\mathfrak X}}

\def\fL{{\mathfrak L}}

\renewcommand{\le}{\leqslant}
\renewcommand{\ge}{\geqslant}
\renewcommand{\wedge}{\, \& \,}

\begin{document}

\tableofcontents

\title[Classical groups: elementary equivalence]{Elementary equivalence \\
of infinite-dimensional classical groups}
\thanks{Supported in part by the Russian Foundation of Fundamental Research
Grant 96-01-00456}
\author{Vladimir Tolstykh}
\address{Kemerovo State University, Department of Mathematics,
Krasnaja, 6, 650043, Kemerovo, Russia}
\email{vlad\symbol{64}corvette.kuzb-fin.ru}

\begin{abstract}
{\rm Let $D$ be a division ring such that the number
of conjugacy classes in the multiplicative
group $D^*$ is equal to the power of $D^*.$
Suppose that $H(V)$ is the group $\gl V$ or $\pgl V,$ where
$V$ is a vector space of infinite dimension $\vk$
over $D.$ } We prove, in particular,
that, uniformly in $\vk$ and $D,$ the first order
theory of $H(V)$ is mutually syntactically interpretable
with the theory of the two-sorted structure $\str{\vk,D}$
(whose only relations are the division ring operations
on $D$) in the second order logic with quantification
over arbitrary relations of power $\le \vk.$ A certain
analogue of this results is proved for the groups
$\Gl V$ and $\pGl V.$ These results imply criteria
of elementary equivalence for infinite-dimensional
classical groups of types $H=\Gamma$L, P$\Gamma$L,
GL, PGL over division rings, and solve, for these
groups, a problem posed by Felgner. It follows
from the criteria that if $H(V_1) \equiv H(V_2)$ then
$\vk_1$ and $\vk_2$ are second order equivalent
as sets.
\end{abstract}

\maketitle


In the present paper we deal with the problem
to what extent the first order theory of an infinite-dimensional
classical group over a division ring determines
the dimension of the group and the ring.

In the case of finite dimension, for many types of classical groups,
the problem can be easily reduced to the problem when two classical groups
of the same type are isomorphic. Indeed, by the Keisler--Shelah theorem,
structures $\mathcal M$ and $\mathcal N$ are elementarily equivalent iff,
for some ultrafilter $F,$ the ultrapowers
${\mathcal M}^F$ and ${\mathcal N}^F$ are isomorphic.
The following Isomorphism Theorem is known
\cite{HO'M}.
For
$H= \text{GL, SL, PGL, PSL}$
and any division rings $D_1, D_2,$
if $n_1,n_2\ge 3$ then
$$H(n_1,D_1)\simeq H(n_2,D_2)\text{\ if and only if\ }
n_1=n_2, \text{\ and\ }
D_1\simeq D_2
\text{\ or\ }
D_1\simeq D_2^{\rm op}.$$
For
$H= \rm{GL},\,\rm{SL},\,\rm{PGL}$,
the same holds even for $n_1,n_2\ge 2.$
(For
$H=\rm{PSL}$, in the case of dimension 2 there are some
exceptional isomorphisms.)
Taking into account
$H(n,D)^F\simeq H(n,D^F),$ we have that,
for $H, n_1, n_2$ satisfying the conditions of
the Isomorphism Theorem,
$$H(n_1,D_1)\equiv H(n_2,D_2)\text{\ if and only if\ }
n_1=n_2,\text{\ and\ }
D_1\equiv D_2
\text{\ or\ }
D_1\equiv D_2^{\rm op}.$$
Maltsev \cite{Malt} proved the latter result in the special case
of groups over fields of characteristic 0; his proof was based on an
interpretation of the field $D$ in the group $H(n,D).$

In \cite{DS} Felgner suggested to study the problem
of elementary equivalence for infinite-dimensional
general linear groups and other classical groups over
fields.  In the present paper we solve Felgner's
problem for infinite-dimensional groups of types GL,
PGL, $\Gamma$L, P$\Gamma$L for a wide class of
division rings.

In a more general setting, the subject of the paper can be
described as a study of the expressive power of the first order logic
for infinite-dimensional classical groups and related
structures. The similar problem was considered in many
papers, in particular, in the papers \cite{Sh1,Sh2} by
Shelah on infinite symmetric groups, in his paper
\cite{Sh3} devoted to endomorphism semi-groups of free
algebras, in a series of papers on automorphism groups
of Boolean algebras by Rubin and Shelah (e.g.
\cite{RSh}), in the paper \cite{MRRS} by Magidor,
Rosental, Rubin and Srour on lattices of
closed subsets of Steinitz exchange systems.

According to \cite{Sh3,BSh}, one can measure the expressive power
of a first order theory by the richness of the
fragment of set theory interpretable in it.
In  \cite{Sh1,Sh3,MRRS} this idea has been realized
in the following way. With every
structure $\cM$ from a given class of structures
a structure $\cM^*$ is associated, so that the elementary equivalence of
structures $\cM$ and $\cN$  from the class
implies the $\mathcal L$-equivalence of
$\cM^*$ and $\cN^*$ in a certain logic $\mathcal L.$
The structures of the form $\cM^*$ are chosen to be `algebra-free'
as much as possible, and the logic $\mathcal L$ is chosen to be as
`strong' as possible. A nice illustration
of this method is given by a following version
of Theorems 1.6 and 3.1 from \cite{MRRS}:
if ${\mathcal K}=\str{K,+,\cdot}$ is an uncountable algebraically closed field
then the full second order theory of the set $K$
is syntactically interpretable (uniformly in $\mathcal K$) in the first order
theory of the lattice $L({\mathcal K})$ of algebraically closed subfields
of $\mathcal K.$ Hence
for any uncountable algebraically closed fields ${\mathcal K}_1$ and
${\mathcal K}_2,$ the elementary equivalence of
the lattices $L({\mathcal K_1})$ and
$L({\mathcal K_2})$ implies the equivalence of the sets $K_1$ and $K_2$
in the full second order logic.

The paper is also concerned with the question
when set theory is interpretable in the automorphism
groups of algebras which are free in a variety $\mbox{\bf V}$ and
of infinite rank. This question, as Shelah notes
in his paper \cite{Sh3}, is natural in view of the following result obtained
in that paper: set theory is interpretable in the endomorphism
semi-group of a free $\mbox{\bf V}$-algebra which is of `large' infinite
rank. The answer in the automorphism group case
essentially depends on the variety $\mbox{\bf V}$ (for
example, one cannot interpret set theory in any
infinite symmetric group -- the automorphism
group of an algebra in empty language \cite{Sh1,Sh2}).

Let $V$ be a vector space of infinite
dimension $\vk$ over a division ring $D.$ The present paper
can be divided into three parts.  The aim of the first
part (Sections~1--5) is to interpret the
projective space of $V$ (that is,
${\mathcal P}=\str{P(V),\subseteq},$
the lattice of subspaces of $V$)
in the group $\pgl V$
(Theorem \ref{PrjSpaceInPrjGroup}). The
assumption $\vk\ge\aleph_0$
is essential for the proof. As one of the key points of the proof we
show the $\varnothing$-definability of the set of involutions of the first
kind in $\pgl V$
(that is, involutions induced by involutions in $\gl V$).
This solves the problem of group-theoretic characterization
of involutions of the first kind in $\pgl V$ posed by Rickart
\cite{Ri3} and enables us to describe
isomorphisms of infinite-dimensional groups
of types $\Gamma$L, P$\Gamma$L, GL, and PGL
by classical methods. This description modulo the mentioned
problem was known since the early fifties, but has been justified
only in 1977 by O'Meara \cite{O'Mea}
who used non-classical techniques.

In the second part of the paper
(Sections~6--8) we show that, uniformly in $\vk$ and $D,$
$$
\TH(\Gl V) \ge \TH(\pGl V) \ge \TH(\pgl V) \ge \TH(\gl V)
$$
(Theorem \ref{Gl>pGl>pgl>gl}).  Here $\ge$ means
`syntactically interprets' (for the definition, see
Section \ref{0.12}).  Moreover, we prove that
$\TH(\pGl V) \ge \TH(\Gl V)$ (Theorem \ref{PH<>H}).
Thus, the logical power does not drop under the
transition to the projective image. Since the group
$\pgl V$ is obviously interpretable in the group $\gl
V,$ we can reconstruct the projective space in all the
groups $\Gl V,$ $\pGl V,$ $\pgl V,$ and $\gl V.$

Let $\lambda$ be an infinite cardinal. We denote (in the
manner of Shelah \cite{Sh1}) by $\LII(\lambda)$ the second order
logic, which allows to quantify over arbitrary
relations of power $<\lambda.$ The monadic
fragment of this logic, $\Mon(\lambda),$ allows
to quantify over arbitrary subsets of power $<\lambda.$ We
denote by $\str{\vk,D}$ the two-sorted structure, whose first sort
is the cardinal $\vk,$ the second one is the division
ring $D,$ and the only relations of this structure are the
standard ring operations on $D.$ Consider also the
two-sorted structure $\str{V,D},$ by combining the abelian
group of $V$ and the division ring $D$ with their basic
relations and the ternary relation for the
 action of $D$ on $V.$

{\rm Let $D$ be division ring such that

\qquad \parbox{10cm}{the number of conjugacy classes
of the multiplicative group
$D^*$ is equal to the power of
$D^*.$} \hfill $(*)$
}

In the third part of the paper (Sections~
\ref{3.1}--\ref{3.4}) we demonstrate that various
theories associated with the vector space $V$ are
pairwise mutually syntactically interpretable, uniformly in
$\vk$ and $D$ (Theorem \ref{MainTheorem}).
In particular, we prove this for the first order theories of
\begin{itemize}
\item the projective space $\cP,$
\item $\End V,$ the endomorphism semi-group of $V,$
\item the groups $\pgl V$ and $\gl V,$
\end{itemize}
and the second order theories
\begin{itemize}
\item $\TH(\str{V,D},\Mon(\vk^+)),$
\item $\TH(\str{\vk,D},\LII(\vk^+)).$
\end{itemize}
{\rm (Note that the mentioned first order
theories are mutually interpretable
for {\it arbitrary} division rings.)}
As a consequence, $\TH(\gl V) \ge \TH_2(\vk),$ or, in
other words, the automorphism groups of
infinite-dimensional vector spaces interpret set
theory. This provides a solution to the question from
\cite{Sh3} mentioned above for any variety of vector
spaces over a fixed division ring {\rm with $(*)$.}

An early version of Theorem \ref{MainTheorem}, without
the elementary theories of classical groups in the
list of mutually interpretable $V$-theories and under
the additional assumption of commutativity of $D,$ has
been proved in the joint paper of the author and
Belegradek \cite{BT}.

Theorem \ref{MainTheorem} gives a solution to  Felgner's
problem for infinite-dimensional linear groups
of types GL and PGL:
\begin{align} \tag{$**$}
&\gl{\vk_1,D_1} \equiv \gl{\vk_2,D_2} \Leftrightarrow\nonumber\\
&\pgl{\vk_1,D_1} \equiv \pgl{\vk_2,D_2}\Leftrightarrow \nonumber\\
&\TH(\str{\vk_1,D_1},\LII(\vk_1^+))=\TH(\str{\vk_2,D_2},\LII(\vk_2^+)),
\nonumber
\end{align}
where $\vk_1,\vk_2$ are infinite cardinals,
$D_1$ and $D_2$ are arbitrary division rings {\rm with $(*)$}.
Theorem \ref{MainTheorem} provides also more accurate
estimate of the logical power of the elementary theory
of projective space $\TH(\cP)$: according to
\cite[Theorem 1.7]{MRRS} if $D$ is commutative, then the theory
$\TH(\cP)$ has the logical power at least that of
second order logic on the cardinal $\min(\vk,|D|).$

To estimate the logical strength of the elementary
theories $\TH(\Gl V)$ and $\TH(\pGl V)$ we need a
stronger logic than $\LII(\vk^+)$ is.  This logic
is $\cL_D(\vk^+),$ extending expressive power of
$\LII(\vk^+)$ by a possibility to quantify over
arbitrary automorphisms of the division ring $D.$
Theorem \ref{MnThForGls} states that the theories $\TH(\Gl V),$
$\TH(\pGl V),$ and $\TH(\str{\vk,D},\cL_D(\vk^+))$
are pairwise mutually syntactically interpretable, uniformly
in $\vk$ and $D.$ This enables us to give a criterion of
elementary equivalence of infinite-dimensional
semi-linear groups similar to $(**)$.

We do not consider in this paper the
classification of elementary types for the class of
infinite-dimensional linear groups of types E and
E${}_\fX,$ which are natural infinite-dimensional analogues of
finite-dimensional groups of the type SL over fields
(see \cite[1.2, 2.1]{HO'M} for details).
We prove in \cite{To2} that all infinite-dimensional
groups of the types E and E${}_\fX$ over a fixed
division ring $D$ are elementary equivalent.

In Section \ref{3.4} we examine the condition
\begin{equation}
\TH(\str{\vk_1,D_1},\LII(\vk_1^+))=\TH(\str{\vk_2,D_2},\LII(\vk_2^+)).
\end{equation}
In particular, this makes possible to prove that
$\gl{\aleph_0,\R} \equiv \gl{\vk,D}$ iff
$\vk=\aleph_0$ and $D \cong \R,$ and $\gl{\aleph_0,\C}
 \equiv \gl{\vk,D}$ iff $\vk=\aleph_0$ and $D$ is an
uncountable algebraically closed field of
characteristic zero.  Furthermore,
using results from \cite{MRRS}, we prove
that $\Gl{\aleph_0,\C} \equiv \Gl{\vk,D}$ iff
$\vk=\aleph_0$ and $D \cong \C.$ This demonstrates
that the condition (\theequation) does not suffice
for the elementary equivalence of semi-linear groups.

In Section 0 we recall a number of basic facts of
linear group theory and a small portion of
mathematical logic.

\setcounter{section}{-1}

\section{Basic concepts and notation} \label{0.12}

Let  $V$  be always (throughout all the text) a {\it left
infinite-dimensional vector space over a division ring
$D.$} The dimension of $V$ will be denoted by $\vk.$  We
denote the elements of $V$  by lower case Latin letters
$a,b,c,\ldots,$ and the elements of $D$  by lower case Greek
letters  $\l, \mu,\nu.$ We shall use the letter
$W$  as the notation of an {\it arbitrary} left vector space
over $D.$

The {\it projective space} $P(W)$ is treated as the
set of all subspaces of $W$ \cite{Ba,Art,O'Mea0};
$P^*(W)$ will denote the set of all proper non-zero
subspaces of $W.$ $P^n(W)$ is the standard notation
for the set of all $n$-dimensional subspaces of  $W$
\cite{O'Mea0}.  We denote by $P^{(n)}(W)$ the set of
all subspaces in $W$ of dimension or codimension $n.$
The subspaces of $W$ of dimension
one and codimension one will be called {\it lines} and {\it hyperplanes},
respectively; the term `line' will be never used in the
present paper in the sense of projective geometry.
The letter $N$ usually denotes a line of $W$
and the letter $M$ a hyperplane of $W$ (possibly with indices, primes, etc.).

Let  $f$ be an isomorphism between division rings $D_1$ and $D_2,$
and $W_1,$ $W_2$ be vector spaces over $D_1$ and
$D_2,$ respectively. Recall that a transformation
$\sigma$ from $W_1$ to $W_2$ is a {\it semi-linear
transformation with respect to the {\rm(}associated{\rm)}
isomorphism $f,$} if
\begin{alignat*}2
&\sigma(a+b) = \sigma(a)+\sigma(b),    &&a,b \in W, \\
&\sigma(\lambda a) = f(\lambda)\cdot \sigma(a), &\quad &a \in
W, \lambda  \in D_1.
\end{alignat*}
Since a semi-linear transformation determines uniquely
its associated isomorphism, the action of the
associated isomorphism is usually written in the form
$\lambda^\sigma.$

The group of all bijective semi-linear transformations
from $W$ into itself (col\-li\-ne\-a\-tions) is called the
{\it semi-linear} ({\it collinear}) {\it group} of the
space $W$ \cite{O'Mea}. The standard notation
is $\Gl W.$ The subgroup of all linear transformations
from $\Gl W$ is the {\it general linear group} of the
space  $W;$ it is written as $\gl W.$

Every collineation $\sigma \in \Gl W$ induces in a
natural way a permutation  $\hat{\sigma}$ of the set
$P(W).$ The transformation  $\hat{\sigma}$ is said to
be the {\it projective image} of $\sigma.$ The set of
all projective images of the elements of the group
$\Gl W$ with the composition law is the {\it
projective semi-linear group} of $W.$ Notation:
$\pGl W.$ Clearly, the mapping $\,\hat{}\,$ is a
homomorphism from the group  $\Gl W$ to the group
$\pGl W.$  The {\it projective general linear group}
of the space $W$ is a subgroup of $\pGl W,$ consisting
of the projective images of linear transformations; it
is written as $\pgl W.$

A transformation $\tau : W \to W$ such that for
some $\lambda_{\tau} \in D^*$
$$
\tau a = \l_{\tau} \cdot a,\quad \forall a \in  W,
$$
is called a {\it radiation.} We shall denote $\tau$  by
$\l\cdot \id(W).$ The set of all radiations of
$W$  with the composition law is obviously the group
isomorphic  to $D^*,$ multiplicative group of $D.$
Notation: $\rl W.$

\begin{Prop} \label{0.1.1}
{\rm (\cite[Chapter III, Section 3]{Ba}).} Let $\dim
W \ge  2,$ and $\sigma_1, \sigma_2 \in  \Gl W.$ Then
$\hat\sigma_1=\hat\sigma_2$ if and only if
$\sigma_1\sigma^{-1}_2\in  \rl W.$
\end{Prop}

The group $\rl W$ is clearly a normal subgroup of $\Gl W.$ It
is easy to see that a radiation  $\tau$ lies in $\gl W$
iff $\l_{\tau} \in  Z(D),$ where  $Z(D)$  is the
center of $D.$ Furthermore, $Z(\gl W) = Z(\rl W) =
\gl W \cap \rl W.$ Thus,

\begin{Cor} \label{PrjsAsFactorByRads}
Let  $\dim W \ge  2.$ Then

{\rm (a)} The group $\pGl W$ is isomorphic
to the quotient group $\Gl W/\rl W.$

{\rm (b)} The group  $\pgl W$ is isomorphic to the
quotient group $\gl W/Z(\rl W).$
\end{Cor}

We describe now the involutions in the group  $\gl W.$
Let $\sigma$ be an arbitrary involution of $\gl W.$
There are two different cases:  $\chr D \ne 2$  and
$\chr D =2.$

{\it I. The characteristic of $D$  is not}  2. In the
case we have a decomposition
\begin{equation}
W = W^-_\s\oplus  W^+_\s,
\end{equation}
where $W^+_\s= \{a\in W: \s a=a\}$  and  $W^-_\s=\{a\in W: \s a=-a\}.$
The subspaces  $W^-_\s$ and $W^+_\s$ are called the
{\it subspaces} of $\s.$ The decomposition
(\theequation) implies that there is a basis of $W$ in
which $\s$  is diagonalized. Furthermore, one can
easily prove the following

\begin{Lem} \label{0.1.3}
Let  $\s_1,\ldots,\s_n$ be pairwise
commuting involutions in $\gl W.$ Then there is a
basis  of  $W$  in which all  $\s_1,\ldots
,\s_n$ are diagonalized.
\end{Lem}

An involution $\s \in \gl W$ is called {\it extremal}
if some its subspace is a line (or, equivalently, a
hyperplane).

{\it II. The characteristic of  $D$  is equal to} 2.  In this case
we can also assign  to  an involution   $\s \in \gl W$
two subspaces of $W.$ These subspaces are $\Fix(\s) =
\{a\in W~:~\s a=a\}$  and $\Rng(\id(W)+\s),$ where
$\Rng(\pi)$ is the image of a transformation  $\pi.$

Choose a linearly independent set $\{d_i\,:\,i\in I\}$ such
that
$$
W = \Fix(\s) \oplus  \str{d_i\,:\, i\in I}.
$$
Let  $e_i=d_i+\s d_i, i\in I.$ The set $\{e_i: i\in I\}$
is obviously a linearly independent  subset of
$\Fix(\s).$

Let  $\{e_j: j\in J\}$  be a complement of
$\{e_i:i\in I\}$  to a basis of  $\Fix(\s).$ Thus, $\s$
acts on the basis  $\{e_i: i\in I\} \cup  \{e_j:\in J\}
\cup  \{d_i: i\in I\}$ of $W$ as follows
\begin{alignat}2 \label{eq0.1.1}
&\s e_i = e_i,    && i \in I, \\
&\s e_j = e_j,    && j \in J,  \nonumber \\
&\s d_i = d_i+e_i, &\quad & i \in I. \nonumber
\end{alignat}
On the other hand, any $\s \in \gl W,$ which
acts on some basis of $W$ similar to
(\ref{eq0.1.1}), is an involution.

The definition of the extremal involutions remains the
same: an involution of the group $\gl W$ is called
{\it extremal,} if it has a subspace of dimension one. It
is important that in the case when $\chr D =2,$ the
extremal involutions are also {\it transvections,}
that is, linear transformations of the form
$$
\s a = a +\delta(a)b,
$$
where $\delta$  is a non-zero linear function from
$W$  to $D$ such that $\delta (b)=0.$ Clearly,
$\Fix(\sigma) = \ker(\delta)$ and $\Rng(\id(W)+\sigma)
= \str b.$ The line  $\str b$ is called the {\it line}
of  $\s$ and the hyperplane $\ker(\delta)$ is called
the {\it hyperplane} of  $\sigma$ ({\it subspaces} of
$\sigma$).

\begin{Lem} \label{TransvsBasics}
\mbox{\rm (a)} Let  $M$ and  $N$ be a hyperplane and a
line with $N \subseteq M.$ There is a transvection
$\sigma$  in $\gl W$  such that the subspaces of  $\sigma$
are identical to  $M$  and  $N.$

{\rm (b)} Let  $\sigma_1$ and $\sigma_2$  be two transvections in
$\gl W,$ and let $N_k,M_k$ be the subspaces  of  $\s_k,$
where $k=1,2.$ Then $\sigma_1\sigma_2= \sigma_2\sigma_1$
if and only if $(N_1 \subseteq  M_2 \wedge M_1\supseteq N_2).$

{\rm (c)} Let  $\sigma_1$   and   $\sigma_2$   be two distinct
transvections  in $\gl W.$ Then  $\sigma_1$  and $\sigma_2$
have a mutual subspace if and only if $\sigma_1\sigma_2$ is
a transvection.
\end{Lem}

Lemma \ref{TransvsBasics} is a well-known result,
which can be found, for example, in \cite[pp.
101-102]{O'Mea}, where it is formulated for arbitrary
vector spaces over division rings (not only for
finite-dimensional ones as in most of the works in linear
group theory, but also for infinite-dimensional
vector spaces).

{\bf Remarks.} (a) It should be pointed out that both
methods of assigning subspaces to an involution
$\sigma$ (whether characteristic is equal to 2, or
not) could be treated in a uniform way, if we assign,
in the style of O'Meara, to an involution of $\gl
W$ its {\it fixed} and {\it residual} subspaces, where
the residual one is the subspace $\Rng(\id(W)-\sigma).$

(b) Note also that in both cases we assign to each
involution in $\gl W$ an {\it unordered} pair of
subspaces of $W.$

Let us describe the involutions in the group $\pgl
W$ of dimension at least two. An involution  $\hat\s \in \pgl W$
is said to be an {\it involution of the first kind} in the group
$\pgl W,$ if  $\hat\s$  is induced by an involution
of  $\gl W$ \cite[p. 8]{Die1}. The involutions that
are not of the first kind are called
{\it involutions of the second kind.} We denote the
identity element of  $\pgl W$  simply by  $1.$ By
Proposition \ref{0.1.1} $\hat\s^2=1$  iff  $\s^2=
\l\cdot \id(W),$ where  $\l \in  Z(D).$ It is easy to
see that if  $\s^2= \l\cdot \id(W),$ then  $\hat\s$
is an involution of the first kind iff $\l$  is a
square in  $Z(D).$

The $\pgl W$-involutions induced by extremal $\gl
W$-involutions are called {\it extremal}, too.  Let
$\hat\sigma$ be  an involution of the first kind
induced by a $\gl W$-involution  $\sigma.$ The {\it
subspaces} of $\hat \s$ are surely chosen (well-defined)
to be identical to the subspaces of $\s.$
Let $\s$  be an involution of  $\gl W$ or a $\pgl V$-involution of the
first kind. We call  $\s$ a $\gamma$-{\it involution,}
if  $\gamma$  is  equal to $\min(\dim R, \dim S),$
where  $R$  and  $S$  are the subspaces of $\s.$

We discuss now the important notion of a {\it minimal
pair.} The notion appeared in the paper of Mackey
\cite{Ma}. That paper has a section devoted to the
groups of autohomeomorphisms of infinite-dimensional
normed linear spaces over the field of reals. In
\cite{Ri1,Ri2,Ri3} Rickart extended the methods of
Mackey from the groups of autohomeomorphisms to some
classical groups.

Let  $\dim W \ge  3.$ Minimal pairs interpret the
elements of $P^{(1)}(W)$  as follows. An extremal
involution determines two subspaces: a line  $N$  and
a hyperplane $M.$ Then, if we have a pair  $\str{\s_1,\s_2}$
such that
\begin{itemize}
\item[(1)] $\s_1,\s_2$ are extremal $\gl W$-involutions;
\item[(2)] $(N_1=N_2 \wedge M_1\neq M_2)$  or
$(N_1\neq N_2 \wedge M_1=M_2),$
\end{itemize}
then involutions $\s_1$ and $\s_2$ have the unique
mutual subspace -- namely, the subspace which
is in both the pairs of subspaces associated
with $\s_1$ and $\s_2,$ and therefore the tuple $\str{\s_1,\s_2}$
codes some subspace of $W.$ In
the case, when $\chr D =2$  it is technically convenient to
add the condition
\begin{itemize}
\item[(0)] $\s_1\s_2= \s_2\s_1.$
\end{itemize}
So a pair of  $\str{\s_1,\s_2}$ of
$\gl W$-involutions satisfying (1,2) in the case $\chr
D \neq 2$  or the conditions  (0,1,2)  in the case
$\chr D =2$  is called a {\it minimal pair} of the group
$\gl W.$

Rickart in the above mentioned  above papers
\cite{Ri1,Ri2,Ri3} modifying the methods of Mackey
showed that  if $\chr D \ne 2$  then the property of
being a $\gl W$-minimal  pair is group-theoretic.
Rickart denoted by $c(I)$ the set of all
involutions in the centralizer of a subset $I$ of the
group $\gl W.$

\begin{Th} \mbox{ (\cite[Theorem 2.6]{Ri1}). }  \label{Mckey&Rckrt}
Let  $\chr D \ne 2$  and let $\dim W \ge 3.$ Then the
following properties are equivalent

{\rm (a)} $\str{\sigma_1, \sigma_2}$  is a minimal pair in
$\gl W;$

{\rm (b)} $c(c(\sigma_1, \sigma_2)) = c(c(\pi_1,\pi_2)),$
where $\pi_1$ and $\pi_2$ are arbitrary non-commuting
extremal involutions in $c(c(\sigma_1,\sigma_2)).$
\end{Th}

Thus, modulo definability of the extremal involutions,
the property of being a minimal pair is even first
order.

In the case when $\chr D =2$ we may obtain the
first order characterization of  $\gl W$-minimal pairs
(modulo definability of extremal involutions) by
applying Lemma \ref{TransvsBasics}.

\begin{Prop} \label{MPsInChar2}
Let  $\chr D =2,$   $\dim W \ge  3,$ and  $\s_1,\s_2$ be commuting
transvections in $\gl W.$ Then the following
conditions are equivalent

{\rm (a)} $\str{\s_1,\s_2}$  is a minimal pair in
$\gl W;$

{\rm (b)} $\s_1\s_2$  is a transvection  and
there is a transvection  $\s,$ commuting with $\s_1,$
but not with $\s_2.$
\end{Prop}

\begin{proof}
$\Rightarrow$. Suppose that  $\s_k$ has the subspaces
$N_k,M_k,$ where $k=1,2.$ We assume that $(N_1 \neq N_2 \wedge
M_1=M_2).$ Then we may take as $\s$ a
transvection with the line $N_1$  and the
hyperplane  $M',$ where $M'$ is a hyperplane which does not contain
$N_2.$ In the dual case $(N_1=N_2 \wedge M_1\neq M_2)$ one may
choose a line  $N',$  the linear span of an element
from  $M_1 \setminus M_2,$ and then construct  $\s$  by
the subspaces  $N'$  and  $M_1.$

$\Leftarrow$. If $\s_1$  and  $\s_2$ commute, but the
pair  $\str{\s_1,\s_2}$ is not minimal, then
$(N_1=N_2 \wedge  M_1=M_2)$ and the second condition
in \ref{MPsInChar2}(b) is false by
\ref{TransvsBasics}(b), since any transvection,
commuting with $\s_1,$ must commute with $\s_2.$
\end{proof}

A pair  $\str{\hat{\s}_1,\hat{\s}_2} \in \pgl W$ is
called {\it minimal,} if there are involutions $\sigma_1$  and
$\sigma_2$  in the preimages of $\hat{\sigma}_1$  and
$\hat{\sigma}_2,$ respectively, such that $\str{\s_1,\s_2}$
is a $\gl W$-minimal pair.

Using   minimal    pairs  Rickart, Dieudonn\'e
\cite{Ri1,Ri2,Ri3,Die1} and other authors described the
groups  of automorphisms for various types of
classical groups.

\medskip

{\bf Remark.} Note that certain automorphisms of {\it
finite-dimensional} group $\gl W$ (or $\pgl W$) send
minimal pairs with a mutual {\it line} to minimal pairs
with a mutual {\it hyperplane}; this is impossible in
the infinite-dimensional case (see \cite{Die2,HO'M,O'Mea} for
details). Thus, roughly speaking, there is no hope
to distinguish lines and hyperplanes coded by minimal
pairs, and, moreover, the subspaces of dimension
$k$ and the subspaces of codimension $k$
in the finite-dimensional case.
On the other hand, we shall interpret in
the infinite-dimensional linear group
$\pgl V$ the elements of the projective
space $P(V),$ and then interpret the inclusion
relation on $P(V).$ This will demonstrate
that the set of all minimal pairs with a  mutual
line is $\varnothing$-definable in the infinite-dimensional
group $\pgl V$ in contrast to the finite-dimensional case.

\medskip

We close this section with a portion of logic.

We shall denote by $\TH(\cM,\cL)$ the theory of a
structure $\cM$ in a logic $\cL.$

Let $\{T^0_i : i \in {\bold I}\}$ and $\{T^1_i : i \in
{\bold I}\}$ be two families of theories in logics
$\cL_0$ and $\cL_1,$ respectively.  We say that the
theory $T^0_i$ is {\it syntactically interpretable} in
$T^1_i$ {\it uniformly in} $i \in {\bold I},$ in
symbols $T^0_i \le T^1_i,$ if there is a mapping
$\,{}^*\,$ from the set of all $\cL_0$-sentences to
the set of $\cL_1$-sentences, such that, for every
$\cL_0$-sentence $\chi$ and for every $i \in {\bold
I},$ $T^{0}_i \vdash \chi$ iff $T^1_i
\vdash \chi^*$ \cite[Chapter V]{BSh,Ho}.  If, in addition,
$T^1_i \le T^0_i$ uniformly in $i \in {\bold I},$ the
theories $T^0_i, T^1_i$ are said to be {\it
mutually syntactically interpretable uniformly in $i \in
{\bold I}.$} The relation  $\le$ is clearly reflexive
and transitive.

Quite informally, in cases when the class of indices
is clear from the context, we shall often write
interpretability results in the form $T^0 \le T^1.$

We state now two sufficient conditions
for uniform syntactical interpretation.

A structure  $\cM$  is said  to  be {\it
$\varnothing$-interpretable/reconstructible
{\rm(}without parameters{\rm)} in a structure
$\cN$  by means of  a  logic  $\mathcal L,$}  if
there  are  a positive integer  $n,$ a
$\varnothing$-definable by means of  $\mathcal L$  set
$X$ of $n$-tuples in  $\cN,$  and a surjective
mapping  $f: X \to \cM$ such that  $f$-preimages
of all the basic relations on $\cM$ (including
the equality relation) are $\varnothing$-definable by
means of $\mathcal L$ in $\cN.$ From this
definition, one can easily realize, what means
`an interpretation with parameters' or
`an interpretation uniform in ...'.

A usual way to  construct a syntactical
interpretation is the following well-known
result (see e.g. \cite[Chapter V]{Ho}).

\begin{Th}[Reduction Theorem] \label{0.2.1}
If, uniformly in $i \in {\bold I},$ a structure $\cM_i$ is interpretable
without parameters in the structure $\cN_i$ by means
of logic $\cL,$ then, uniformly in $i,$ the theory
$\TH(\cM_i,\cL)$ is syntactically interpretable in
the theory $\TH(\cN_i,\cL).$
\end{Th}

In some cases, to provide the conclusion of the Theorem above,
one can use interpretations {\it with} parameters.
Suppose, there is a $\cL$-formula $\chi(\av x)$ such that
for each tuple  $\av a_i$ in the domain of the
structure $\cN_i,$ satisfying  $\chi,$ the structure $\cM_i$
is reconstructible (uniformly) in the structure
$\str{\cN_i;\av a_i}$ by means of $\cL.$ Therefore the theory
$\TH(\cM_i,\cL)$  is uniformly syntactically interpretable
in the theory $\TH(\str{\cN_i;\av a_i},\cL).$
Let  $\theta \mapsto  \theta^*(\av a)$ be the corresponding
mapping of $\cL$-sentences. Then the mapping
$$
\theta  \mapsto  (\forall \av x)(\chi(\av x)  \rightarrow
\theta^*(\av x))
$$
provides a uniform syntactical interpretation
$\TH(\cM_i,\cL)$ in $\TH(\cN_i,\cL)$ \cite[Chapter V]{Ho}.
We shall often use such a trick below.

Most of our structures will be multi-sorted; they
can be treated as ordinary ones, in a usual way.

\section{Relation  {\it Cov} }\label{1.1}

In his paper \cite{Die1} Dieudonn\'e used at times the
binary relation `$y$ is the product of two commuting
conjugates of $x$':
\begin{equation} \label{eq1.1.1}
(\exists z_1 z_2)(y = x^{z_1}x^{z_2}=x^{z_2}x^{z_1}),
\end{equation}
where  $x^z=zxz^{-1}.$

In  \cite{To1}  the  author applied the relation given
by formula (\ref{eq1.1.1})  in order to prove the
following result:  if $\gl{\aleph_{\alpha};D_1} \equiv
\gl{\aleph_{\beta};D_2},$ where $D_1$ and $D_2$  are
division rings of characteristic $\neq  2,$ then
$\str{\alpha;<} \equiv  \str{\beta;<}.$ Note that
McKenzie \cite{McKen} proved the similar
result for the class of symmetric groups
$S_{\alpha}=\text{Sym}(\aleph_{\alpha}).$

We say that  `{\it $\sigma$  covers  $\pi$}',  if the
pair  $\str{\sigma,\pi}$ satisfies the formula (\ref{eq1.1.1}). So
we denote this formula by $\Cov(x,y).$

Recall our convention from Section \ref{0.12}: $V$ is
the notation for an arbitrary infinite-dimensional
vector space over a division ring. The letters $D$ and
$\vk$ {\it are always associated} with $V,$ and denote the
underlying division ring of $V$ and the dimension of
$V,$ respectively.

We shall apply below the relation $\Cov$  to solve the
problem of group-theoretic (first order) characterization  of
involutions of the first kind in the group $\pgl V;$ the latter
ones will be used later for a reconstruction of the projective
space $\str{P(V);\subseteq}$ in the group $\pgl V.$
In this section we describe the behaviour of the
relation  $\Cov$ on the set of all $\pgl V$-involutions
of the first kind.

Again, according to Section \ref{0.12}, a
$\gamma$-involution of the group $\pgl V$  is the
projective image of a $\gamma$-involution of the
group $\gl V.$ The conditions of being  a
$\gamma$-involution of the group $\pgl V$ for some
$\gamma$ and of being a  $\pgl V$-involution of the
first kind are clearly equivalent.

We shall denote elements of the
group  $\pgl V$  by lower case Greek letters
$\sigma,\pi,\ldots$ and elements of the group  $\gl V$
by lower case Latin letters $s,p,\ldots$  so that
$\sigma~=~\hat s,$ $\pi~=~\hat p,$ $\rho~=~\hat r$  and so
on.  It  is convenient  to  agree that in the
situation, when $\sigma= \hat s$  and  $\sigma$  is a
$\pgl V$-involution of the first kind $s$  is also an
{\it involution} (in $\gl V).$ Following this
agreement, we state that

\begin{Lem} \label{PMLem}
If $\chr D \ne 2,$ and $\s,\pi$ are involutions
in the group $\pgl V,$ then $\s\pi=\pi\s$ if
and only if $sp=\pm ps.$
\end{Lem}

\begin{proof}
Indeed, by \ref{0.1.1} $\sigma\pi= \pi\sigma$
implies  $sp = \mu ps,$ where  $\mu \in  Z(D).$ Since
$p$  induces an involution,  $p^2= \lambda \cdot \id(V),$
where  $\lambda  \in  Z(D).$ We have
$$
(sps^{-1})^2 = (\mu p)^2 \Rightarrow
\lambda \cdot \id(V) = \mu^2 \lambda \cdot \id(V).
$$
So $\mu^2=1$ that is $\mu = \pm 1.$
\end{proof}

\begin{Lem} \label{Cov(s,pi)}
Suppose that $\chr D\ne 2,$ and $\gamma, \gamma'$
are cardinals $\le \vk=\dim V.$ Then

{\rm (a)} a $\g$-involution in the group $\pgl V$ with
infinite $\gamma$ covers any $\gamma'$-involution if and
only if $\gamma' \le \gamma;$

{\rm (b)} a $\gamma$-involution in $\pgl V$ with
finite $\gamma$ covers any $\gamma'$-involution if and only
if  $\gamma'$ is finite even cardinal and  $\gamma'  \le  2\gamma.$
\end{Lem}

\begin{proof} (a)  Assume that  a $\gamma$-involution
$\sigma$ covers some $\gamma'$-involution  $\pi.$ This means
that $\pi= \sigma_1\sigma_2= \sigma_2\sigma_1,$ where $\sigma
_1,\sigma_2$  are conjugate to  $\sigma.$ Since
$\sigma_1\sigma_2= \sigma_2\sigma_1,$ we
have by \ref{PMLem} $s_1s_2=\pm s_2s_1,$ where $\sigma_k=
\hat s_k,$  $k=1,2.$  First suppose  that  $s_1$
and $s_2$ commute. We have  $p = \mu s_1s_2,$
where  $\pi= \hat p, \mu  \in  Z(D).$
Since $p$ is an involution, then $\mu=\pm 1.$
Without loss of generality we can assume that $\mu  =
1,$ because  $(-p)$  induces  $\pi,$ too.

For an involution $s,$ inducing $\s,$ we
have $\vk=\dim V=\dim V_s^-+\dim V_s^+.$
Therefore we can assume that $(\dim V^{-}_s= \gamma  \wedge \dim V^{+}_s= \vk)$
(otherwise one can use $-s$).

By Lemma \ref{0.1.3} there is a basis
$\{e_i: i < \vk\}$ of $V,$
in which both  $s_1$  and  $s_2$ are diagonalized. Let
$A_k= \{i: s_k e_i= -e_i\}, k=1,2.$ Then
\begin{equation} \label{eq1.1.2}
V^-_p= \str{ e_i: i \in  (A_1 \cup A_2) \setminus (A_1\cap A_2)}
\end{equation}
Therefore $\dim V^-_p \le \gamma,$ because  $|A_1|  =
|A_2| = \gamma.$ So
$$
\gamma'= \min(\dim V^-_{p}, \dim V^+_{p}) \le \gamma.
$$

Suppose now that  $s_1s_2= -s_2s_1.$ Then $s_2$  is
conjugate to  $(-s_2).$ Hence the involution $s$ is a
$\vk$-involution. So $\gamma' \le \gamma.$

Now we prove the converse.  Suppose that $\gamma' \le \gamma.$
Choose  a basis of  $V$ in
the form  $\{e_i: i \in  I_1\} \cup \{e_i~:~i\in
I_2\} \cup  \{e_i: i \in  I\},$ where $I_1,I_2,I$  are
disjoint index sets of powers $\gamma,\gamma',$ and
$\vk,$ respectively.  We define $s_1$  and  $s_2$
as follows:
\begin{alignat*}4
s_1e_i &= -e_i,&\quad  & i \in  I_1,\qquad\qquad  &  s_2e_i &= -e_i,&\quad &i \in  I_1,  \\
s_1e_i &= e_i,        && i \in  I_2,              &  s_2e_i &= -e_i, && i \in  I_2, \\
s_1e_i &= e_i,        && i \in  I,                &  s_2e_i &= e_i,  && i \in  I.
\end{alignat*}
Clearly, $s_1$   and   $s_2$   are  conjugate and
commuting $\gamma$-involutions, and the involution
$\widehat{s_1s_2}$ is a $\gamma'$-involution.

Now  we  prove  (b).  If  a $\gamma$-involution
covers some $\gamma'$-involution, then by  (\ref{eq1.1.2})
we have
\begin{align*}
\gamma' = & |(A_1\cup A_2) \setminus (A_1\cap A_2)| =|A_1| + |A_2| - 2|A_1\cap A_2| =\\
          & 2(\gamma -|A_1\cap A_2|).
\end{align*}
The converse is obvious.
\end{proof}

\begin{Cor} \label{OnlyVKsCoverMe}
Suppose that $\chr D \ne 2.$ Then

{\rm (a)} Every $\vk$-involution of $\pgl V$
covers every involution of the first kind;

{\rm (b)} if some $\gamma$-involution covers some $\vk$-involution,
then $\gamma=\vk.$
\end{Cor}

\begin{proof}
By Lemma \ref{Cov(s,pi)}.
\end{proof}

\section{`First kind' and `first order'} \label{1.2}

In this section  we  solve  the  problem of a group-theoretic
characterization of involution of the first kind in
the group  $\pgl V.$ We show that the set
of  all involutions of the first kind is a $\varnothing$-definable
subset of  $\pgl V.$

In his book  \cite[pp. 8-13]{Die1}  Dieudonn\'e
showed that the extremal involutions in the projective
general linear group $\pgl W$ of finite dimension at
least three over a division ring of characteristic
$\ne 2$ can be  distinguished from other involutions
of this group by group-theoretic methods.

He applied (if $\dim W \ne 6$) the following
techniques. First of these is the use of maximal sets
of pairwise commuting and pairwise conjugate
involutions ($m$-sets, for short). Second one is (in
our terms) the use of the relation $\Cov.$ For
example, the power of an $m$-set of extremal
involutions is less than the power of an $m$-set of
$\gamma$-involutions for any $\gamma  \ge  2.$ The
involutions of the second kind either provide the
greater powers of $m$-sets or cover (up to conjugacy)
more involutions than the extremal ones. In the case
$\dim W = 6$ more delicate methods are used.

Our task is more general: we have to distinguish
the involutions of the first kind in $\pgl V$  from involutions
of the second kind by means of first order logic. The method
of $m$-sets is (in the general case) not first order.
Furthermore, in the infinite-dimensional case
it even has not the algebraic efficiency, because, for
example, the power of an  $m$-set of extremal
$\pgl V$-involutions coincide with the power of an
$m$-set of  $n$-involutions for any natural  $n \in  \N.$

\medskip

{\it I.} Let {\it $D$  be a division ring of characteristic
$\ne 2.$ } The following formula is an obstacle for
the involutions of the second kind, because there is no
involution of the second kind satisfying it:
$$
Ob(x) = (\exists y_1y_2y_3) (\bigwedge_{k\ne m} x \sim x^{y_k}x^{y_m}
\wedge x = x^{y_1}x^{y_2}x^{y_3}),
$$
($\sim$ is the conjugacy relation).

\begin{Prop} \label{BestHorses}
An involution  $\sigma \in \pgl V$ satisfies the formula  $Ob$
if and only if  $\sigma$  is  a $\gamma$-involution, where
$\gamma  = 4\gamma'$ for some cardinal $\gamma'.$
\end{Prop}

\noindent {\bf Remark.} Therefore $\gamma$ is either an
infinite cardinal, or finite, which is a multiple
of four.

\begin{proof}
Suppose for a contradiction that some involution of the second kind
$\sigma$ satisfies the formula  $Ob.$ Let  $\sigma =
\hat s,$ where  $s^2= \lambda \cdot \id(V)$ and
$\lambda$ is in $Z(D),$ but is not a square in  $Z(D).$

It is clear that each involution  $\sigma'$
which is a conjugate of $\sigma$  is induced (in
particular) by a  transformation $s'$ such
that  $s'{}^2= \lambda \cdot \id(V).$

If  $\models  Ob[\sigma],$  then  $\sigma$  covers
itself, and hence  $\sigma = \sigma_1\sigma_2=
\sigma_2\sigma_1,$ where $\sigma_1$ and
$\sigma_2$  are conjugates of $\sigma$ (the second
equality holds, since all $\s,\s_1,\s_2$ are involutions). For some
$s_1$ and  $s_2,$ which induce  $\sigma_1$ and
$\sigma_2$  we have
\begin{alignat*}2
&\bullet s  = \nu \cdot s_1s_2,         && \nu  \in  Z(D),\\
&\bullet s^2_k = \lambda \cdot \id(V),  &\quad &k=1,2, \\
&\bullet s_1s_2 = \pm s_2s_1.
\end{alignat*}
To verify the latter equality, one can use an argument
similar to that used to prove Lemma \ref{PMLem}.

Therefore
$$
s^2= \nu s_1s_2\nu s_1s_2= \nu^2 s_1s_2s_1s_2=  \pm \nu^2s^2_1s^2_2=
 \pm \nu ^2\lambda^2\cdot \id(V).
$$

Thus, there exists  $s$ with $\sigma = \hat s$ such that
$s^2= -\id(V).$ The condition  $\models Ob[\sigma]$
also implies  that   $\sigma$
is  a product of three conjugates of $\sigma$:
$\sigma = \sigma_1\sigma_2\sigma_3.$
As we have just observed, $s_ks_m= -s_ms_k,$ where $k\ne m.$
So for some $\mu \in Z(D)$
\begin{align*}
s^2 &=  \mu s_1s_2s_3\mu s_1s_2s_3= \mu^2(-1)^2s^2_1s_2s_3s_2s_3\\
    &= \mu^2(-1)^{2+1}\cdot s^2_1s^2_2s^2_3=\mu^2(-1)^{6}\id(V)=\mu^2\id(V).
\end{align*}

Hence  $\sigma$  is an involution of the first kind, a
contradiction.

Assume now that  $\models Ob[\sigma],$ where  $\sigma$
is an involution of the first kind. Therefore $\sigma$
covers itself. By  \ref{Cov(s,pi)} $\sigma$  is a $\gamma$-involution,
where the cardinal  $\gamma $  is even (in particular,
infinite): $\gamma  = 2\delta.$ Clearly, we should
only consider the case $\gamma  < \aleph_0.$
Let  $\sigma = \sigma_1\sigma_2\sigma_3$  or
$s = \mu s_1s_2s_3,$ where $\s_i$ is a conjugate of $\s$
and $\mu  \in  Z(D).$ It is easy to see that $\mu=\pm 1$:
\begin{equation}
s = s_1 s_2 s_3  \text{  or  } s = -s_1 s_2 s_3.
\end{equation}
By multiplying both sides of the equations in
(\theequation) by $-\id(V)$ if necessary we may
suppose that $\gamma  = \dim V^-_s < \aleph_0.$
Furthermore, $s_ks_m= s_ms_k, k\ne m,$ because the
equation $s_k s_m= -s_m s_k$ holds only for
$\vk$-involutions.

Since $s_1,s_2,s_3$ are pairwise commuting, we may
apply Lemma \ref{0.1.3}. Therefore the $V^+$-subspace of the
involution $(-s_1s_2s_3)$ has finite dimension, and
it cannot be equal to $s.$

Thus,  $s = s_1s_2s_3.$ Consider a basis  $\{e_i:
i<\vk \}$  in which all $s_1,s_2,s_3$  are
diagonalized. Construct (as in the proof of Lemma \ref{Cov(s,pi)})
the sets $A_k= \{i: s_ke_i=-e_i\}, k=1,2,3.$ It is
obvious that the cardinal $\dim V^-_s$  is equal to
$$
|A_1 \cap A_2 \cap A_3|+ |A_1\setminus (A_2 \cup A_3)|+
|A_2 \setminus (A_1 \cup A_3)|+ |A_3 \setminus  (A_1 \cup A_2)|.
$$
Since   $s \sim s_k s_{m}, k \ne m,$ we have $|A_1 \cap A_2|=|A_1 \cap A_3|
=|A_2 \cap A_3| = \delta,$ because, for example,
$$
\gamma  = \dim V^-_s= |A_1|+ |A_2|- 2|A_1 \cap A_2| =
4\delta - 2|A_1 \cap A_2|
$$
Therefore
\begin{align*}
|A_1 \setminus (A_2 \cup A_3)| & =
|A_1| - (|A_1 \cap A_2|+ |A_1 \cap A_3| - |A_1 \cap A_2 \cap A_3|) \\
       &= |A_1 \cap A_2 \cap A_3|.
\end{align*}
Hence  $\dim V^-_s= 4|A_1 \cap A_2 \cap A_3|.$

We  prove  the   converse.  Let
$$
\{e_{i,n}:  i <\gamma', n<7\} \cup \{e_i: i<\vk \},
$$
where $\gamma' \le \vk$ is a cardinal, be a basis of  $V.$ Consider the
involutions $s_1,s_2,s_3~\in~\gl V$ such
that
\begin{itemize}
\item[(1)]  $s_ke_i= e_i,$ where $i<\vk;$
\item[(2)]  the following table is realized for all $i < \gamma'$:
\end{itemize}
$$
\begin{array}{rrrrrrrr}
        & s_1 & s_2 &  s_3 & s_1s_2s_3 & s_1s_2  & s_1s_3 & s_2s_3\\
e_{i,0} & -1  &  1  &   1  & -1        & -1      & -1     &  1\\
e_{i,1} & -1  &  1  &  -1  &  1        & -1      &  1     & -1\\
e_{i,2} & -1  & -1  &  -1  & -1        &  1      &  1     &  1\\
e_{i,3} & -1  & -1  &   1  &  1        &  1      & -1     & -1\\
e_{i,4} &  1  & -1  &  -1  &  1        & -1      & -1     &  1\\
e_{i,5} &  1  & -1  &   1  & -1        & -1      &  1     & -1\\
e_{i,6} &  1  &  1  &  -1  & -1        &  1      & -1     & -1
\end{array}
$$
(a column of the table demonstrates the behaviour of
the corresponding transformation on the set $e_{i,0},
e_{i,1},\ldots, e_{i,6}).$ It is easy to check that
$\hat{s}_1$ (a $4\gamma'$-involution) satisfies the
formula $Ob.$
\end{proof}

\begin{Cor} \label{I'mVKInv}
$\vk$-involutions of the group $\pgl V$ over a
division ring of characteristic $\ne 2$ are exactly
involutions, satisfying  the formula
$$
K(x) = Ob(x) \wedge (\forall y)(Ob(y) \rightarrow \Cov(x,y)).
$$
\end{Cor}
\begin{proof}
By  Corollary \ref{OnlyVKsCoverMe}  and Proposition \ref{BestHorses}.
\end{proof}

\begin{Prop} \label{FKsAreDefnbl1} Let  $\sigma$  be
an involution of the group $\pgl V$ over a division
ring of characteristic $\ne 2.$ Then the following
statements are equivalent:

{\rm (a)} $\sigma$  is an involution of the first kind;

{\rm (b)} $\sigma$  satisfies the formula
$$
FK_1(x) =
K(x) \vee (\exists y)(K(y) \wedge \Cov(y,x) \wedge \neg \Cov(x,y)).
$$
\end{Prop}

\begin{proof}
(a) $\Rightarrow$ (b) is obvious by Corollaries \ref{OnlyVKsCoverMe}
and \ref{I'mVKInv}.

(b) $\Rightarrow$ (a). Assume that $\models FK_1[\sigma]$  and
$\sigma$  is not a $\vk$-involution. Let  $\pi$
be a $\vk$-involution, covering  $\sigma.$ Hence
there is a transformation $s \in  \gl V$  in the
preimage of  $\sigma,$  which is a product
$p_1p_2$ of two  $\vk$-$\gl V$-involutions. We
have  $p_1p_2= \pm  p_2p_1.$ If  $p_1$ commutes with
$p_2,$  then  $s$  is a $\gl V$-involution.

So let us consider the case  $p_1p_2=-p_2p_1.$
Clearly $p_1 V^+_2= V^-_2,$
where  $V^{\pm}_2$  are subspaces of  $p_2.$
Choose two bases  $\{e_i: i<\vk\}$  and
$\{e_{i^*}: i<\vk\}$ of the subspaces
$V^+_2$ and  $V^-_2,$ respectively, such that
$p_1e_i= e_{i^*}, i<\vk .$ So
\begin{alignat*}4
&se_i     &&= p_1p_2e_i     &&= e_{i^*},            &&i<\vk,\\
&se_{i^*} &&= p_1p_2e_{i^*} &&= -p_1e_{i^*} = -e_i.
\end{alignat*}

We shall show that  $\sigma$  covers  $\pi.$ Partition the
set  $\vk$  into four subsets
$I_1,I_2,I_3,I_4$ of power $\vk.$ We define the
transformations  $s_1$ and  $s_2$  as follows
\begin{xalignat*}2
s_1e_{i_1} &= e_{i_3}, &  s_2e_{i_1} &=e_{i_2}, \quad i_k\in  I_k, \quad k\in \{1,2,3,4\},\nonumber \\
s_1e_{i_2} &= e_{i_4}, &  s_2e_{i_2} &=-e_{i_1},\nonumber \\
s_1e_{i_3} &= -e_{i_1}, &  s_2e_{i_3} &=e_{i_4},\nonumber \\
s_1e_{i_4} &= -e_{i_2}, &  s_2e_{i_4} &=-e_{i_3}.\nonumber
\end{xalignat*}

Clearly,  $\sigma_1$  and   $\sigma_2$  are
conjugate to  $\sigma.$ It is easy to verify that
$s_1s_2= s_2s_1.$ Let $r$ denote the
transformation $s_1s_2.$ Since  $s_1^2=
s_2^2= -\id(V),$ $r$  is a $\gl V$-involution.
Furthermore,  $r(e_{i_1}) = e_{i_4}$ and
$r(e_{i_2}) = -e_{i_3},$ where $i_k \in  I_k$ and $k\in \{1,2,3,4\},$
and hence $r$  is a $\vk$-involution. So $\sigma$
covers each $\vk$-$\pgl V$-involution, contradicting
$\models FK_1[\s].$
\end{proof}

{\it II. $D$ is a division ring of  characteristic}  2.  In
contrast to the case  {\it I},  we may apply the  methods
of  Dieudonn\'e from  \cite{Die1}.
As a byproduct of his classification of automorphisms
of the groups $\pgl{n,D},$ for $3 \le n < \aleph_0$ and
$\chr D =2,$ he proved that in this group the conditions
`to cover itself' and `to be an involution of the first
kind' are equivalent. We transfer this result
to  infinite dimensions.

\begin{Prop} \label{FKsAreDefnbl2}
Let  $\chr D =2.$ Then the set
of all $\pgl V$-involutions covering themselves coincides
with the set of all involutions of the first kind.
\end{Prop}

\begin{proof}
If  $\sigma,\pi$  are  $\pgl V$-involutions
then $\sigma\pi = \pi\sigma$ iff  $sp = ps.$
Suppose that a  $\pgl V$-involution  $\sigma$ of the
second kind covers itself. Let  $\sigma = \hat{s}$
and  $s^2= \lambda\cdot \id(V),$ where  $\lambda$   is
 an element of $Z(D)$ which is not  a square in
$Z(D).$ If  $\sigma =\sigma_1\sigma_2$ for some
$\sigma_1$  and $\sigma_2,$ which are conjugate to
$\sigma,$ then  $s = \mu s_1s_2,$ where  $\mu  \in
Z(D)$ and  $s^2_k= \lambda \cdot \id(V), k=1,2.$
Squaring the both parts of the equation  $s = \mu
s_1s_2,$ we obtain the equation  $\lambda  = \mu^2
\lambda^2.$ So  $\lambda  = \mu^{-2},$ but this is
impossible. The argument is due to Dieudonn\'e
\cite[p. 17]{Die1}.

Conversely, we show that every
$\pgl V$-involution of the first kind covers itself. Let  $s$  be a
$\gl V$-involution. According to Section \ref{0.12},
there exists a basis  $\{e_i~:~ i~\in~I~\} \cup
\{e_j~:~j~\in~J\} \cup  \{d_i~:~i~\in~I~\}$  of  $V$
such that
\begin{alignat}2 \label{eq1.2.1}
&se_i = e_i,    &&       i\in I,  \\
&se_j = e_j,     &&       j\in J, \nonumber \\
&sd_i = d_i+e_i,&\quad & i\in I. \nonumber
\end{alignat}

Assume that  $|I| \le  |J|.$ Then the cardinal  $|J|$
coincides with  $\dim V,$ in particular, $|J| \ge  \aleph_0.$
So one can partition $J$  into two subsets  $J_0$  and
$J_1$ of powers  $|I|$  and  $|J|,$ respectively.
Consider  $s_0\in  \End V$ such that
\begin{alignat*}2
& s_0e_i = e_i,              &&      i \in  I, \\
& s_0e_j = e_j,              &&      j \in  J,\\
& s_0d_i = d_i+e_i+e_{f(i)}, &\quad &i \in  I
\end{alignat*}
where  $f$  is a bijection from $I$  onto  $J_0.$  The
element  $s_0$ belongs to  $\gl V,$ because the
system  $\{e_i: i\in I\} \cup  \{e_j: j\in J\}
\cup  \{d_i+e_i+e_{f(i)}: i\in I\}$  is a basis of
$V.$ It is easy to check that
$s_0 \sim s$  and $ss_0 \sim s.$ Indeed,
$ss_0$ acts on the vectors of the
basis $\{e_i: i\in I\} \cup  \{e_j: j\in J\} \cup
\{d_i: i\in I\}$  as follows
\begin{alignat}2  \label{eq1.2.2}
& ss_0e_i= e_i,               && i \in  I,\\
& ss_0e_j= e_j,               && j \in  J,\nonumber \\
& ss_0d_i= d_i+e_{f(i)},&\quad & i \in  I. \nonumber
\end{alignat}
Equations (\ref{eq1.2.2}) may be rewritten
in the following form:
\begin{alignat*}2
& ss_0e_j = e_j,     &&       j \in J_0, \\
& ss_0e_k = e_k,     &&       k \in  I\cup  J_1,\\
& ss_0c_j = c_j+e_j, &\quad & j \in  J_0.
\end{alignat*}
where $c_j=d_{f\inv(j)}.$ The conjugacy of  $s$  and
$ss_0$ follows then from $|I| = |J_0|$  and
$|I\cup  J_1| =|J|.$

Suppose that  $|I| > |J|.$ Clearly,
$|I| = \vk  \ge  \aleph_0.$ Partition  $I$
into two subsets  $I_0$  and  $I_1$ of power $\vk.$
Consider the transformations  $s_0,s_1\in  \gl V$ such that
\begin{alignat*}3
s_0e_i &=  e_i,                  & s_1e_i &=  e_i,  && i \in  I,\\
s_0e_j &=  e_j,                  & s_1e_j &=  e_j,   && j \in  J,\\
s_0d_i &=d_i+e_i+e_{f(i)},\qquad & s_1d_i &=  d_i+e_i, && i \in  I_0,\\
s_0d_i &=  d_i+e_i,              & s_1d_i &= d_i+e_i+e_{f^{-1}(i)},
&\quad & i \in  I_1,
\end{alignat*}
where  $f$  is a bijection from  $I_0$  onto
$I_1.$ Both  $s_0$  and $s_1$ are conjugates of $s$
(since, for example, for the former one the system
$\{e_i+e_{f(i)} : i \in I_0\} \cup \{e_i : i \in I_1\}$
is a basis of the subspace $\str{e_i : i \in I =I_0 \cup I_1}$).
Their product $s_0s_1$ is also a conjugate of $s$:
\begin{alignat*}2
s_0s_1e_i &= e_i,              &&        i \in  I,\\
s_0s_1e_j &= e_j,              &&        j \in  J,\\
s_0s_1d_i &= d_i+e_{f(i)},     &&        i \in  I_0,\\
s_0s_1d_i &= d_i+e_{f^{-1}(i)}, &\quad & i \in  I_1.
\end{alignat*}
\end{proof}

\begin{Th} \label{FKsAreDefnbl}
There exists a first order formula $FK(x)$ in the
group-theoretic language such that $\pgl V \models FK[\sigma]$ if and
only if $\sigma$  is an involution of the first kind.
\end{Th}

\begin{proof}
All we have to do is to
construct a sentence  $\theta$ such that
$\pgl V \models  \theta$  iff $\chr D =2.$ One may take as
$\theta$  the sentence
\begin{equation} \label{eqCharD=2}
(\forall x)(x^2=1 \rightarrow  \Cov(x,x)) \vee
(\forall x)(x^2 =1 \wedge \Cov(x,x) \rightarrow  Ob(x))
\end{equation}

Let  $\chr D \neq  2.$ By Lemma \ref{Cov(s,pi)} a
1-involution cannot cover itself. By Lemma \ref{Cov(s,pi)}
and Proposition \ref{BestHorses} a 2-involution can cover itself,
but does not satisfy the formula  $Ob.$

Let now the underlying division ring $D$ be of characteristic 2 and
of power $\ge  5.$ Then there are two distinct elements  $\mu_1,\mu_2$
in $D \setminus \{0,1\}$  such that
$\mu_1+\mu _2+1 \neq  0.$ Let   $\mu_0=1.$ So
if $s_0,s_1,s_2 \in  \gl V$  are defined by the following
conditions where $k \in \{0,1,2\}$
\begin{alignat*}2
s_ke_i &= e_i,  && i\in I,   \\
s_ke_j &= e_j,  && j\in J, \\
s_kd_i &= d_i+\mu _ke_i, &\quad & i\in I,
\end{alignat*}
then  (a) $s_k \sim s_0, k=1,2,$ (b) $s_0 \sim s_0s_1s_2,$
(c) any product of two distinct
transformations in  $\{s_0,s_1,s_2\}$  is
conjugate to  $s_0.$ Therefore $\models Ob[\hat{s}_0].$

Let, finally, $D$ have characteristic 2 and $|D|\le
4.$ By the Wedderburn theorem, $D$ is a field.  As any
finite field is perfect, every element in $D$ is a
square. So the group $\pgl V$ has no involutions of
the second kind.
\end{proof}

\section{The reconstruction of the betweenness relation
($\mbox{\rm char} D \ne 2$)} \label{1.3}

In his book \cite{Ba} Baer described the
automorphisms of the groups  $\gl W$  over division
rings of characteristic $\neq 2.$ His methods were
different from the methods by
Mackey--Rickart--Dieudonn\'e. Namely, instead of
interpretation of the set  $P^{(1)}(W),$ lines and
hyperplanes of $W,$ in the group  $\gl W$  he
reconstructed (in logical terms) {\it by means of
monadic second order logic} the structure
$\str{P(W);B};$ here  $B$  is the ternary {\it
betweenness} relation on the set  $P(W)$:
$B(L_0,L_1,L_2)$ iff $(L_0 \subseteq L_1 \subseteq
L_2)$ or $(L_0~\supseteq~L_1~\supseteq~L_2),$ where
$L_0,L_1,L_2$ are elements of $P(W),$ the projective
space over $W.$

In this  and in the next sections we show that the
structure
$$
\cP(V)=\str{P(V);\,\subseteq}
$$
is reconstructible in the group $\pgl V$  by means of
first order logic. Since the group $\pgl V$  is
interpretable in  $\gl V,$ the structure $\cP(V)$ is
reconstructible in  $\gl V$  by means of first order
logic. So the above-mentioned results from \cite{Ba}
are significantly strengthened.

First, we shall reconstruct in  $\pgl V$  the
structure
$$
\cPG'(V) = \str{\pgl V,P^*(V);\circ,B,act},
$$
where  $P^*(V)$  is the set of proper non-zero subspaces of
$V,$  $\circ$  is the
composition law on  $\pgl V,$ $B$ is the restriction of
the betweenness relation  on  $P^*(V),$ and $act$ is
the ternary relation for the action of the
group $\pgl V$ on the set $P^*(V).$ Then we shall
reconstruct $\cP(V)$  in  $\cPG'(V).$

\begin{Th} \label{RecoveringOfBetweenness}
The structure  $\cPG'(V)$  can be {\rm(}uniformly in $\dim V$ and $D${\rm)}
reconstructed without parameters in the group $\pgl V$
by means of first order logic.
\end{Th}

In this section we investigate the case  $\chr D \ne
2$ and in the next section we prove Theorem
\ref{RecoveringOfBetweenness} for
the case $\chr D =2.$ In each case we shall realize
the following strategy. The first step will be the
reconstruction of the set $P^{*}(V).$ Let $C$ denote
the two-placed relation on $P(V)$ defined by the
condition $(L_0~\subseteq~L_1) \vee (L_0~\supseteq~L_1).$
The next step will be a reconstruction of the relation  $C.$
Having the reconstruction of  $C$  one can
easily reconstruct  $B,$ because  $B(L_0,L_1,L_2)$  is
true iff
\begin{align} \label{eqC=>B} C(L_0,L_2) \wedge
(\forall L)(L \in  P^{(1)}(V) \wedge C(L,L_0) \wedge
     C(L,L_2) \rightarrow C(L,L_1)).
\end{align}

We have proved the $\varnothing$-definability of the set of all $\pgl
V$-involutions of the first kind (Theorem \ref{FKsAreDefnbl}).
Thus, we may use the variables  $x,y,z,\ldots$ {\it  only for the involutions of
the first kind.} Since we shall work later only
with the involutions of the first kind,  it is
convenient to call them simply {\it involutions.}

{\it From now on and to the end of this  section  we
suppose that  $\chr D \ne 2.$} Convention: the letters
$R$  and   $S$   will be always used for  the
subspaces  of  an  involution   $\sigma.$ Recall that
the subspaces  of   $\sigma$   are  identical  to
the subspaces  $V^{\pm }_{s}$  of an involution
$s \in  \gl V$  in the preimage of  $\sigma.$ The
letter   $\rho$   will  be  used  for  the  extremal
involutions (1-involutions); we shall also always use
the letter $N$ for the line of an extremal involution
$\rho,$ and the letter $M$ for its hyperplane.
The following fact is simple, but very useful.

\begin{Lem}{\rm (\cite[Lemma 2.2]{Ri2}). } \label{Heads&Tails}
\mbox{\rm (a)} Let  $r,s$  be involutions of the
group  $\gl W$ over a division ring of characteristic  $\ne 2,$
and suppose that $r$ is extremal. Then  $rs = sr$
if and only if a subspace of $s$ is contained in the
hyperplane of $r$ and another subspace of $s$
contains the line of $r;$

{\rm (b)} in particular, extremal $\pgl V$-involutions
$\rho_1$ and $\rho_2$ commute if and only if
$(N_1\subseteq M_2 \wedge N_2\subseteq M_1)$
or $(N_1=N_2\wedge M_1=M_2).$
\end{Lem}

In order to reconstruct $P^*(V),$ we first actually
reconstruct $P^{(1)}(V),$ interpreting its elements by
minimal pairs in $\pgl V.$ We start therefore with

\begin{Lem} \label{MPsAreDefnbl1} The set of all
minimal pairs is a $\varnothing$-definable subset of
the group $\pgl V.$
\end{Lem}

\begin{proof} As we noted in Section \ref{0.12}, it
follows from Theorem \ref{Mckey&Rckrt} by Rickart
that, modulo definability of extremal involutions, the
set of minimal pairs is a $\varnothing$-definable in
$\gl V.$ A careful analysis of the conditions of
Theorem  \ref{Mckey&Rckrt} shows that in order to
obtain $\varnothing$-definability of minimal pairs in
$\pgl V$ only two things need to be done:
\begin{itemize}
\item we have to prove $\varnothing$-definability of
$\pgl V$-extremal involutions;
\item and to construct a first order formula, say, $Com(x,y)$
such that $\models Com[\s,\pi],$
where $\s,\pi$ are involutions (of the first kind)
in $\pgl V,$ if and only if $\s$ and $\pi$ commute iff
they have preimages $s,p \in \gl V$ which commute (it may happen
that $sp=-ps$).
\end{itemize}

By Lemma \ref{Cov(s,pi)}, extremal $\pgl V$-involutions
(1-involutions) cover only 2-in\-vo\-lu\-tions, or, in other
words, involutions covered by extremal ones are conjugate.
The latter condition is obviously $\varnothing$-definable,
and hence there is a first order formula, which will be denoted
by $E_1(x),$ whose realizations are exactly $\pgl V$-extremal involutions.
(We  use the index ${}_1$ to specify the case
under consideration: $\chr D \ne 2;$ for the same
purpose in the next section we shall use the index
${}_2.$)

Let us construct the formula $Com.$ The key technical
step here is a proof of $\varnothing$-definability of
the set of all triples $\str{\s;\rho_1,\rho_2}$ such
that the subspaces of involution $\s$ and the
subspaces of extremal involutions $\rho_1,\rho_2$
realize the following configuration:
\begin{align} \label{eqOneHeadOnAPillow}
& \{ (N_1\subseteq  S \wedge M_1\supseteq  R) \wedge
(N_2\subseteq  R \wedge M_2\supseteq  S) \} \vee \\
& \{ (N_1\subseteq  R \wedge M_1\supseteq  S) \wedge
(N_2\subseteq  S \wedge M_2\supseteq  R)\},  \nonumber
\end{align}
that is, the lines of involutions $\rho_1,\rho_2$ lie in distinct
subspaces of $\sigma,$ and, dually, their hyperplanes
contain distinct subspaces of $\sigma.$

\begin{Clm}
The set of all triples $\str{\s,\rho_1,\rho_2}$ with
{\rm (\theequation)} is $\varnothing$-definable in $\pgl V.$
\end{Clm}

Let us deduce the conclusion of the lemma from the latter
Claim. Suppose that a formula $\vartheta(x;x_1,x_2)$ defines the
triples with (\theequation).

If involutions  $\sigma,\pi\in  \pgl V$ as well as some
their  preimages commute, then  $\pi$
preserves the subspaces of $\sigma.$ So for  every
pair  of  extremal  involutions   $\str{\rho_1,\rho_2}$
realizing with $\sigma$  the
configuration (\ref{eqOneHeadOnAPillow}),  the triple
$\str{\sigma;\rho^{\pi}_1,\rho_2}$  also realizes
(\ref{eqOneHeadOnAPillow}).

If  $\sigma$   and   $\pi$   commute,  but   $sp =
-ps,$  then   $\pi$ moves  the  subspace
$V^-_{s}$   to  the   subspace    $V^+_{s}.$
Therefore if  $\str{\sigma;\rho_1,\rho_2}$
satisfies  (\ref{eqOneHeadOnAPillow}), then
$\str{\sigma;\rho^{\pi}_1,\rho_2}$ does not.

Thus, one may take as a formula  $Com$  the formula
$$
[x,y]=1 \wedge (\forall x_1,x_2)(\vartheta(x;x_1, x_2)
\rightarrow  \vartheta(x;x^y_1, x_2))
$$

Now we prove the Claim. The following formula can serve as the
formula $\vartheta$ characterizing our triples:
\begin{align*}
\vartheta(x; x_1, x_2) = & \bigwedge_{k=1}^2 E_1(x_k) \wedge [x_k,x] =1
\wedge (x_1\neq x_2) \wedge [x_1,x_2]=1 \wedge \\
& (\forall y)(E_1(y) \wedge [x,y]=1
\rightarrow \bigvee_{k=1}^2 [y,x_k]=1).
\end{align*}

We demonstrate that a triple  $\str{\sigma;\rho_1,\rho_2}$
realizes (\ref{eqOneHeadOnAPillow}) iff $\models
\vartheta[\sigma;\rho_1,\rho_2].$ The necessity part is easy by
Corollary \ref{Heads&Tails}. Suppose now, towards a contradiction, that the
configuration is not realized  by a triple
$\str{\sigma;\rho_1,\rho_2},$ but
$\models \vartheta[\sigma;\rho_1,\rho_2].$
We claim that there exists an  extremal involution
$\rho$  commuting with $\sigma,$ but  not  commuting
with both $\rho_1$  and $\rho_2.$  Since  $\rho_1 \ne
\rho_2\wedge [\rho_1,\rho_2] =1,$ then  $N_1$  and
$N_2$ are distinct lines. Thus, $[\rho_1,\rho_2]=1$
implies  $N_2 \subseteq M_1.$ The line  $N
=\str{a_1+a_2},$ where  $N_k=\str{a_k}$ and $k=1,2,$
does not  lie in  $M_1.$  As (\ref{eqOneHeadOnAPillow})  is not  realized by
$\str{\sigma;\rho_1,\rho_2},$  one  can find both $N_1$
and  $N_2$  in a subspace  $S$  of  $\sigma;$ another
subspace $R$ of $\s$ is therefore contained in $M_1.$ Thus, the involution $\rho,$
constructed by  the  subspaces   $N$ and $M_1,$
commutes with  $\sigma.$ On the other hand, $\rho$
commutes neither with $\rho_1,$ nor  $\rho_2,$
because  $(N\neq N_1\wedge N\neq N_2)$  and $(N \not\subseteq
M_1\wedge N \not\subseteq M_2).$

The proof of the lemma is now completed.
\end{proof}

Let $MP_1(x_1,x_2)$ denote a first order formula
defining the set of all minimal pairs in $\pgl V.$

A natural generalization of the latter Lemma is

\begin{Th}  \label{TheyHaveSmthMutual}
The relation `involutions $\sigma_1$ and  $\sigma_2$
have a unique mutual subspace' is a $\varnothing$-definable
{\rm(}uniformly in $V${\rm)} relation on the group  $\pgl V.$
\end{Th}

\begin{proof} Fix a minimal pair $\avpi_{l},$ which determines a line,
and a pair $\avpi_{l}$ non-conjugate to $\avpi_{h}$ (which
must determine a hyperplane). Let  $\rho$
be an extremal involution, commuting with a non-extremal involution
$\sigma.$ Suppose that the line of $\rho$ lies in a
subspace  $S$  of $\sigma.$ Consider the set of
extremal involutions, satisfying the condition
$$
\chi_{l}(x;\sigma,\rho) = (\exists y)
(\vartheta(\sigma;\rho,y) \wedge MP_1(x,y) \wedge
\avpi_{l}\sim \str{x,y}).
$$
Since the formula  $\vartheta$  describes only configurations
of the form (\ref{eqOneHeadOnAPillow}), then the line $N_1$ of
extremal involution  $\rho_1,$ satisfying $\chi_{l},$
must lie in  $R.$ Conversely, let $\rho_1$ be an
arbitrary extremal involution such that its line $N_1$
is in  $R.$ Assume that $M_1$ is the hyperplane of $\rho_1.$
Clearly, $R = N_1 +(R \cap  M_1).$ Assume that  $N_1 = \str{a_1}.$ Let
$a$ be a non-zero element in $R \cap  M_1$ and  $R \cap
M_1 = \str{a} \oplus  R'.$ We construct an extremal involution
$\rho_2$  using the line $N_1$ and the hyperplane
$M_2 =\str{a+a_1} \oplus  R' \oplus  S.$
The pair  $\str{\rho_1,\rho_2}$ is obviously minimal
$(M_1\ne M_2);$ as $N_1 \subseteq R$ and
$S \subseteq  M_2,$ then  $\rho_2$ commutes with
$\sigma.$ Hence  $\rho_1$  satisfies the condition
$\chi_{l}.$

Therefore, if for a pair of involutions
$\str{\sigma_1,\rho_1},$ where $\sigma_1$ is non-extremal,
we have $\models (\forall x)(\chi_{l}(x;\sigma_1,\rho_1)
\rightarrow  \chi_{l}(x;\sigma,\rho)),$ then  $R_1 \subseteq  R.$
The replacement of the pair $\avpi_{l}$  in $\chi_l$
by the pair $\avpi_{h}$ gives us the condition $\chi_{h}(x;\sigma,\rho)$
such that if
$
\models  (\forall x)(\chi_{h}(x;\sigma_1,\rho_1)
\rightarrow  \chi_{h}(x;\sigma,\rho))
$
then  $S_1 \supseteq  S$  (the  condition
dual to $R_1 \subseteq  R$).

Let us formalize our considerations.

\begin{Clm} \label{SubsInC}  The relation `a subspace
of an involution $\sigma_1$ is in the relation  $C$ with
a subspace of an involution $\sigma_2$' is a $\varnothing$-definable
relation on the group $\pgl V.$
\end{Clm}

\begin{proof}
Consider the formula
\begin{align*}
\chi(t;x,\avy,z) = & E_1(t) \wedge MP_1(\avy) \wedge
E_1(z) \wedge [x,z] =1 \wedge \\
&(\exists z_1)(\vartheta(x;z,z_1) \wedge MP_1(z_1,t) \wedge
\avy  \sim \str{z_1,t}).
\end{align*}
It follows from the above arguments that the formula
$$
(\exists \avy,z_1,z_2)(\forall t)(\chi(t;x_1,\avy,z_1)
\rightarrow  \chi(t;x_2,\avy,z_2))
$$
guarantees the conclusion of Claim \ref{SubsInC} in the case, when
both involutions $\sigma_1,\sigma_2$ are non-extremal. In
the case, when exactly one involution in the pair
$\str{\sigma_1,\sigma_2}$ is extremal, we can use the
formula
$$
(\exists \avy,z_1,z_2)(\chi(x_2;x_1,\avy,z_1)
\vee  \chi(x_1;x_2,\avy,z_2)).
$$
Finally, if both involutions $\sigma_1,\sigma_2$
are extremal, the condition
\begin{align*}
(\exists y)\{ E_1(y) \wedge & (([x_1,y] =1 \wedge MP_1(y,x_2)) \vee 
([x_2,y] =1 \wedge MP_1(y,x_1))) \}.
\end{align*}
may be used.

Summing up all the cases, we can easily construct a formula
$\chi'(x_1,x_2),$ providing the definability of the relation
specified in the claim.
\end{proof}

Let us complete the proof of Theorem
\ref{TheyHaveSmthMutual}. Put
$$
\chi''(x_1,x_2) = (x_1\ne x_2) \wedge
(\exists \avy,z_1,z_2)(\forall t)
(\chi(t;x_1,\avy,z_1) \leftrightarrow
\chi(t;x_2,\avy,z_2))
$$
Hence an arbitrary pair of distinct involutions
$\str{\sigma_1,\sigma_2}$  has a mutual subspace iff  $\models MS_1[\sigma_1,\sigma_2],$ where
$$
MS_1(x_1,x_2) = MP_1(\avx) \vee  (\neg E_1(x_1) \wedge
\neg E_1(x_2) \wedge \chi''(\avx)).
$$
\end{proof}

We are ready now to give a {\it proof of Theorem}
\ref{RecoveringOfBetweenness} {\it for the case $\chr D \ne
2.$} We shall code subspaces from $P^*(V)$ by pairs of
involutions, satisfying the formula
$$
MS^*_1(x_1,x_2) = MS_1(x_1,x_2) \wedge (\exists y)(E(y) \wedge
\bigwedge_{k=1}^2 [x_k,y]=1).
$$
The formula  $MS^*_1$ is useful due to the following property:
if $\models  MS^*_1[\avs]$ and  $S,R_1,R_2$  are subspaces
of $\sigma_1$  and $\sigma_2,$ then there is no
involution with the subspaces $R_1$  and $R_2,$ because
either they both lie in the same hyperplane or their
intersection is at least one-dimensional
(Lemma \ref{Heads&Tails}). Therefore, if $\models~
MS^*_1[\avs]$ and $\models~MS_1[\sigma_1,\tau] \wedge
MS_1[\sigma_2,\tau],$ then a subspace of
$\tau$  coincides with $S.$

Let a pair $\avs$ satisfy  $MS^*_1$ and $S(\avs)$
denote by the mutual subspace of
$\sigma_1$  and $\sigma_2.$ If $\theta(x_1,x_2)$
is the formula
$$
(x_1=x_2) \vee  MS_1(x_1,x_2),
$$
then a quadruple of involutions $\str{\avs,\avpi}$  satisfies
the formula
$$
EP_1(\avx,\avy) = MS^*_1(\avx) \wedge MS^*_1(\avy)  \wedge
\bigwedge_{i,j=1}^2 \theta (x_i,y_j),
$$
iff $S(\avs) = S(\avpi).$ Indeed, assume that $\models
EP_1[\avs,\avpi],$ but  $S(\avs) \ne  S(\avpi).$
Let  $(S,S_1)$ and $(S,S_2)$ be the subspaces of $\sigma_1$ and
$\sigma_2,$  respectively,   $(P,P_1)$ and $(P,P_2)$ be the
subspaces of $\pi_1,\pi_2.$ Without loss of generality we can assume
that $P = S_1,$ therefore $P_1= S,$ because $\models  MS^*_1[\avs].$
By symmetry $P_2= S.$ Hence $\pi_1= \pi_2,$ and
$\not\models  MS^*_1[\avpi].$ The converse is easy.

We reconstruct now the relation  $C.$ Let $\avs, \avpi$ be
two pairs of involutions, satisfying $MS^*_1.$
We claim that $C(S(\avs),S(\avpi))$ holds iff $\models C_1[\avs,\avpi],$
where $C_1$ denotes the formula
$$
C_1(\avx,\avy) = (\forall \avz)(EP_1(\avx,\avz)
\rightarrow \bigwedge_{i,j=1}^2 \chi'(z_i,y_j)).
$$
The `only if' part is obvious. To prove the converse, we
can use the following fact: there exists a subspace
$R,$ a direct complement of $S(\avs)$ such that the
subspace $R$ is in the relation $C$ neither with any
subspace of  $\pi_1,$ nor with any subspace of
$\pi_2.$

To reconstruct the betweenness relation $B$ we may use
(\ref{eqC=>B}).

Clearly, $\varphi\, S(\avs) = S(\avpi),$ where $\varphi  \in
\pgl V,$ holds iff the pairs
$\str{\sigma_1^{\varphi},\sigma_2^{\varphi}}$ and
$\str{\pi_1,\pi_2}$  satisfy the formula $EP_1.$
This provides an interpretation of the action of
$\pgl V$ on  $P^*(V).$ The proof of Theorem
\ref{RecoveringOfBetweenness} for the case $\chr D \ne
2$ is completed.  $\qed$

\section{The reconstruction of the betweenness relation
($\mbox{\rm char } D=2$)} \label{1.4}

{\it Throughout this section we assume, unless otherwise
stated, that $\chr D=2.$}
According to Section \ref{0.12}, one can assign two
subspaces $R=R(\sigma)$ and $S=S(\sigma)$ of $V$ with $R \subseteq S$ to
each involution  $\sigma=\hat{s} \in \pgl V,$ namely
$S=\Fix(s)$ and $R =\Rng(\id(V)+s).$ As in Section
\ref{1.3} we shall call involutions of the first kind
just involutions.

We shall apply the strategy described in the previous
section. Our first step is therefore

\begin{Lem} \label{MPsAreDfnbl2}
The set of all minimal pairs is a $\varnothing$-definable
in $\pgl V.$
\end{Lem}

\begin{proof}
Proposition \ref{MPsInChar2} from Section \ref{0.12} says
that the minimal pais are $\varnothing$-definable
in $\gl V.$ This time in order to transfer this result
to the group $\pgl V$ we need only to prove the
$\varnothing$-definability of $\pgl V$-extremal
involutions, because in the case when $\chr D=2$
commuting involutions (of the first kind) $\s,\pi \in
\pgl V$ have preimages $s,p \in \gl V$ which also
commute (and hence there is no need in a formula similar
to $Com$ from the previous section).

The condition `there are exactly two conjugacy classes of
involutions covered by $x$' was used by Dieudonn\'e for
characterization of the extremal involutions in the
projective general linear groups over division rings
of characteristic 2 and of finite dimensions at least
6 \cite[pp. 13-14]{Die1}. Let $E_2(x)$ denote a first order
sentence
corresponding to the mentioned condition.
One may prove a $\varnothing$-definability of the
extremal involutions in the infinite-dimensional case
by a slight modification of Dieudonn\'e's arguments.

Indeed, we know from Section \ref{1.2} that any
$\pgl V$-involution of the first kind covers
itself (Proposition \ref{FKsAreDefnbl2}).
Any 1-involution can cover some 2-involution:
\begin{alignat*}3
s_0e_i &= e_i,      & s_1e_i &= e_i, &&i < \vk,  \nonumber \\
s_0d_0 &= d_0+e_0,\quad    & s_1d_0 &= d_0, \nonumber\\
s_0d_1 &= d_1,        & s_1d_1 &= d_1+e_1 \nonumber
\end{alignat*}
($s_0$ and $s_1$ are commuting 1-involutions
whose product $s_0s_1$ is a 2-involution).
On the other hand, if $s_0$ and $s_1$ are extremal
involutions in $\gl V$ then $R(s_0s_1) \subseteq R(s_0)+R(s_1)$ since
for every $a \in V$
$$
s_0s_1 a+a=(s_0(s_1 a)+s_1 a)+(s_1 a+a).
$$
Hence $\dim R(s_0s_1) \le 2$ (we reproduce here
an argument from \cite[p. 100]{O'Mea}). Thus,
extremal involutions cannot cover $\gamma$-involutions
with $\gamma > 2,$ and therefore any extremal
involution satisfies $E_2(x).$

We claim now that every $\gamma$-involution,
where $\gamma > 1,$ covers involutions
in at least three conjugacy classes of
involutions, and hence does not satisfy
$E_2(x).$

Let first $\gamma$ be a cardinal $ > 2.$
Then any $\gamma$-involution can cover
some 1-involution:
\begin{xalignat*}4
&\text{\rm (i)}\! & s_0e_i &= e_i,   & s_1e_i &= e_i,    &  i &< \gamma, \nonumber \\
&    & s_0e_j &= e_j,     & s_1e_j &= e_j, \,   &  j &\in  J, \nonumber \\
&    & s_0d_0 &= d_0+e_0, & s_1d_0 &= d_0+e_1,  &   & \nonumber \\
&    & s_0d_1 &= d_1+e_1, & s_1d_1 &= d_1+e_0,  &   & \nonumber \\
&    & s_0d_i &= d_i+e_i, & s_1d_i &=d_i+e_i,\, &  i &\in  \gamma \setminus  2 \nonumber
\end{xalignat*}
as well as some 2-involution:
\begin{xalignat*}4
&\text{\rm (ii)} \! & s_0e_i &= e_i,     & s_1e_i &= e_i,          & i  &< \gamma, \nonumber \\
&        & s_0e_j &= e_j,     & s_1e_j &= e_j,          & j  &\in  J,\nonumber \\
&        & s_0d_0 &= d_0+e_0, & s_1d_0 &= d_0+e_0+e_1,  &    &  \nonumber \\
&        & s_0d_1 &= d_1+e_1, & s_1d_1 &=d_1+e_0,       &    &\nonumber \\
&        & s_0d_i&= d_i+e_i,  & s_1d_i &=d_i+e_i,       & i  &\in  \gamma \setminus  2.\nonumber \\
\end{xalignat*}
So any $\gamma$-involution covers 1-involutions, 2-involutions,
and $\gamma$-in\-vo\-lu\-tions.

In the case when $\gamma=2$ we show that any 2-involution
can cover, for example, some 3-involution (and hence we again
have elements from at least three conjugacy classes):
\begin{xalignat*}3
s_0e_i &= e_i,           & s_1e_i &= e_i,     & i &< \vk , \nonumber \\
s_0d_0 &= d_0+e_0,       & s_1d_0 &=d_0,      &   &\nonumber \\
s_0d_1 &= d_1+e_0+e_1,   & s_1d_1 &= d_1+e_0, &   &\nonumber \\
s_0d_2 &= d_2,           & s_1d_2 &= d_2+e_2. &  & \nonumber
\end{xalignat*}
\end{proof}

It is easy to construct a formula (we denote it by
$\avx  \equiv  \avy$), which is satisfied by a tuple
$\str{\avs_1,\avs_2},$ where  $\avs_1$  and  $\avs_2$
are minimal pairs iff the subspace determined by
$\avs_1$ coincides with the subspace determined by
$\avs_2.$ By \ref{TransvsBasics}(c) the  desired
formula may  be  chosen  in  the following  form:
$$
\bigwedge_{i,j=1}^2 (x_iy_j \sim x_i) \vee (x_i=y_j).
$$

Let us fix, as in the previous section, a minimal pair
$\avpi_l$ with a mutual line and a non-conjugate to
$\avpi_l$  minimal pair  $\avpi_h.$ Consider the
formula
$$
C'_2(x_1,x_2;\avy)= (\forall \avz)(\avz  \sim \avy  \wedge
\avz^{x_1} \equiv  \avz  \rightarrow  \avz^{x_2} \equiv  \avz),
$$
where  $\avy$  is of length 2.

First step to the reconstruction of binary relation $C$
on $P^*(V)$ is

\begin{Clm} \label{C'}
Let  $\sigma_1$  and  $\sigma_2$  be involutions. Then

{\rm (a)} $\models  C'_2[\sigma_1,\sigma_2;\avpi_l]$ iff
$S_1 \subseteq  S_2;$

{\rm (b)} $\models  C'_2[\sigma_1,\sigma_2;\avpi_h]$ iff
$R_1 \supseteq  R_2.$
\end{Clm}

\begin{proof}
Take an arbitrary involution $\s.$ If $N$ is an
arbitrary line, then  $\sigma N=N$
iff  $N \subseteq  S;$ if $M$ is any hyperplane, then
$\sigma M=M$  iff  $R \subseteq  M.$ Let us consider
the second `iff' statement. Assume $\s=\hat s.$
Since $R=\{a+s a : a \in V\} \subseteq M,$ then
for every $m \in M$ the element $m+s m$ is in $M.$
On the other hand, take an element $a \notin M.$
If suffices to prove that $a+s a \in M.$ The elements
$a$ and $s a$ are linearly dependent over $M$:
$\mu a+s a \in M$ for some non-zero $\mu \in D.$ Due to the
$s$-invariance of $M$ we have that $\mu s a+ a\in M,$
and hence $(1+\mu^2)a \in M.$ As the underlying ring
$D$ is of characteristic 2, then $1+\mu^2=(1+\mu)^2=0$
or $\mu=1.$
\end{proof}

So by Claim \ref{C'}(a) a tuple  $\str{\sigma_1,\sigma_2;\avpi_l}$
satisfies the formula
$$
MS_2(x_1,x_2;\avy) = C'_2(x_1,x_2;\avy) \wedge C'_2(x_2,x_1;\avy)
$$
iff   $S_1=S_2$  and  a tuple
$\str{\sigma_1,\sigma_2;\avpi_h}$  satisfies  this
formula iff $R_1=R_2.$

We have also to build formulae
$C''_2(x_1,\avy_1;x_2,\avy_2)$ and
$C'''_2(x_1,\avy_1;x_2,\avy_2)$  describing the
situations of the form $R_1 \subseteq  S_2$  and $R_1
\supseteq  S_2.$ Note that the situation of the form
$S_1 \subset  R_2$ (strict inclusion) is realized iff
the dimension of the underlying vector space is
infinite.

We use as a formula  $C''_2$  the formula
\begin{align*}
C''_2(x_1,\avy_1;x_2,\avy_2) =
(\exists x_3,x_4)[&MS_2(x_1,x_3;\avy_1)
\wedge MS_2(x_2,x_4;\avy_2) \wedge \phantom{aaaaaa} \\
& \{ (C'_2(x_3,x_4;\avy_1) \wedge C'_2(x_3,x_4;\avy_2)) \vee \\
& (C'_2(x_4,x_3;\avy_1) \wedge C'_2(x_4,x_3;\avy_2))\}]
\end{align*}

\begin{Clm} \label{C''}
Let  $\sigma_1$  and  $\sigma_2$ be involutions. Then

{\rm (a)} $\models C''_2[\sigma_1,\avpi_l;\sigma_2,\avpi_h]$
iff  $S_1\supseteq  R_2;$

{\rm (b)} $\models  C''_2[\sigma_1,\avpi_h;\sigma_2,\avpi_l]$  iff
$R_1\subseteq  S_2.$
\end{Clm}

\begin{proof}
(a) There are involutions  $\sigma_3,\sigma_4$
such that $S_1=S_3$  and  $R_2=R_4.$ On
the other hand,  $(S_3 \subseteq S_4 \wedge R_3 \supseteq R_4)$
or $(S_3 \supseteq S_4 \wedge R_3\subseteq R_4).$
Since  $S_3 \supseteq R_3$
and  $S_4\supseteq R_4,$ then in both cases
$S_1 \supseteq  R_2.$

We prove the converse. Suppose that  $\codim S_1 > \dim R_2.$
Hence we can find subspaces  $L_0$  and
$L_1$  such that  $V = S_1 \oplus L_0 \oplus L_1$
and  $R_2 \cong  L_1.$ Then there is an
involution  $\sigma_4$ such that  $S_4=S_1\oplus L_0$
and  $R_4=R_2.$ An involution
$\sigma_3$  is constructed as follows:
$S_3=S_1$  and  $R_3$ is
a subspace in  $S_1$  containing  $R_2$  and
isomorphic to  $L_0 \oplus L_1.$ In the case
$\codim S_1 \le  \dim R_2$  we choose a subspace
$S_4$ with $S_4 \supseteq  R_2$  lying in  $S_1$
such that  $\codim S_4= \dim R_2.$ An involution
$\sigma_4$  is constructed by the subspaces
$S_4$  and $R_2.$ We choose a
subspace  $R_3$ with $R_3 \subseteq  R_2,$  of dimension
equal to $\codim S_1.$ Then we take as $\sigma_3$ an involution
with the subspaces  $S_3=S_1$  and  $R_3.$

Part (b) can be proved by analogy with (a).
\end{proof}

Let us construct  $C'''_2$:
$$
C'''_2(x_1,\avy_1;x_2,\avy_2) =
(\forall x)(MS_2(x,x_1;\avy_1) \rightarrow C'_2(x_2,x;\avy_2)).
$$

\begin{Clm}   \label{C'''}
Let  $\sigma_1$  and  $\sigma_2$  be involutions. Then

{\rm (a)} $\models
C'''_2[\sigma_1,\avpi_l;\sigma_2,\avpi_h]$  iff
$S_1\subseteq  R_2$;

{\rm (b)} $\models  C'''_2[\sigma_1,\avpi_h;\sigma_2,\avpi_l]$  iff
$R_1 \supseteq  S_2.$
\end{Clm}

\begin{proof}
(a) Suppose $S_1\subseteq  R_2.$ Consider an
involution  $\sigma$  such that  $S = S_1.$ Then
$R_2 \supseteq  R,$ because  $R_2 \supseteq  S_1
= S \supseteq R.$ So $\models C'''.$ Conversely,
if for each involution  $\sigma$  such
that  $S = S_1$ we have  $R_2 \supseteq  R,$ then
$R_2 \supseteq  S_1.$ Indeed, the sum of all
subspaces  $R = R(\sigma)$  for such  $\sigma$'s
is equal to  $S_1.$

(b) Suppose $R_1 \supseteq  S_2.$ We have then $S
\supseteq  S_2$ for an involution $\sigma$ with
$R = R_1.$ Hence $\models  C'''.$
Conversely, if for each involution  $\sigma,$  such
that  $R = R_1$  we have  $S_2 \subseteq  S,$ then
$S_2\subseteq R_1,$ because the intersection of all
$S = S(\sigma)$ for  such  $\sigma$'s  is equal to
$R_1.$ \end{proof}

Let us now turn to a {\it proof of Theorem}
\ref{RecoveringOfBetweenness} {\it for the case $\chr D
=2.$}

The elements of $P^*(V)$ will be interpreted by
triples $\avs=\str{\sigma_1,\sigma_2,\sigma_3},$ where
$\sigma_1$ is a (non-identity) involution, and
$\str{\sigma_2,\sigma_3}$ is a minimal pair.

In the case when $\str{\sigma_2,\sigma_3}$ is a minimal
pair with a mutual line, we assign to the triple the
subspace $S_1= \Fix(s_1), \sigma_1= \hat s_1.$
Otherwise we assign to the triple $\avs$ the subspace
$R_1= \Rng(\id(V)+s_1).$

Let $S(\avs)$ denote the subspace, which corresponds
to a triple of involutions
$\str{\sigma_1,\sigma_2,\sigma_3}.$ By Claims \ref{C'}--\ref{C'''}
$C(S(\avs_1),S(\avs_2))$  iff $\models C_2[\avs_1,\avs_2],$
where $C_2$ denotes the formula
\begin{align*}
C_2(x_1,\avy_1;x_2,\avy_2) = &[\avy_1 \sim \avy_2 \wedge C'_2(x_1,\avy_1;x_2,\avy_2)] \vee \\
                             &[\avy_1 \not\sim \avy_2 \& \, (C''_2(x_1,\avy_1;x_2,\avy_2)  \vee  C'''_2(x_1,\avy_1;x_2,\avy_2))].
\end{align*}

Using the equivalence
$$
L_0= L_1 \Leftrightarrow  (\forall L)(C(L,L_0) \leftrightarrow C(L,L_1))
$$
we may construct a formula  $ET_2(x_1,\avy_1;x_2,\avy_2),$
where $|\avy_k|=2,$
such that  $\models ET_2[\avs_1;\avs_2]$  iff  $S(\avs_1) =
S(\avs_2).$ Then, applying (\ref{eqC=>B}), we may reconstruct
the betweenness relation $B.$ Having $ET_2,$ we may easily reconstruct the
action of $\pgl V$  on  $P^*(V).$ $\qed$

Thus, we have interpreted by means of first order
logic the structure $\cPG'(V)=\str{\pgl
V,P^*(V);\circ,B,act}$ in the group $\pgl V$ uniformly
in $\dim V$ for the case, when $\chr D\ne 2$ and for the
case $\chr D=2.$ Recall that these cases can be
distinguished from one another by a suitable
first order sentence (see the formula (\ref{eqCharD=2})).  Using
standard techniques, one can therefore build an
interpretation of $\cPG'(V)$ in $\pgl V$ which is also
uniform in $D.$

\section{The reconstruction of the inclusion relation} \label{1.5}

The following theorem is our crucial result.

\begin{Th}   \label{PrjSpaceInPrjGroup} The projective
space $\cP(V)=\str{P(V);\subseteq}$  can be
reconstructed without parameters in the projective
linear group $\pgl V$  by means of first order logic.
\end{Th}

In the previous section we reconstructed in  $\pgl V$
the structure
$$
\cPG'(V) = \str{\pgl V,P^*(V);\circ,B,act}
$$
(one may surely use the corresponding binary relation
$C$ instead of the betweenness relation $B,$ but this is
sometimes less technically convenient).  In this
section we shall prove that the inclusion relation  on
$P^*(V)$ is a definable relation in the structure
$\cPG'(V),$ that is  the structure
$$
\cPG''(V) =\str{\pgl V,P^*(V);\circ, \subseteq,act}
$$
can be reconstructed in  $\pgl V.$ Therefore
the projective space can be reconstructed in $\pgl V,$
too, due to

\begin{Clm}
The structure $\cP(V)$ is $\varnothing$-interpretable
in $\str{P^*(V),\subseteq}.$
\end{Clm}

\begin{proof}
For $R_1,R_2$ in $P^*(V)$ put $f(R_1,R_2)=R_1$
if $R_1=R_2,$ put $f(R_1,R_2)=V$ if $R_1 \subset R_2$
and put $f(R_1,R_2)=\{0\}$ in other cases. Clearly,
$f$ maps $P^*(V)^2$ onto $P(V),$ and the $f$-preimages
of the equality and inclusion relations are $\varnothing$-definable
in $\str{P^*(V),\subseteq}.$
\end{proof}

So let us concentrate our efforts on the proof
of following

\begin{Prop} \label{RecovInclInPG'}
The inclusion relation on $P^*(V)$  is a $\varnothing$-definable
relation on the structure $\cPG'(V)$
{\rm(}uniformly in $V${\rm)}.
\end{Prop}

{\it Proof.} Clearly, the inclusion relation on  $P^*(V)$ is
definable in  $\cPG'(V)$  iff the condition
$(\dim L=1)$  is definable in this structure. Indeed,
$L_0 \subseteq  L_1$  is equivalent to the condition
$$
(\forall L)(\dim L=1 \wedge C(L,L_0) \rightarrow C(L,L_1)),
$$
where  $C(L_0,L_1)$  is an abbreviation for the
formula $B(L_0,L_0,L_1).$

Consider  the  following function from $P(V)$   to
$\vk$:
$$
\dcd L = \min(\dim L,\codim L).
$$

\begin{Lem} \label{1.5.4}
The  following  conditions  are
definable in  $\cPG'(V)$:

{\rm (a)}  $\dcd L =1;$

{\rm (b)}  $\dcd L \ge  \aleph_0;$

{\rm (c)} $\dcd L = \aleph_0.$
\end{Lem}

\begin{proof}
(a) The condition  $\dcd L = 1$  is
equivalent to
$$
(\forall L_0,L_1)(B(L_0,L,L_1) \rightarrow \bigvee_{k=0}^1 L=L_k).
$$

(b) It is easy to see that the condition is equivalent
to
$$
(\exists \varphi)(\varphi L\ne L \wedge C(L,\varphi L)).
$$

The proof of (c) is based on the
following property of each pair  $\str{S_0,S_1} \in
P^*(V)$  with
$$
(\dim S_0 =\codim S_1 =\aleph_0) \wedge (S_0 \subseteq S_1):
$$
every subspace  $S$  of infinite dimension and
codimension can be transformed by an element of
$\pgl V$  to a subspace lying between  $S_0$  and
$S_1.$

Conversely, if a pair  $\str{S_0,S_1},$ where $\dcd
S_0 \ge  \aleph_0$ and $\dcd S_1 \ge \aleph_0,$ has
the mentioned property, then   $\dcd S_0= \dcd S_1=
\aleph_0.$ Indeed, if a subspace  $S$ with $\dcd S \ge
\aleph_0$ lies between  $S_0$ and  $S_1,$ then
\begin{alignat*}3
& \dim S  &  &\ge  \min\{\dim S_0,\dim S_1\}    && \ge  \aleph_0, \\
& \codim S&  &\ge  \min\{\codim S_0,\codim S_1\}&& \ge \aleph_0.
\end{alignat*}

Since we may choose  $S$ such that  $\dim S = \aleph_0$
or $\codim S = \aleph_0,$ then
$$
\min\{\dim S_0,\dim S_1\} = \min\{\codim S_0,\codim S_1\} = \aleph_0.
$$
Hence $\dcd S_0= \dcd S_1= \aleph_0.$
So  the condition   $\dcd  L = \aleph_0$ is
definable, because the equivalent condition
\begin{align*}
(\dcd L \ge  \aleph_0) \wedge & (\exists L_1)(\dcd L_1\ge  \aleph_0\wedge \\
& (\forall L_2)(\dcd L_2 \ge \aleph_0 \rightarrow
(\exists \varphi)(B(L,\varphi L_2,L_1)))
\end{align*}
does.
\end{proof}

Let  $\varphi  \in  \pgl V$  and  $\varphi  =\hat f.$ The
least  $f$-invariant subspace, containing a line  $N =
\str a,$  is the subspace  $N_{\varphi}=
\sum_{n \in \Z}  \varphi^{n}N = \str{ f^{n} a: n\in \Z}.$
Consider the dual version. Let $M$ be an arbitrary hyperplane and
$M_{\varphi}$ denote the greatest $f$-invariant subspace, which
is contained in  $M.$ Clearly,  $M_{\varphi}$  equals to
$\bigcap_{n \in \Z} \varphi^{n} M.$ One
can verify that the subspaces $N_{\varphi}$  and
$M_{\varphi}$ are definable with parameters
$\{N,\varphi\}$  and $\{M,\varphi\},$ respectively
(uniformly for lines and hyperplanes). Indeed, let
\begin{align*}
\theta(L;L',\psi) = & (\psi L=L) \wedge C(L,L') \wedge \\
 & (\forall L'')((\psi L''=L'' \wedge C(L',L'')) \rightarrow B(L',L,L'')).
\end{align*}
It is easy to see that if $S$ is a line or a hyperplane
then $S_\f$ is the unique (if any) realization of the
formula $\theta(L;S,\f).$ Thus, we may
use in our formulae the expressions of the form
$L_{\varphi}$ for  $L$  with  $\dcd L=1.$

Consider the case  $\vk  = \aleph_0.$

\begin{Clm} \label{LnsAreDfnbl0}
Let  $\dim V = \aleph_0.$ Then the following
conditions are equivalent:

{\rm (a)} $S$  is a line;

{\rm (b)}  $\dcd S =1$ and for each  $\varphi \in
\pgl V$ if $S_\varphi \not\in P^*(V)$ then
$\f$ does not preserve any subspace isomorphic to $S.$
\end{Clm}

\begin{proof}
Let  $S = \str a$  be a line and $\varphi=\hat f.$ If
$S_{\varphi} \not\in  P^*(V),$ then $S_{\varphi}= V =
\str{f^{n} a: n\in \Z}.$ So  there are no
$f$-invariant lines in  $V.$ Conversely, choose a
basis of  $V$ in the form $\{e\} \cup  \{e_n: n \in
\Z \}.$ Consider  $f \in  \gl V$ such that
\begin{alignat*}2
&fe   &&= e,\\
&fe_n &&= e_{n+1},\quad n \in  \Z.
\end{alignat*}

The hyperplane  $\str{e_n: n \in  \Z}$  is
$f$-invariant. Let the hyperplane  $M$  be the span of
the set $\{e_0+e\} \cup  \{e_n:n\neq 0\}.$ It is
easy to see that  $M_{\varphi}= \{0\} \not\in  P^*(V).$
\end{proof}

Clearly, the condition \ref{LnsAreDfnbl0}(b) is
definable. So the conclusion of Proposition
\ref{RecovInclInPG'} is true if  $\vk  = \aleph_0.$
Let  $\vk$  be again an arbitrary infinite cardinal.

\begin{Lem}  \label{dcdM_f>=vk'}
There are $\varphi \in  \pgl V$ and
a hyperplane  $M \in  P^*(V)$  such that
$\dcd M_{\varphi}  \ge  \vk',$ where  $\vk' =
\min\{|D|^{\aleph_0},\vk\}.$
\end{Lem}

\begin{proof}
Let  $D^{\Z}$ be the vector space of functions from $\Z$ to $D.$
Clearly,  $\dim D^{\Z}= |D|^{\aleph_0}.$
Then we may find in $D^{\Z}$ a linearly independent set of
functions  $\{F_j:~j~<~\vk'~\}$  of power  $\vk'.$
Choose  a basis  $V$  in the following form:
$$
\{a\} \cup  \{a_{j,n}:j < \vk', n\in \Z\} \cup  \{b_i: i < \vk\}.
$$

Consider  $f \in  \gl V$ acting on the basis as follows
\begin{alignat*}4
& \mbox{\rm (i)}   && fa       && = a;         \\
& \mbox{\rm (ii)}  && fa_{j,n} &&= a_{j,n+1}, &\quad  j &< \vk', n \in \Z; \\
& \mbox{\rm (iii)}\quad && fb_i     &&= b_i,       &       i &< \vk .
\end{alignat*}

Let  $\varphi = \hat f,$ and let $M$ denote the hyperplane
with a basis
$$
\{a_{j,n}-F_j(-n)a: j < \vk', n\in \Z \} \cup  \{b_i: i<\vk \}.
$$

Let  $\{\delta_n : n \in  \Z\}$ be linear
functions from  $V$  to  $D$ such that the kernel of
$\delta_n$  is  $\varphi^nM$  and  $\delta_n(a)
=1.$ We show that  $\delta_n(a_{j,0}) = F_j(n).$
Since
$$
f^n( a_{j,-n}-F_j(n)a) = a_{j,0}-F_j(n)a,
$$
then  $\delta_n(a_{j,0}-F_j(n)a)=0$  and
$\delta_n(a_{j,0}) = F_j(n).$

It is obvious that  $b \in  M_{\varphi}$  iff,  for all
$n \in \Z,$ $\delta_n(b)=0.$ Let  $b$  be
a non-zero  element  in  $\str{a_{j,0}:j<\vk'}.$ We prove
that $b \not\in  M_{\varphi}.$  Indeed, if  for  all  $n \in
\Z,$ $\delta_n(b) =0,$   where
$b =\displaystyle\sum_{j \in J} \mu_j a_{j,0},$
then we have for an arbitrary integer $n$:
$$
\delta_n(b) = \delta_n(\sum_{j \in J} \mu_j a_{j,0}) =
\sum_{j \in J} \mu_j F_j(n) = 0
$$
or   $\displaystyle \sum_{j \in J} \mu_jF_j= 0.$ Hence the set
$\{F_j: j\in J\}$  is linearly dependent.
\end{proof}

Suppose now that  $\vk  > \aleph_0.$ As a consequence
of the latter theorem we have that any line $N$ satisfies
the following definable condition:
\begin{equation} \label{eqLinesOnly}
(\dcd L=1) \wedge (\exists \varphi)(\exists L')(\dcd L'=1
\wedge L' \not\cong  L \wedge \dcd L'_\varphi > \aleph_0),
\end{equation}
where  $L_0\not\cong  L_1$ is an abbreviation for the
formula $\neg (\exists \psi)(\psi L_0=L_1).$ Conversely, assume
(\ref{eqLinesOnly}) is satisfied by a hyperplane  $S.$ Then there are
a line  $S'=\str a$ and $\varphi = \hat f$  such that the
cardinal $\dcd S'_\varphi  = \dcd \str{f^n a:n \in  \Z}$
is strictly greater than  $\aleph_0.$ On the other
hand, as  $\vk > \aleph_0,$ we have
$$
\dcd \str{ f^n a:n~\in~\Z} \le  \aleph_0.
$$
Therefore the condition $\dim L=1$ is a $\varnothing$-definable.

A final remark:  $\vk  = \aleph_0$  iff  the
structure   $\cPG'(V)$ satisfies the
definable condition  $(\forall L)(\dcd L \le  \aleph_0).$
\qed

\section{Semi-linear groups} \label{2.1}

We devote this and two next sections
to a proof of the following theorem.

\begin{Th} \label{Gl>pGl>pgl>gl} Uniformly in  $\dim V$ and $D,$
\begin{equation}  \label{eq2.1.1}
\TH(\Gl V) \ge \TH(\pGl V) \ge  \TH(\pgl V) \ge \TH(\gl V).
\end{equation}
\end{Th}
(Recall that $\ge$ means `syntactically interprets',
see Section \ref{0.12} for the definition).

{\bf Remarks.} (a) We shall prove in Section \ref{3.1}
that the elementary theory of the projective space
$\cP(V)=\str{P(V);\subseteq}$ syntactically interprets
the second order theory
$\Theory(\str{\vk,D},\LII(\vk^+))$ (see the
Introduction for the definitions of the latter
structure and the logic  $\LII(\vk^+)$). Therefore
$$
\Theory(\pgl V) \ge \Theory(\cP(V)) \ge \Theory(\str{\vk,D},\LII(\vk^+)).
$$
Under the assumption $\vk=\dim V \ge |D|$ the theory
$\Theory(\str{\vk,D},\LII(\vk^+))$ becomes
the full second theory of the structure $\str{\vk,D};$
one can interpret in this theory the elementary
theory of the semi-linear (general linear) group of $V$:
$$
\Theory(\str{\vk,D},\LII(\vk^+))\ge \Theory_2(\str{\vk,D}) \ge \Theory(\Gl V) \ge \Theory(\gl V).
$$
Then it follows from Theorem \ref{Gl>pGl>pgl>gl} that,
under the assumption $\dim V \ge |D|,$ the elementary
theories of the groups $\Gl V,$ $\pGl V,$ $\pgl V,$
and $\gl V$ are pairwise mutually syntactically
interpretable, or, in other words, they have the
same logical power.

We shall prove in Section \ref{2.3} that both
relations $\Theory(\pGl V) \ge \Theory(\Gl V)$ and
$\Theory(\pgl V) \ge \Theory(\gl V)$ hold without any
assumptions on the dimension of $V.$ Note also that in
the general case $\Theory(\gl V) \not\ge \Theory(\Gl
V)$ (Section \ref{3.4}).

(b) The results, we shall consider in Sections
\ref{2.1}--\ref{2.3}, will not be used for a proof of
mutual syntactical interpretability of the elementary
theories of the groups $\pgl V$ and $\gl V$ with the
second order theory $\Theory(\str{\vk,D},\LII(\vk^+)).$ So a
reader who is not interested in the semi-linear case
may now move to Section \ref{3.1}.

We shall prove Theorem \ref{Gl>pGl>pgl>gl} in three
steps following the sign $\ge$  in (\ref{eq2.1.1}).

To start a proof of the first relation, $\TH(\Gl V) \ge \TH(\pGl V),$
we need some facts on so-called semi-involutions. We shall also
use this information in a proof of the second relation,
$\TH(\pGl V) \ge  \TH(\pgl V).$

According to \cite[Chapter I, Section 3]{Die2}, a transformation $\sigma \in
\Gl W$ is said to be a {\it semi-involution} if it
induces an involution in the group $\pGl W.$ Consider
a semi-involution $\sigma.$ Clearly, the square of a
semi-involution is a radiation: $\sigma^2=\lambda \cdot \id(V).$
The semi-involution $\sigma$ can induce the same involution
in $\pGl W$ as an {\it involution} $\pi \in \Gl W,$ that
is $\hat\sigma=\hat\pi.$ It is easy to see that this is
possible iff the following condition holds
\begin{equation} \label{eqMuSigmaMu}
(\exists \mu \in D)(\lambda=\mu^\sigma \mu).
\end{equation}
(Recall from Section \ref{0.12} that $\mu^\sigma$
denotes the action of the associated automorphism of $\sigma$
on a scalar $\mu.$) Hence there is no involution in $\Gl W$
which induces the involution $\hat\sigma$ iff
\begin{equation} \label{eqNoMu's!}
(\forall \mu  \in D)(\lambda  \ne  \mu^{\sigma}\mu).
\end{equation}
Notice the similarity between the concepts we introduce
here and the concepts which produce the partition of the
set of $\pgl W$-involutions on involutions of the first
kind and involutions of the second kind.

In the next section we shall discuss semi-involutions satisfying
(\ref{eqNoMu's!})  in more detail. The key fact on semi-involutions
satisfying (\ref{eqMuSigmaMu}) is the following

\begin{Prop}  \label{SemisInDie2} \mbox{\rm
(\cite[Chapter I, Section 3]{Die2}). }
Assume that a semi-involution $\s \in \Gl W$ satisfies
\eqref{eqMuSigmaMu} that is $\s^2=\l \cdot \id(W)$
and $\l=\mu^\s \mu$ for some $\mu \in D.$ Then
there exists a basis of $W$ on which $\s$ acts
as the radiation $\mu \cdot \id(W).$
\end{Prop}

As an immediate consequence we have

\begin{Cor} \label{BeautyOfNLInvs} \mbox{\rm (a)}
Every non-linear involution in the group $\Gl W$ has a
basis of $W$ which it pointwise fixes.

{\rm (b)} Non-linear involutions in the group $\Gl W$ over
division ring $D$ are conjugate if and only if their associated
automorphisms are conjugate in the group $\aut D.$
\end{Cor}

Proposition \ref{SemisInDie2} is proved in \cite{Die2}
formally for finite-dimensions, but, in fact, the proof
works for arbitrary vector space $W.$ The
proof of Corollary \ref{BeautyOfNLInvs} for semi-linear
groups of characteristic $\ne 2$ can be also found in
\cite[Chapter  VI, Section 6]{Ba}.

\begin{Th} \label{Gl>pGl}
$\TH(\Gl V) \ge \TH(\pGl V).$
\end{Th}

\begin{proof}

By Corollary \ref{PrjsAsFactorByRads} it suffices to show that
the group $\rl V$ is a $\varnothing$-definable
subgroup of $\Gl V.$ For semi-linear groups of
characteristic $\ne 2$ it follows from the well-known
result from \cite{Ba} (Theorem \ref{rl_in_Gl}
below), the group-theoretic description $\rl W$ in the
group $\Gl W.$ The author has no information about a
similar result for the semi-linear groups of
characteristic 2. We shall realize in this case a
natural geometrical plan which works without any
assumptions on the characteristic. Suppose we have
proved that the set of all $\gl V$-minimal pairs is a
definable subset in the group $\Gl V.$ Then the
definability of $\rl V$ can be easily deduced from the fact
that the radiations and only the radiations preserve
all the subspaces in $P^{(1)}(V).$

We can use all the results on the behaviour of the
relation $\Cov$ we obtained earlier due to the
following simple fact:

\begin{Lem} \label{PCov=Cov}
Let $\gamma,\gamma'$ be cardinals $\le \dim V.$
Then some $\gamma$-involution covers some $\gamma'$-involution
in the group $\gl V$ if and only if some $\gamma$-involution
of the first kind covers some $\gamma'$-involution
of the first kind in $\pgl V.$
\end{Lem}

{\it I. The characteristic of  $D$  is not} 2.
The above mentioned result from \cite{Ba}
is the following

\begin{Th} \label{rl_in_Gl} \mbox{\rm (\cite[Chapter VI, Section  6]{Ba}). }
Let $W$ be a vector space of dimension at least $3$ over
a division ring of characteristic $\ne 2.$ Then the subgroup $\rl W$
of the group $\Gl W$ is the centralizer
of the set of all involutions $\sigma$ such that
$\sigma \not\sim -\sigma.$
\end{Th}

In particular, the group $\rl V$ is a definable
subgroup of the group $\Gl V,$ because $-\id(V)$ is the
unique involution in the center of $\Gl V.$

The following claim will help us later to distinguish
the cases $\chr D\ne 2$ and $\chr D=2.$

\begin{Clm} \label{LinsInTheirHands1}
Let $\chr D \ne 2.$ Then every non-linear involution in $\Gl V$
covers every linear involution.
\end{Clm}

\begin{proof}
By Corollary \ref{BeautyOfNLInvs}(b), all non-linear involutions with the same
associated automorphism are conjugate. Using
Corollary \ref{BeautyOfNLInvs}(a), choose a
basis $\{e_i : i < \vk\},$ where $\vk=\dim V,$ which $\sigma$ pointwise
fixes. Let
$I_1\cup I_2$ be a partition of $\vk.$ Consider
an involution $\sigma_1,$ having the same associated
automorphism (of the order 2) as $\sigma$ has, and such that
\begin{alignat*}2
\sigma_1e_i &= -e_i, &\quad i&\in I_1, \\
\sigma_1e_i &= e_i,  & i&\in I_2.
\end{alignat*}
As the associated automorphism of the product of two
transformations in the semi-linear group is the
product of the associated automorphisms of the
factors, the transformation $\sigma_1\sigma$ $(=\sigma\sigma_1)$ is
in $\gl V.$
\end{proof}

We could finish the consideration of the case here,
but we shall do a little more.  After completing
a proof of Theorem \ref{Gl>pGl>pgl>gl}, we are going to
consider the structure of isomorphisms for infinite-dimensional
linear groups of types $\Gamma$L, P$\Gamma$L, GL, and
PGL.  For this purpose we need a group-theoretic
characterization of minimal pairs in these groups.

It is easy to see that $\gl V$ is the centralizer of the subgroup $\rl V$ in $\Gl
V.$ Therefore by Theorem \ref{rl_in_Gl} $\gl V$ is
also a definable subgroup of the group $\Gl V.$
According to Lemma \ref{PCov=Cov}, $\gl V$-extremal involutions are
exactly those $\gl V$-involutions which cover up to conjugacy
just one involution, and hence they are $\varnothing$-definable in $\gl V.$
Then by applying Theorem \ref{Mckey&Rckrt}, we conclude that

\begin{Clm} \label{MPsAreDefnbl1.2}
Let $\chr D \ne 2.$ Then the set of all $\gl V$-minimal
pairs is $\varnothing$-definable in the group $\Gl V.$
\end{Clm}

{\it II. The characteristic of  $D$ is} 2.  We have
demonstrated in Section \ref{1.4} that $\pgl
V$-extremal involutions are exactly $\pgl
V$-involution of the first kind that cover, up to
conjugacy, only two conjugacy classes of involutions. Hence by Lemma
\ref{PCov=Cov} the extremal involutions are
$\varnothing$-definable in $\gl V.$  By Proposition
\ref{MPsInChar2} if the set of extremal involutions is
definable,  the set of all minimal pairs is definable,
too. Therefore,

\begin{Clm} \label{MPsAreDefnlb2.2}
If $\chr D=2,$ the set of all $\gl V$-minimal pairs is
$\varnothing$-definable
in the group $\Gl V.$
\end{Clm}

Let us prove now a fact similar to Claim \ref{LinsInTheirHands1}.

\begin{Clm} \label{LinsInTheirHands2} Let $\chr D=2.$
Then every non-linear involution in $\Gl V$ covers
every linear involution.
\end{Clm}

\begin{proof}
Let  ${\mathcal B} = \{e_i: i\in I\} \cup  \{e_j: j\in J\}
\cup  \{d_i: i \in I\}$ be a basis of $V,$ where
$I \cup  J$ is a partition of $\vk.$ Choose
non-linear involutions  $\sigma_1,\sigma_2\in \Gl V,$
having the same associated automorphism such that
\begin{alignat*}3
&(1)\quad && \sigma_1 \text{ pointwise fixes }{\mathcal B} ;\\
&(2) && \sigma_2e_i= e_i,     && i\in I, \\
&    && \sigma_2e_j= e_j,     && j\in J, \\
&    && \sigma_2d_i= d_i+e_i, && i\in I.
\end{alignat*}
Since  $\sigma_1$ and $\sigma_2$  have the same
associated automorphism, then $\sigma_1 \sim
\sigma_2.$ It is easy to see that $\sigma_1\sigma_2 =
\sigma_2\sigma_1.$
According to the description
of $\gl V$-involutions given in Section \ref{0.12},
an arbitrary involution in $\gl V$ can be obtained
in this way.
\end{proof}

The first order condition `there is an involution
covering, up to conjugacy, just one involution' holds
for the group $\Gl V$ iff the characteristic of the
division ring $D$ is not 2.  Indeed, by Claims
\ref{LinsInTheirHands1} and \ref{LinsInTheirHands2}
non-linear involutions are out of play. Then we use
Lemma \ref{PCov=Cov}. Hence the interpretation of the
theory $\TH(\pGl V)$ in the theory $\TH(\Gl V)$ can be
done uniformly in $D.$
\end{proof}

\section{ $\mbox{\rm PGL}(V)$ is a definable subgroup of
$\mbox{\rm P$\Gamma$L}(V)$} \label{2.2}

\begin{Th} \label{pGl>pgl} The subgroup $\pgl V$
is $\varnothing$-definable in the group $\pGl V,$
and therefore the theory $\TH(\pGl V)$ syntactically
interprets the theory $\TH(\pgl V).$
\end{Th}

\begin{proof}
We start with the well-known description of the
subgroup $\pgl W$ in the group $\pGl W$ of dimension
at least $3$ \cite[Chapter III, Section  2]{Ba}. Let us discuss key points
of this description.  Let $\hat \s$ is an element
of $\pGl V.$ It is first proved that
$\hat{\sigma}$ is in the subgroup $\pgl W$ of $\pGl V$ iff
the associated automorphism of $\sigma$ is
inner.  Then the family of transformations satisfying the
condition
\begin{equation} \label{eqTheyAreInPGL}
(\exists L)(\dim L = 2 \wedge (\forall N)(N \subset  L \rightarrow
\hat{\sigma}N = N)),
\end{equation}
(the variable $N$ passes through the set of all lines) is
considered; all such transformations belong to $\pgl
W.$ The final point of the description of ~ $\pgl W$
in ~ $\pGl W$ is

\begin{Th} \label{pgl_in_pGl} \mbox{\rm
(\cite[Chapter III, Section  1]{Ba}) } Let
$\dim W \ge  3.$ A transformation $\hat{\sigma}$ in
~ $\pGl W$ belongs to ~ $\pgl W$ if and only if
$\hat{\sigma}$ is a product of {\rm(}at most
three{\rm)} transformations satisfying the condition
\eqref{eqTheyAreInPGL}.
\end{Th}

We shall prove that the group $\pgl V$ is a definable
subgroup of ~ $\pGl V.$ We shall use a weaker version
of (\ref{eqTheyAreInPGL}). Namely,
\begin{equation} \label{eqTheyAreInPGL+}
(\ref{eqTheyAreInPGL})  \vee (\exists L)(\codim L =2 \wedge
      (\forall M)(M \supseteq  L  \rightarrow \hat{\sigma}M = M)),
\end{equation}
where the variable $M$ passes through the set of all hyperplanes.
Clearly, the second disjunctive term in \eqref{eqTheyAreInPGL+} is
~ the condition dual to the condition \eqref{eqTheyAreInPGL}. Assume
$\dim W \ge  3.$ Let us show that in the group $\pGl W$ the
condition \eqref{eqTheyAreInPGL+} holds  only for elements
$\hat{\sigma} \in \pgl W.$
Indeed, suppose there is a subspace
$L$  of codimension  2 such that every hyperplane
containing $L$ is a $\hat{\sigma}$-invariant. Let $e_1, e_2$ be a pair
of linearly independent elements over $L.$ Since
the hyperplane $\str{e_1}\oplus L$ is $\sigma$-invariant,
$\sigma e_1 = \lambda_1 e_1+ m_1,$ where $m_1 \in L.$
On the other hand, $\sigma e_2 = \lambda_2e_2 +m_2.$
The $\sigma$-invariance of
$\str{e_1+e_2} \oplus L$  gives the equalities  $\lambda  =
\lambda_1 = \lambda_2.$ Furthermore, if $\mu$ is
any element of $D,$ we have
$\sigma(e_1+\mu e_2) = \nu (e_1+\mu e_2)$ + $m$ for
some $\nu \in D.$ Let us calculate  $\sigma(e_1+\mu e_2)$ in another
way: $\sigma(e_1+\mu e_2) = \lambda_1 e_1+m_1+
\mu^{\sigma}(\lambda_2 e_2+m_2).$ This implies $\nu=\lambda$
and $\nu\mu=\mu^\s\lambda,$ and hence $\lambda \mu\lambda^{-1}= \mu^{\sigma}.$
The associated automorphism of $\sigma$ is therefore
inner, and $\hat{\sigma}$ is in $\pgl W.$

So Theorem \ref{pgl_in_pGl} remains true if we replace
\eqref{eqTheyAreInPGL} with \eqref{eqTheyAreInPGL+}.
The group $\pgl V$ can be therefore interpreted in the structure
$\str{\pGl V, P^{(2)}(V); \circ,B,act},$ where $P^{(2)}$
is the set of all subspaces in $P(V)$ of dimension
or codimension $\le 2.$ Indeed, the condition
of being a line or a hyperplane (or $\dcd L=1$ for short as in Section
\ref{1.5}) is surely
definable in the latter structure, and hence the following
condition does:
\begin{align*}
(\exists L_0,L_1)(\bigwedge_{k=0}^1 \dcd L_k  \ne 1 \wedge
& L_0\ne L_1 \wedge B(L_0,L_1,L_1) \wedge \\
 (\forall &L)(\dcd L =1 \wedge B(L,L_0,L_1)  \rightarrow
\varphi L=L)).
\end{align*}
Clearly, this condition is equivalent to the condition
\eqref{eqTheyAreInPGL+}.  Thus, we have to interpret
the structure $\str{P^{(2)};B}$ in $\pGl V.$ First we consider
some facts on semi-involutions of $\Gl V$ in order
to distinguish later involutions they induce in $\pGl V$
from involutions of the first kind of $\pgl V.$

Recall from the previous section that a
semi-involution $\sigma \in \Gl W$ with
$\sigma^2=\lambda\cdot \id(W)$ satisfies the condition
\begin{align*} \tag{\ref{eqNoMu's!}}
(\forall \mu  \in D)(\lambda  \ne  \mu^{\sigma}\mu)
\end{align*}
iff the involution $\hat\sigma$ is not induced
by any {\it involution} in the group $\Gl W.$ If the group
$\Gl W$ consists of such a semi-involution, then the cardinal $\dim W$
is necessarily even that is $\dim W= 2\gamma$ and there is a basis
$$
\{e_i: i<\gamma \} \cup  \{e_{i^*}:i < \gamma \}
$$
of $W$ such that
\begin{alignat}3 \label{eq2.2.4}
&\sigma e_i     &&= e_{i^*},       &\quad & i < \gamma, \\
&\sigma e_{i^*} &&= \lambda e_i.   && \nonumber
\end{alignat}
The latter can be proved using arguments from
Dieudonn\'e's book  \cite[Chapter I, Section  3, A)]{Die2}.
We give a simple proof of this fact for the group $\Gl
V,$ which does not use the mentioned arguments
of Dieudonn\'e.

Take an arbitrary non-zero element $e_0\in V.$ Let $e_{0^*}$ denote
the element $\sigma e_0.$ The vectors $e_0$ and
$e_{0^*}$ are linearly independent, because $\sigma e_0 = \mu e_0$
implies $\l=\mu^\s \mu$. Suppose
now that the system $\{e_i, e_{i^*}: i < \beta \},$ where
$\beta  < \vk,$ is linearly independent. As $|\beta| < \vk=\dim V,$
there is a vector  $e_{\beta }\not\in
\str{e_i,e_{i^*}~:~i < \beta}.$ Assume
$e_{\beta^*} = \sigma e_{\beta}$ and show that
the system $\{e_i,e_{i^*}: i \le \beta \}$ is linearly
independent. If not, we have
$$
\sigma e_{\beta}= \sum_{i < \beta} (\mu_ie_i+\nu_ie_{i^*})
+ \mu_{\beta} e_{\beta}.
$$
Applying $\sigma$ to both parts of this equation, we get
$$
\lambda e_{\beta}= \sum_{i < \beta}(\mu^{\sigma}_ie_{i^*} +
\nu^{\sigma}_i\lambda e_i) + \mu^{\sigma}_{\beta} \sigma e_{\beta}.
$$
Clearly, $\mu_{\beta} \ne  0.$ Therefore
$$
\lambda e_{\beta } - \sum_{i < \beta}  (\mu^{\sigma}_i e_{i^*}
+ \nu^{\sigma}_i \lambda e_i) =
\sum_{i < \beta}  (\mu^{\sigma}_{\beta} \mu_i e_i
+ \mu^{\sigma}_{\beta} \nu_ie_{i^*}) +
\mu^{\sigma}_{\beta } \mu_{\beta} e_\beta
$$
and $\lambda$ must be equal to
$\mu^{\sigma}_{\beta} \mu_{\beta},$ a contradiction.

\begin{Lem} \label{SemisVs2Gammas}
Suppose $\sigma \in \Gl V$ is a semi-involution such that
$\sigma^2 = \lambda \cdot \id(V)$ and $\lambda$ satisfies
the condition \eqref{eqNoMu's!}. Then
the involution $\hat{\sigma}$  covers
in $\pGl V$ every $2\gamma$-involution of $\pgl V,$
where $\gamma \le \vk=\dim V.$
\end{Lem}

\begin{proof}
We use the same method as in the proof of Proposition
\ref{FKsAreDefnbl1}.  Consider a basis of
$V$ of the form
$$
\bigcup_{k=1}^4 \{ e_{i_k}: i_k \in I_k \} \cup
\{ e_i: i < \vk \} \cup  \{e_{i^*}: i < \vk \},
$$
where the index sets $I_1, I_2, I_3, I_4,$ all of
power $\gamma,$ are disjoint from each other and from
$\vk.$ Let $\sigma_1,\sigma_2$ be semi-involutions
with the same associated automorphism as $\sigma$
has. Suppose $\sigma_1$ and $\sigma_2$ take the
vectors from the basis as follows
\begin{xalignat*}2
&\sigma_1e_{i_1} = e_{i_3},          & &\sigma_2e_{i_1}=e_{i_2}, \quad        i_k\in I_k,\quad k=1,2,3,4,\\ \nonumber
&\sigma_1e_{i_2}= e_{i_4},         & &\sigma_2e_{i_2}=\lambda e_{i_1},  \\ \nonumber
&\sigma_1e_{i_3}= \lambda e_{i_1},   & &\sigma_2e_{i_3}= e_{i_4},       \\ \nonumber
&\sigma_1e_{i_4}= \lambda e_{i_2}, & &\sigma_2e_{i_4}= \lambda e_{i_3},  \\ \nonumber
&\sigma_1e_i = e_{i^*},              & &\sigma_2e_i = e_{i^*},\quad i < \vk, \\ \nonumber
&\sigma_1e_{i^*} = \lambda e_i,      & &\sigma_2e_{i^*} =\lambda e_i.  \nonumber
\end{xalignat*}

The transformations $\sigma_1$ and $\sigma_2$ are conjugates
of $\sigma.$
To check the properties below, one should know that
$\lambda^{\sigma}= \lambda$ ($\sigma^3= \sigma^2\sigma =
\lambda \cdot \sigma = \sigma\sigma^2= \lambda
^{\sigma}\sigma$). The transformation $\tau =
\lambda^{-1} \sigma_1\sigma_2$ sends one to another: (a)
$e_{i_1}$  and $\lambda^{-1}e_{i_4},$  (b)
$e_{i_2}$ and $e_{i_3},$ where $i_k\in I_k$ and $k=1,2,3,4.$ Check, for
example, (a):
\begin{align*}
& \tau e_{i_1}= \lambda^{-1} \sigma_1 \sigma_2 e_{i_1}=
\lambda^{-1} \sigma_1 e_{i_2}= \lambda^{-1} e_{i_4},\\
& \tau (\lambda^{-1} e_{i_4}) =
\lambda^{-1} \sigma_1 \sigma_2(\lambda^{-1} e_{i_4}) =
\lambda^{-1} \sigma_1((\lambda^{-1})^{\sigma} \cdot \lambda e_{i_3})
= e_{i_1}.
\end{align*}

Furthermore, (c) $\tau$ fixes each element in
$\{e_i: i < \vk\} \cup  \{e_{i^*}~:~i < \vk\}.$
It is easy to see that  $\sigma_1\sigma_2=
\sigma_2\sigma_1.$ Since the square of the associated
automorphism of $\sigma$  is the inner automorphism
$\mu  \mapsto  \lambda \mu \lambda^{-1},$
and $\mu \mapsto \lambda^{-1} \mu \lambda$
is the associated automorphism of $\lambda^{-1}\id(V),$
then $\tau$  is a linear transformation. According
to (a,b,c) above, $\tau$ (and $\hat \tau$ as well)
is a $2\gamma$-involution. We therefore have $\tau=\l^{-1}\sigma_1\sigma_2=
\l^{-1} \sigma_2 \sigma_1,$ and then the involution $\hat \s$
covers $\hat \tau$ in $\pGl V.$
\end{proof}

We shall reconstruct the structure $\str{P^{(2)}(V);B}$ in
$\pGl V,$ using the methods of Section \ref{1.3} and Section
\ref{1.4}. Much the easier case here is the case when

{\it I. The characteristic of  $D$  is} 2. Let us
agree that the term `a $\gamma$-involution' means `a
$\gamma$-involution of the first kind in the subgroup
$\pgl V$'. We use as the `building materials' for $P^{(2)}(V)$
extremal involution and 2-involutions in $\pgl V.$

Our immediate task is therefore to prove the definability of
$\pgl V$-extremal involutions in $\pGl V.$  Then we show that
2-involutions are involutions, covered by
extremal ones, but not extremal. By Proposition \ref{MPsInChar2}, if
the extremal involutions are definable in  $\pGl V,$
then in $\pGl V$ the set of $\pgl V$-minimal pairs is
definable. As in Section \ref{1.4}, we code the
elements of $P^{(2)}(V)$ by triples
$\str{\sigma,\sigma_1,\sigma_2},$ where  $\sigma$ is a
1- or 2-involution, and  $\str{\sigma_1,\sigma_2}$ is
a minimal pair.  The formulae $C'_2, C''_2,
C'''_2$ {\it mutatis mutandis} retain all needed properties

The extremal involutions could be distinguished by
familiar first order condition `to cover, up to
conjugacy, only two involutions' from Section
\ref{1.4} (see the proof of Proposition
\ref{MPsAreDfnbl2}). The fact that all non-extremal
involutions do not satisfy this condition follows from
the results from Section \ref{1.4}, Proposition
\ref{LinsInTheirHands2}, and Proposition \ref{SemisVs2Gammas}.

{\it II. The characteristic of $D$  is not} 2.
Although the extremal involutions and 2-involutions
are still needed, in contrast to the previous case,
we have to show that the set of {\it all}
involutions of the first kind, or, in other words,
the set of all $\gamma$-involutions, where $\gamma$
is an arbitrary cardinal $\le \vk$ is definable in the group $\pGl V.$

We have to do this, because in Section \ref{1.3} the
proof of definability of $\pgl V$-minimal pairs was
based on Theorem \ref{Mckey&Rckrt}. The first order
formula $MP_1(x_1,x_2),$ where $x_1$ and $x_2$ represent
extremal involutions, and which is a translation of the
conditions of the Theorem into a first order logic,
requires quantification over the elements of sets $c(x_1,x_2)$ and
$c(c(x_1,x_2)).$ Recall that $c(I),$ where $I
\subseteq \gl W$ is the set of all involutions in the
centralizer (in $\gl W$) of $I.$ On the other hand,
if $\sigma_1,\sigma_2$ are $\gl V$-extremal involutions,
the set $c(\sigma_1,\sigma_2)$ consists of
$\gamma$-involutions for every $\gamma  \ge
2$ \cite[Lemma 2.4]{Ri2}. In particular, we cannot
restrict ourselves to the work with extremal involutions as in
the previous case.

We show the definability in $\pGl V$ of the set of all
extremal involutions from the subgroup $\pgl V.$ Since
by \ref{Cov(s,pi)}, \ref{LinsInTheirHands1} and
\ref{SemisVs2Gammas} any non-extremal involutions
covers elements in at least two conjugacy classes of
involutions, we can again use the definable condition
`to cover, up to conjugacy, just one involution'.

To prove the definability of the $\pgl V$-involutions
of the first kind in $\pGl V$ we shall
use the same idea as in Section \ref{1.2}.
First we prove the definability of the set of $\vk$-involutions,
and then apply Proposition \ref{FKsAreDefnbl1}. Obviously, a $\vk$-involution
can cover only elements in the subgroup
$\pgl V$ of the group  $\pGl V.$ According to Proposition \ref{FKsAreDefnbl1},
a $\pgl V$-involution  $\hat{\sigma}$ is of the
first kind iff either it is a $\vk$-involution or it is covered
by some  $\vk$-involution  $\hat{\pi},$ but does not
cover  $\hat{\pi}.$

So let us prove the definability of $\vk$-involutions
in $\pGl V.$ Let the formula $Com^* (x,y)$ be the
formula  $Com(x,y)$ (see the proof of Proposition
\ref{MPsAreDefnbl1}), in which we have replaced all
special variables with ordinary ones. (Recall that all
the variables in $Com$ were special: we required all
the variables to denote {\it involutions of the first
kind}).

If $\hat{\sigma}$  is a $\gamma$-involution such that
$\gamma  < \vk$ and $\hat{\sigma}\hat{\pi} =
\hat{\pi}\hat{\sigma},$ where  $\hat{\pi}$ is an
arbitrary element in $\pGl V,$  then  $\sigma\pi =
\pm \pi\sigma.$ Since $\sigma \not\sim -\sigma,$ then
$\sigma\pi = \pi\sigma.$ Hence  $\pi$
preserves the subspaces of $\sigma.$ Thus,
a $\gamma$-involution satisfies the formula
$$
\chi_0(x) = (\exists y)(xy=yx \wedge
\neg Com^* (x,y))
$$
iff  $\gamma=\vk.$

There could be $\pgl V$-involutions of the second kind
in the set of all realizations of $\chi_0(x).$ We cut
them off using the formula $Ob$ from Section \ref{1.2}:
$$
\chi_1(x) = \chi_0(x) \wedge Ob(x).
$$
Clearly, the formula $\chi_1$ are satisfied only
by $\vk$-involutions, and possibly by some involutions
in $\pGl V~\setminus~\pgl V.$ Any involution
in $\pGl V~\setminus~\pgl V$
can cover all $\vk$-involutions; any $\vk$-involution
covers itself. Hence the formula
$$
\chi_2(x) = \chi_1(x) \wedge (\forall y)(\chi_1(y)  \rightarrow \Cov(y,x))
$$
is satisfied by every $\vk$-involution. On the other
hand, an involution in $\pGl V~\setminus~\pgl V$ cannot
satisfy  $\chi_2,$ since $\vk$-involutions cover only
elements from $\pgl V.$

It follows from the results in Section \ref{1.3} that,
having the set of all $\pgl V$-involutions of the
first kind definable in $\pGl V,$ we can reconstruct
in this group the structure $\str{P^*(V);B},$  and
hence its definable reduct $\str{P^{(2)}(V);B}$
(interpreting the structure $\str{P^*(V);B}$ in the
group $\pgl V,$ we worked {\it inside} the set of all
$\pgl V$-involutions of the first kind, and the
only relations on this set we used were the
relations `$x$ commutes with $y$' and `$x$ is conjugate
to $y$').

The definable condition we used in the end of the
previous section -- `there is an involution covering,
up to conjugacy, just one involution' distinguishes
between the cases $\chr D\ne 2$ and $\chr D=2.$ This
completes the proof of Theorem \ref{pGl>pgl}.
\end{proof}

\begin{Prop} \label{pGl>(pGl,P)}
The structure $\str{\pGl V,P(V);\circ,\subseteq,act}$ can
be interpreted without parameters in the group  $\pGl V$  by means
of first order logic {\rm(}uniformly in $V${\rm)}.
\end{Prop}

\begin{proof} By Theorems \ref{pGl>pgl} and \ref{PrjSpaceInPrjGroup}.
\end{proof}

\section{Overcoming the projectivity} \label{2.3}

The following theorem proves that the expressive power
of first order logic for the infinite-dimensional
classical groups is preserved under the taking of
the projective image. In particular, $\Theory(\pgl V)\ge
\Theory(\gl V),$ and hence the proof of Theorem
\ref{Gl>pGl>pgl>gl} will be completed.

\begin{Th} \label{PH<>H}  Let $H(V)$  be the group
$\gl V$  or the group $\Gl V,$ and $PH(V)$ the
projective image of $H(V).$ Then the theories
$\TH(PH(V))$ and  $\TH(H(V))$ are mutually
syntactically interpretable.
\end{Th}

{\it Proof.} It is obvious that $\Theory(\gl V) \ge \Theory(\pgl
V);$ the relation $\Theory(\Gl V) \ge \Theory(\pGl V)$
has been proved in Theorem \ref{Gl>pGl}.  Then by
Theorems \ref{PrjSpaceInPrjGroup} and
\ref{pGl>(pGl,P)} it suffices to prove the following
proposition.

\begin{Prop} \label{(PH,P)>H}  Let $H(V)$ be $\Gl V$
or $\gl V.$  Then the theory $\TH(H(V))$ is
syntactically interpretable in the elementary theory
of the structure
$\str{PH(V),P(V);\circ,\subseteq,act}$
\end{Prop}

\begin{proof}
Let first  $H(V) =\Gl V.$ Suppose that $L^*_1$ and
$L^*_2$ are elements of $P(V)$ such that
\begin{equation} \label{CondsOnPars}
(\dim L^*_1=1) \wedge (L^*_1\oplus  L^*_2= V).
\end{equation}
We shall interpret the elements of $\Gl V$ by transformations
$\varphi  \in \pGl V$ satisfying the
$\{L^*_1,L^*_2\}$-definable condition
\begin{equation} \label{eqPhiFixesThem}
(\varphi L^*_1 = L^*_1) \wedge (\varphi L^*_2= L^*_2),
\end{equation}
Clearly, the set $\Gamma$ of all elements $\f \in \pGl V$
satisfying (\theequation) is a subgroup of $\pGl V.$

Let us construct the interpretation mapping  $\varepsilon.$
Fix a non-zero element  $a \in L^*_1.$ If some
$\varphi \in \pGl V$ satisfies (\ref{eqPhiFixesThem}), then for a
transformation $f \in \Gl V$ inducing $\varphi,$ there is
a scalar $\lambda_{f}\in D$ such that $fa =
\lambda_{f}a.$ Assume that $\varepsilon(\varphi) =
\lambda^{-1}_{f}\! \left. f \right|_{L_2^*}.$ It is
easy to check that $\varepsilon$ is well-defined.
Indeed, for any $f'$ also inducing  $\varphi,$ we have
$f' = \mu f,$ where  $\mu  \in D.$ Then $\lambda_{f'} =
\mu \lambda_{f},$ and hence
$$
\lambda^{-1}_{f'} \left. f' \right|_{L_2^*} =
\lambda^{-1}_{f} \mu^{-1} \mu \left. f \right|_{L_2^*}  =
\lambda^{-1}_{f}  \left. f\right|_{L_2^*}.
$$

The groups $\Gl{L^*_2}$ and $\Gl V$ are evidently isomorphic.
On the other hand, we show that

\begin{clm}
The mapping $\varepsilon$  is an isomorphism between
the groups $\Gamma$ and $\Gl{L^*_2}.$
\end{clm}

\begin{proof}
Clearly, $\varepsilon$ is a surjective. Let now
$\varepsilon(\varphi_1) = \varepsilon(\varphi_2),$
where $\varphi_k \in \Gamma,$ $\varphi_k = \hat{f}_k,$ and $k=1,2.$
Suppose that $f_k a = \lambda_k a,$ where $k=1,2.$ The associated
automorphism of the transformation $\lambda^{-1}_1
f_1$ is $\mu \mapsto \lambda^{-1}_1 \mu^{f_1}
\lambda_1.$ Thus, if $\lambda^{-1}_1 \left.
f_1\right|_{L_2^*} = \lambda^{-1}_2  \left.
f_2\right|_{L_2^*},$ then $\lambda^{-1}_1 \mu^{f_1}
\lambda_1= \lambda^{-1}_2 \mu^{f_2} \lambda_2$ for
every $\mu \in D.$ Let $m$ be an arbitrary element of
$L^*_2.$ We have
\begin{align*}
\lambda^{-1}_1f_1(\mu a+m) &=
\lambda^{-1}_1 \mu^{f_1} f_1(a) + \lambda^{-1}_1 f_1(m) =
\lambda^{-1}_1 \mu^{f_1} \lambda_1 a + \lambda^{-1}_1 f_1(m) \\
&=\lambda^{-1}_2  \mu^{f_2} \lambda_2 a + \lambda^{-1}_2 f_2(m) =
 \lambda^{-1}_2  \mu ^{f_2} f_2(a) +
\lambda^{-1}_2f_2(m) \\
&=\lambda^{-1}_2f_2(\mu a+m).
\end{align*}
Hence $\lambda^{-1}_1 f_1= \lambda^{-1}_2 f_2,$
or $\varphi_1= \varphi_2.$

Assume further that $\varphi  = \varphi_1 \circ \varphi_2,$
where  $\varphi,\varphi_1,\varphi_2$ are elements of $\Gamma.$
Then  $f = \nu f_1 \circ f_2,$ where
$\varphi_k = \hat{f}_k$ and $k=1,2.$ The scalar $\nu$ is
determined by its behavior on $a$:
$$
fa = \lambda a = \nu f_1 \circ f_2(a) =
\nu f_1(\lambda_2 a) =
\nu \lambda_2^{f_1}   \lambda_1 a.
$$
So  $\nu  = \lambda \lambda^{-1}_1 (\lambda^{-1}_2)^{f_1},$
and then   $\lambda^{-1} f =
\lambda^{-1}_1 (\lambda^{-1}_2)^{f_1} f_1\circ f_2,$ or
$\varepsilon (\varphi) =\lambda^{-1}_1 f_1 \circ \lambda^{-1}_2 f_2
= \varepsilon(\varphi_1) \circ \varepsilon (\varphi_2).$
\end{proof}

Since the set of all pairs $\str{L_1^*,L_2^*}$
satisfying (\ref{CondsOnPars}) is
$\varnothing$-definable, the proof of the Proposition in
the case $H(V)=\Gl V$ is completed.

Consider now the case $H(V) = \gl V.$ The choice of parameters
should be surely done in a different way. If
$V= N \oplus M,$ and for some $f \in \gl V$ both subspaces
$N= \str{a}$ and $M$ are $f$-invariant, then it can happen
that $fa = \lambda a,$ where the scalar $\lambda$ is not
necessarily in the center of $D.$ In this case the transformation
$\lambda^{-1} \left. f \right|_M$ is not in $\gl M.$

This difficulty is easily overcome, if we take as
$L^*_1$ and $L^*_2$  a couple of subspaces satisfying
the condition
$$
(\dim L^*_1=2) \wedge (L^*_1\oplus  L^*_2= V),
$$
and replace the condition (\ref{eqPhiFixesThem}) with
\begin{equation} \label{eqPhiFixesThem+}
(\forall N)(N \subseteq  L^*_1 \rightarrow  \varphi N = N)
\wedge (\varphi L^*_2= L^*_2),
\end{equation}
where the variable $N$ passes as usual through $P^1(V)$
(we have already used the first conjunctive
term in Section \ref{2.2}). If  $\varphi  \in \pgl V$
satisfies (\ref{eqPhiFixesThem+}), then $\varphi$ is induced
by an element $f \in \gl V$ such that
$$
fa = \lambda_{f}a,\quad \lambda_{f} \in Z(D)
\text{ for all  $a$ in  } L^*_1.
$$
Hence $\varepsilon (\varphi) = \lambda^{-1}_{f}
\left. f\right|_{L_2^*}$ is an isomorphism from the
group of all $\varphi$ with (\theequation) onto
$\gl{L^*_2}.$ This completes the proof of the
Proposition, and hence the proof of Theorem
\ref{PH<>H}.
\end{proof}

We close this section with two isomorphism theorems
for infinite-dimensional linear groups (both theorems
easily follow from general isomorphism theorems proved
by O'Meara in \cite[Theorem 5.10, Theorem 6.7]{O'Mea}).
We explained in the Introduction the
reason we consider these theorems: we prove them by
{\it classical methods} basing on the machinery
developed by Mackey, Dieudonn\'e, and Rickart.

\begin{Th}  Let $H(V)$ be the group $\gl V$ or the
group $\Gl V,$ $V_1$ an infinite-dimensional vector
space over a division ring $D_1,$ and suppose that the
group $H(V_1)$  is of the same type as $H(V)$ is. Then

{\rm (a)}  $H(V) \cong  H(V_1)$ if and only if
$\str{V,D} \cong  \str{V_1,D_1};$

{\rm (b)}  every isomorphism  $\Lambda$ between the groups  $H(V)$  and
$H(V_1)$ has the following form
\begin{equation} \label{eq_gl_isoms}
\Lambda(\varphi) =
\varepsilon(\varphi)g\circ \varphi \circ g^{-1}, \quad
\varphi \in  H(V),
\end{equation}
where  $\varepsilon$ is a homomorphism from  $H(V)$ to
$\rl{V_1},$ and  $g$ is a collineation from $V$ onto
$V_1.$
\end{Th}

\begin{proof}
Suppose that  $\Lambda$  is an isomorphism from the
group $H(V)$  onto the group $H(V_1).$ By Corollary
\ref{PrjsAsFactorByRads} and Theorems \ref{Gl>pGl},
\ref{pGl>pgl} the isomorphism $\Lambda$ induces, in a
natural way, an isomorphism $\Lambda'$ of the groups
$\pgl V$ and $\pgl{V_1}.$

We denote by $\cPG(V)$ the two sorted-structure, whose
first sort is the group $\pgl V,$ the second
one is the projective space $\cP = \str{P(V);\subseteq }$ and
the action of $\pgl V$ on $P(V)$ is the only new relation
added to the basic relations on the sorts. It follows from
the proofs of Theorems \ref{RecoveringOfBetweenness} and
\ref{PrjSpaceInPrjGroup} that
$\Lambda'$  induces an isomorphism $\Lambda''$ between structures
$\cPG(V)$  and $\cPG(V_1).$ Therefore by the Fundamental
Theorem of Projective Geometry, we have (a).

As has been shown by Rickart \cite[p.
444-448]{Ri3}, $\Lambda$ has the form
(\ref{eq_gl_isoms}), if it sends any $\gl V$-minimal
pair determining a {\it line} to a $\gl{V_1}$-minimal
pair determining a {\it line} (but he had no proof that
this always takes place; see also the remark below).

By Claims \ref{MPsAreDefnbl1.2} and \ref{MPsAreDefnlb2.2} the
$\gl V$-minimal pairs form a $\varnothing$-definable subset in $\Gl
V.$ By Theorem \ref{Mckey&Rckrt} and Propositiion \ref{MPsInChar2} such
pairs are $\varnothing$-definable in $\gl V.$
Thus, $\Lambda$  preserves the GL-minimal pairs.

Since in the structure $\cPG(V)$ the minimal pairs,
which determine a line, form a $\varnothing$-definable subset,
$\Lambda''$ takes any PGL-minimal pair with a mutual
line to a PGL-minimal pair with a mutual line.
Therefore, by the construction of $\Lambda'',$ the
isomorphism $\Lambda$ must preserve GL-minimal
pairs which determine a line.
\end{proof}

{\bf Remark.} Note that, if the underlying vector space
is of finite dimension, the set of minimal pairs with a mutual
{\it line} can be transformed into the set of minimal pairs with a
mutual {\it hyperplane}; this provides one more class of
isomorphisms, which are not described by the formula
(\ref{eq_gl_isoms}), see, for example, \cite[Chapter IV,
Section  1]{Die2}).

\begin{Th} Let $H(V)$ be the group $\pgl V$ or the
group  $\pGl V,$ $V_1$ an infinite-dimensional vector
space over a division ring $D_1,$ and suppose that the
group $H(V_1)$ is of the same type as $H(V)$ is. Then

{\rm (a)}  $H(V) \cong  H(V_1)$  if and only if
$\str{V,D} \cong  \str{V_1,D_1};$

{\rm (b)}  every isomorphism  $\Lambda$ between the groups $H(V)$  and
$H(V_1)$ has the form
\begin{equation} \label{eq_pgl_isoms}
\Lambda(\varphi) = g\circ \varphi \circ g^{-1}, \quad
\varphi \in H(V),
\end{equation}
where  $g$ is a projective collineation from  $P(V)$ onto
$P(V_1).$
\end{Th}

\begin{proof}
It is known that  $\Lambda$ has the form  (\ref{eq_pgl_isoms}),
if it preserves the PGL-minimal pairs which determine a line
(it easily follows from the arguments
in \cite[Chapter IV, Section  1, Section  6]{Die2}).
\end{proof}

\section{Theories interpretable in {\rm T\lowercase{h}}$(\cP)$} \label{3.1}

Let $\cE,$ $\cP,$ $\cV,$ and $\cD$ denote the
endomorphism ring of $V,$  the projective space over
$V,$ the abelian group of vectors of the space $V,$
and the division ring $D,$ respectively (with their
standard relations). We shall construct new
multi-sorted structures, by gluing together the
structures in the list $\cE,\cP,\cV,\cD.$

Thus, $\cPV$ denotes the following two-sorted
structure: its first sort consists of the elements of
$P(V),$ and the second one consists of the elements of
$V;$ its basic relations are those of $\cP$ and $\cV$
together with membership relation $\in$ between the
elements of $\cV$ and $\cP.$ The elements of the
structure $\cEPV$ are divided into three sorts:
endomorphisms of $V,$ subspaces of $V,$ elements of
$V.$ Its basic relations are those of $\cE,$ $\cP,$ and
$\cV$ together with two ternary relations for the
action of $\End V$ on $V$ and $P(V).$ We denote by $\cVD$
the two-sorted structure whose sorts are $\cV$ and
$\cD,$ and the basic relations are those of $\cV$ and
$\cD$ together with the ternary relation for the
action of $D$ on $V.$

The main personage of the remaining part of the paper, the two-sorted
structure $\str{\vk,D},$ has the following description:
its first sort is the cardinal $\vk$ with no
relations, the second one is the division ring $D$
with standard relations, and there are no other
relations.

Recall that the logic $\LII(\lambda),$ where $\lambda$
is a cardinal, is a second order logic with
 quantification over arbitrary relations of power $<
\lambda,$ and $\Mon(\lambda)$ is its monadic fragment.
The main result of this and the two next sections can
be informally described as follows:  the first order
theories of the structures associated above with $V$
have the logical power at least that of the theory of
the structure $\str{\vk,D}$ (which is `algebra-free'
as much as possible) in the logic $\LII(\vk^+)$ (as
`strong' as possible).

\begin{Th} \label{P>PV>EPV}
$\TH(\cP) \ge \TH(\cPV) \ge  \TH(\cEPV).$
\end{Th}

Let us prove the first $\ge$-statement.

\begin{Prop} \label{3.1.2}
$\TH(\cP) \ge \TH(\cPV).$
\end{Prop}

\begin{proof}
The result is essentially known for arbitrary
dimensions $\ge 3$ (it follows from the well-known
reconstruction the abelian group of vectors of $W$ in
the projective space $P(W)$ over $W$ \cite[Chapter
III]{Ba}), but we suggest an especially simple proof
in the infinite-dimensional case.

Consider two parameters: a line $N^*$  and a
hyperplane $M^*$ such that $N^* \oplus M^* = V.$ We
shall interpret the structure $\cV$ in $\cP$ with the
parameters $N^*$  and $M^*.$  Let $a^*$ be a non-zero
element in $N^*.$ If $a$ is an element of $M^*,$ then
we denote by  $a'$ the line $\str{a+a^*}.$ Clearly,
$a=0$ iff $a' = N^*.$ Let $\Lambda$ denote the set of
all one-dimensional subspaces lying outside $M^*.$ It
is easy to see that the mapping ${}'$ is a bijection
from $M^*$ onto $\Lambda.$ The operation $+$ on $M^*$
induces a binary operation $+'$ on $\Lambda.$ We show
that $+'$  is $\{N^*,M^*\}$-definable.

Consider a pair $a_1,a_2$ of linearly independent elements
of $M^*.$ An element  $a \in M^*$  coincides
with the element  $a_1+a_2$ iff the following hold:

(i)  $\{0\} \subset  \str{a^*,a_1} \cap \str{a^*+a,a^*+a_2}
\subseteq  M^*,$

(ii) $\{0\} \subset  \str{a^*,a_2} \cap  \str{a^*+a,
a^*+a_1} \subseteq M^*$

Necessity:
\begin{align*}
\str{a^*,a_1} \cap & \str{a^*+a,a^*+a_2} =
\str{a^*,a_1} \cap \str{a^*+a_1+a_2, a^*+a_2} = \\
&\str{a^*,a_1} \cap  \str{a^*+a_2,a_1} = \str{a_1},
\end{align*}
because of the linear independence of $\{a^*,a_1,a_2\}.$

Sufficiency. An element $\lambda_1(a^*$+$a) +
\lambda_2(a^*+a_2)$ of the subspace
\mbox{$\str{a^*+a,a^*+a_2}$} is in  $M^*$ iff $\lambda_1=
-\lambda_2.$ Hence if (i) holds, then there exist
$\lambda,\mu  \in D$ such that $\lambda  \ne 0$  and
$\mu a_1= \lambda a - \lambda a_2.$  Since $\lambda
 \ne 0,$ then $\nu a_1= a - a_2$ for some $\nu \in D.$
By analogy one deduces from (ii) that $\nu' a_2= a - a_1$
for some $\nu' \in D.$ We then have that $\nu
a_1+a_2= a_1+ \nu' a_2,$ and therefore  $\nu  = \nu' =
1.$

The linear independence of $a_1,a_2$  is equivalent
to the following conditions: (a) both $a_1'$  and $a_2'$
are different from $N^*,$ and (b) the plane $N^*+a'_1$
does not contain the line  $a'_2.$ On the other hand,
the condition (i) is obviously equivalent to the
condition

(i)$'$ $(N^*+a_1') \cap  (a'+a'_2)$  is different
from $\{0\}$  and lies in $M^*.$

The condition (ii) can be rewritten in a similar way.

Suppose  $a_1,a_2$  are linearly dependent non-zero
elements of  $M^*.$ Then the condition `$a = a_1+a_2$'
is equivalent to the following condition: there exist
$b,c,d \in M^*$ such that

(a) each of the pairs  $\{b,a_1\}, \{c,a_2\}, \{b,a\}$ is
linearly independent, and

(b)  $b+a_1=c,$ $c+a_2=d,$ $b+a=d.$

Thus, we can conclude that the operation  $+'$
on $\Lambda$ is  $\{N^*,M^*\}$-definable in  $\cP.$

Assign to every subspace $L \subseteq  M^*$ the
subspace  $L' = L+N^*.$ The mapping $L \mapsto  L'$ is
injective, and the condition  $a \in L$ is equivalent
to $a' \subseteq  L'.$ Let $P(M^*)'$ be the image of
the set $P(M^*).$ The $\{N^*,M^*\}$-definable
structure $\str{\Lambda,P(M^*)';+', \subseteq}$ is
isomorphic to  $\str{M^*,P(M^*);+,\in},$ and the latter
one is isomorphic to $\cPV.$ Since the set of all
pairs $\str{N^*,M^*},$ whose sum is $V,$ is
$\varnothing$-definable in $\cP,$ the result follows.
\end{proof}

\begin{Prop}
$\TH(\cPV) \ge \TH(\cEPV).$
\end{Prop}

\begin{proof} Let us start with a preliminary remark.
To each endomorphism $\varphi$ of a vector
space $W,$ assign the subspace
$$
L_{\varphi}=\{(a,\varphi a): a \in W\}
$$
of the vector space $W^2.$ On the other hand,
each direct complement $L$ of
the subspace $\{(0,c) : c \in W\}$ in
$W^2$ determines some endomorphism $\varphi  \in
\End W$: if a pair  $(a,b)$  is in  $L,$
then put $\varphi a = b.$ We check that $\varphi $
is well-defined. Indeed, if two pairs  $(a,b)$
and $(a,b')$ are in $L,$ then $(0,b-b') \in L$  and
$b = b'.$ The fact that for every $a \in W$ there exists
an element $b \in W$ such that $(a,b) \in L$ follows
from a decomposition
$$
W^2= L \oplus  \{(0,c): c \in W\}.
$$
It is clear also that $\varphi$ is linear.

Since $V$ is infinite-dimensional, the Cartesian
square of $V$ is isomorphic to $V,$ and it makes sense
to realize the above arguments for the reconstruction
of $\cEPV$ in $\cPV.$

We shall use three parameters: elements
$L^*_1,L^*_2,L^*_3 \in P(V),$ satisfying
the $\varnothing$-definable condition
\begin{equation} \label{eqCondsOnThreeSubs}
\bigwedge_{i \ne j} L_i\oplus L_j= V.
\end{equation}
One easily verifies that  $\dim L^*_i= \codim L^*_i= \vk,$
where $i=1,2,3.$

Let $L$ be a direct complement of $L^*_2$ in
$V.$ The transformation  $\sigma_{L}$ with the graph
$$
\{(a,b): a \in L^*_1,\quad b \in L^*_2,\quad a+b \in L\}
$$
is a linear mapping from $L^*_1$ to $L^*_2,$ as
we have actually proved above. Moreover, every linear
mapping from $L^*_1$ to  $L^*_2$ can be
constructed in such a way. By
(\ref{eqCondsOnThreeSubs}) the transformation
$\sigma=\sigma_{L^*_3}$ is bijective. Hence
$\varphi_{L} = \sigma^{-1}\circ \sigma_{L}$ is an
element of $\End{L^*_1}.$ Formally,  $\varphi_{L}a =
b$  iff the following condition
$$
(a,b \in L^*_1) \wedge (\exists a')(a' \in L^*_2\wedge
a+a'\in L \wedge b+a' \in L^*_3).
$$
is true. The transformations  $\varphi_{L}$ and $\varphi_{L'}$
coincide iff
$$
(\forall a)(a \in L^*_1 \rightarrow  \varphi_{L}a = \varphi_{L'}a).
$$
The analogous arguments may be used for interpretations
of the operations $\circ$  and $+$ on  $\End{L^*_1}.$
Thus, we have reconstructed the first sort of the structure
$$
{\mathcal M} = \str{\End{L^*_1},P(L^*_1),L^*_1},
$$
constructed from $L^*_1$ similarly to the construction of
$\cEPV$ from $V.$ Having the relation $\in$ in the
language of  $\cPV,$ we can reconstruct the relation
$\subseteq$ on $P(L^*_1).$ We reconstructed
$\End{L^*_1}$ with its action on $L^*_1.$ Having the
action of $\End{L^*_1}$  on  $L^*_1,$ one can
obtain the action of $\End{L^*_1}$ on $P(L^*_1).$
And, finally,  ${\mathcal M}  \cong \cEPV,$ because $\dim
L^*_1 = \dim V.$
\end{proof}

\begin{Clm} \label{PV>D}
The division ring $\str{D;+,\cdot}$ can be
reconstructed {\rm(}with parameters from a
$\varnothing$-definable set{\rm)} in the structure $\cPV$
by means of first order logic.
\end{Clm}

\begin{proof} Fix a non-zero element
$a^* \in V.$ We identify the elements of $D$ with the elements
of the line $\str{a^*}.$ Clearly, $\str{D;+} \cong
\str{\str{a^*};+}.$ In \cite[Chapter III, Section 1]{Die2} Dieudonn\'e
proving the Fundamental Theorem of Projective Geometry
interprets (algebraically) the division ring
$\str{D;+,\cdot}$  in the projective space  $\str{P(W);\subseteq},$
where $\dim W \ge 3.$ For the reconstruction of the multiplication
the following diagram is used:

\begin{center}
\unitlength=1.00mm
\special{em:linewidth 0.4pt}
\linethickness{0.4pt}
\begin{picture}(87.67,46.17)
\put(20.67,9.67){\line(4,3){48.67}}
\put(20.67,9.67){\line(5,1){67.00}}
\put(30.00,22.33){\makebox(0,0)[cc]{$b$}}
\put(43.00,32.00){\makebox(0,0)[cc]{$\mu b$}}
\put(42.33,9.67){\makebox(0,0)[cc]{$\nu a$}}
\put(66.00,14.00){\makebox(0,0)[cc]{$\mu\nu a$}}
\put(78.67,16.67){\makebox(0,0)[cc]{$\mu a$}}
\put(53.00,12.00){\makebox(0,0)[cc]{$a$}}
\put(33.00,19.00){\line(3,-2){7.67}}
\put(33.33,19.00){\line(6,-1){18.33}}
\put(43.67,27.00){\line(5,-2){21.00}}
\put(43.67,27.00){\line(6,-1){35.00}}
\end{picture}
{} \\
Diagram 1.
\end{center}

Let us reconstruct the multiplication on $D$
basing on the Diagram 1.

We need one more parameter: an element
$b^* \not\in \str{a^*}.$ Let $a \in
\str{a^*}$ and $\varepsilon(a)$ be the element
of the division ring such that $a = \varepsilon (a)a^*.$
Consider non-zero elements
$a,a_1,a_2$  of $\str{a^*}.$ We claim that
$\varepsilon (a) = \varepsilon (a_1)\varepsilon (a_2)$  iff
$\{a^*, b^*\}$-definable condition
\begin{align} \label{eq3.1.2}
(\exists y)(y \in \str{b^*} \wedge & \str{a^*+ b^*} =
\str{a_1+y} \wedge \\
 & \str{a_2+b^*} = \str{a+y}) \nonumber
\end{align}
is true. If $\models (\ref{eq3.1.2}),$ then $\lambda
(a^*+b^*) = \varepsilon (a_1)a^*+ \mu b^*$ for some
$\lambda, \mu \in D,$ and hence $\mu  = \varepsilon
(a_1).$ For some
$\lambda'\in D$ we have $\lambda'(\varepsilon (a_2)a^*+b^*) =
\varepsilon (a)a^*+ \varepsilon (a_1)b^*,$ and hence
$\varepsilon (a) = \varepsilon (a_1)\varepsilon (a_2).$
The converse is easy.
\end{proof}

\section{Recovering a basis} \label{3.2}

In the remaining part of the paper we shall
suppose that {\it the underlying division ring $D$
satisfies the following condition:}
\stepcounter{equation}

\qquad\parbox{10cm}{\it the number of conjugacy classes of the multiplicative
group $D^*$ is equal to the power of $D^*$.} \hfill (\theequation)

Furthermore, everywhere below the term `division ring'
will be understood to mean only a division ring of the
mentioned form. As the reader will see later in this
section, the condition on $D$ we introduce gives
a natural way of `increasing' of the logical
power of first order theories associated
with $V$ in the case when the dimension
of $V$ is `small' (less or equal to
$|D|$).

\begin{Prop} \label{BasisWeAreLookingFor}
There exist formulae  $\chi(\avX), B(x;\avX)$
in the language of the structure  $\cEPV$ such that
for every tuple $\avA$ from the domain, satisfying  $\chi,$
the set
$$
\{a:\cEPV \models  B[a;\avA]\}
$$ is a basis of
$V.$
\end{Prop}

\begin{Th} \label{EPV>(vk,D)}
$\TH(\cEPV) \ge \TH(\str{\vk,D},\LII(\vk^+)).$
\end{Th}

{\it Proof} (assuming \ref{BasisWeAreLookingFor}).
Let $\avA$ satisfy $\chi$ in $\cEPV,$ and
$\cB=B(\cEPV,\avA).$ Then $\cB$ is a basis of $V.$
Choose in $V$ linearly independent elements
$a^*,b^*$ which lie outside $\cB.$ Let $\avA'=\avA \cup \{a^*,b^*\}.$
We identify $\cB$ with the set $\vk,$ and, using
Claim \ref{PV>D}, introduce on the line $\str{a^*}$ a
structure which is isomorphic to $\cD.$
Put $\cB'= \cB  \cup \str{a^*};$ it will be a copy of the domain
of the structure $\str{\vk,D}.$

It is a well-known fact that the logic with
quantification over arbitrary partial functions and
the full second order logic (which allows
quantification over arbitrary relations) are mutually syntactically
interpretable. Similarly, since in the case of the logic $\LII(\vk^+)$
quantification is allowed only over relations
of power $\le \vk,$ it suffices to
interpret in $\cEPV$ the set of all partial functions
from $\cB'$ to $\cB'$ whose domains are of power less
or equal to $\vk.$

We shall interpret these partial functions by triples
$\avs=\str{\s_0,\s_1,\s_2}$ of endomorphisms of $V$ such that
\begin{itemize}
\item[(a)] $\s_0$ sends each element $b$ of $\cB$ either
to $a^*,$ or to $b^*,$ or to $b;$
\item[(b)] $\s_1$ maps $\cB$ to $\cB$ and the set $\cB_d(\avs)=\{b \in \cB : \s_0 b=a^*\}$
(the preimage under $\s_2$ of the domain of a reconstructible
partial function) onto the set $\cB_r(\avs)=\{b \in \cB : \s_0 b=b^*\};$
\item[(c)] $\s_2(\cB) \subseteq \cB'$ and its restrictions
on $\cB_d$ and $\cB_r$ both are injective.
\end{itemize}

Then, if we substitute any triple $\s_0,\s_1,\s_2$ satisfying
(a,b,c) in the $\av A'$-definable scheme $PF$ below for $\f_0,\f_1,\f_2,$
we obtain a partial function $x_0 \mapsto x_1$ from $\cB'$
to $\cB'$ with the domain of power $\le \vk$:
\begin{align*}
PF(x_0,x_1;&\varphi_0,\varphi_1,\varphi_2) = \\
(\exists &y_0 \in \cB_d(\av\f))(\exists y_1 \in \cB_r(\av\f))
\{ (\varphi_1 y_0=y_1) \wedge (\varphi_2 y_0=x_0) \wedge (\varphi_2 y_1=x_1)\}.
\end{align*}
\qed

{\it Proof of} \ref{BasisWeAreLookingFor}.
We consider here two cases: $\vk > |D|$ and $\vk \le |D|.$

{\it I. $\vk > |D|.$} We shall
use results from the deep paper \cite{Sh3} by
Shelah, where he does, as the title of his paper
says, `interpretation of set theory in the
endomorphism semi-group of a free algebra'. Let $\mathcal
C$ be a variety of algebras in some language $\fL.$
Suppose that $\gamma$ is an infinite cardinal, and
$F_{\gamma}$ is a free algebra with $\gamma$ free
generators. Shelah builds a family
$\fL$-terms, which he calls beautiful terms, satisfying three
special conditions (we describe them below) on free
algebras of infinite rank in $\mathcal C$; we just note
that in some important cases (e.g. for the variety of
abelian groups) the only beautiful and reduced terms
are the terms $x_k,$ where $k \in \N.$ We formulate
one of the key technical results from \cite{Sh3} in
the following form.

\begin{Lem} \label{ShelahInClermont} \mbox{\rm
(\cite[Lemma 4.2]{Sh3}).}
Let $\gamma$ be an infinite cardinal, which is strictly
greater than the power of the language of $\mathcal C.$
Suppose that $\cB$ freely generates $F_\g,$ and
write $\cB$ in the form
$$
\{a^\beta_\alpha : \beta,\alpha < \gamma\} \cup \{b_i : i < \gamma\}.
$$
Then there are a first order formula $\vartheta[x;\avy]$
in the semi-group language and a tuple ${\av \f}^*$
of endomorphisms of $F_\gamma$ such that
$\End{F_\gamma} \models \vartheta[\f;\av{\f}^*]$
if and only if there exist a beautiful
term $t(x_1,\ldots,x_n)$ and ordinals
$\a_1,\ldots,\a_n < \gamma,$ so that
$$
\f(a^\beta_0)=t(a^\beta_{\a_1},\ldots,a^\beta_{\a_n})
$$
for every ordinal $\beta < \gamma.$
\end{Lem}

One deduces from the latter Lemma that

\begin{Cor} \label{AllBeautiesAreX's}
If the only beautiful and reduced terms are
the terms $x_k$ then for any $\beta < \gamma$ the set
$$
\{\f(a^\beta_0) : \End{F_\gamma} \models \vartheta[\f;\av\f^*] \}\quad
(= \{a^\beta_\alpha : \alpha < \gamma \})
$$
is a subset of the basis $\cB$ of power and
copower $\gamma.$
\end{Cor}

Fortunately, we have such a very nice situation for
the variety of vector spaces over the division ring
$D.$ Here the language consists of a two-placed
function symbol $+$ and one-placed function symbols
$\{h_\mu : \mu \in D\}.$

By Shelah's definition a term  $t(x_1,x_2,\ldots,x_n)$ is
said to be {\it beautiful}, if

(A) for every term $q(x_1,x_2,..,x_m)$
\begin{align*}
& t(q(x^1_1,x^1_2,\ldots,x^1_m),
q(x^2_1,x^2_2,\ldots,x^2_m),
\ldots,q(x^n_1,x^n_2,\ldots,x^n_m)) = \\
& q(t(x^1_1, x^2_1,\ldots,x^n_1),t(x^1_2,
x^2_2,\ldots,x^n_2),\ldots,
t(x^1_m, x^2_m,\ldots,x^n_m))
\end{align*}
is an identity of every free algebra $F_\gamma$ of infinite
rank in ${\mathcal C};$

(B)
\begin{align*}
t&(t(x^1_1,x^1_2,\ldots,x^1_n),
t(x^2_1,x^2_2,\ldots,x^2_n),\ldots,
t(x^n_1,x^n_2,\ldots,x^n_n))\\
 &= t(x^1_1,x^2_2,\ldots,x^n_n)
\end{align*}
is an identity of $F_\gamma;$

(C) $t(x,x,\ldots,x) = x$ is an identity of $F_\gamma.$

In the variety of vector spaces over $D$ every term is
equivalent to a term of the form $\sum_{i=1}^n
\mu_ix_i,$ where $\mu_i\in D.$

Clearly, (C) is satisfied only by non-zero terms.
Consider a linearly independent set
$\{e^j_i: i,j \in \N \}$ of power
$\aleph_0$  in  $V.$ If for some non-zero term
$t(x_1,x_2,\ldots,x_n)$ (B) is true, we have
$$
\sum_{i=1}^n \mu_i \sum_{j=1}^n \mu_j e^i_j
=\sum_{i=1}^n \mu_ie^i_i.
$$
Therefore  $\mu_i\mu_j=0$ for $i\ne j$ and
$\mu^2_i= \mu_i$ for each $i=1,\ldots,n.$ Hence the
term $t(x_1,x_2,\ldots,x_n)$ is $x_k$ for some $k.$

Thus, Corollary \ref{AllBeautiesAreX's} and the above
arguments imply that for suitable $a^* \in V$ and
$\av\f^*=(\f^*_1,\ldots,\f^*_m) \in \End V$ we have
that the set of all realizations of the formula
$B_1(x)=B_1(x;a^*,\av \f^*)=\exists \f (\vartheta(\f;\av \f^*) \wedge x=\f a^*)$
is a linearly independent set of power $\vk=\dim V$
such that the linear span of this set has dimension
and codimension $\vk.$ To explain that the set $B_1(\cEPV)$
is linearly independent we write that
$$
(\forall x)\{ B_1(x) \to (\exists M) [\codim M=1 \wedge
x \notin M \wedge (\forall y)((B_1(y) \wedge y\ne x) \to y \in M)]\}
$$
The linear span $L^*$ of $B_1(\cEPV)$ is the unique
realization of the formula
$$
(\forall x)(B_1(x) \to x \in L) \wedge
(\forall L_1)\{(\forall x)(B_1(x) \to x \in L_1) \to (L \subseteq L_1) \}
$$
To explain further that $\dim L^*=\codim L^*=\vk$ we need one parameter.
It can be an invertible endomorphism $\pi^* \in \End V$ such
that
$$
(\pi^* L^* \cap L^*=\{0\}) \wedge (V=\pi^* L^* +L^*).
$$
Therefore the set of all realizations of the formula
$$
B(x;a^*,\av \f^*,\pi^*)=B_1(x) \lor (\exists y)(B_1(y) \wedge x=\pi^* y)
$$
form a basis of $V.$

Finally, the tuple of parameters $\avA=(a^*,\av\f^*,\pi^*)$
can be replaced by any tuple $(b^*,\av{\psi}^*,\rho^*)$ in $\cEPV$ of the same
length which satisfies the $\varnothing$-definable condition
$$
\chi(\avX)=B(\cEPV;\avX) \text{ is a basis of $V.$}
$$

{\it II. $\vk \le |D|.$} Suppose that $\{\lambda_i : i
< \vk\}$ is a set of pairwise non-conjugate elements
of $D^*,$ the multiplicative group of $D,$ and $\mathcal B=\{a_i
: i < \vk\}$ is a basis of $V$ (recall that
$D$ satisfies the condition (10.1)). Let us consider a
diagonalizable transformation $\f^* \in \End V$ such
that
$$
\f^* a_i = \lambda_i a_i, \quad i < \vk.
$$
It is easy to see that $\f^*$ preserves a line
$N=\str{\sum_{i \in I} \mu_i a_i}$ if and only if the
elements $\mu_i \lambda_i \mu_i\inv$ are all equal.
Hence, by the choice of the elements $\l,$ the line
$N$ is of the form $\str{a_j}$ for a suitable $j \in I.$

Taking an endomorphism $\rho^*$ and a
non-zero element $a^* \in V$ such
that
$$
\rho^* a_i =a^*, \quad \forall\ i \in I,
$$
we obtain that the realizations
of the formula
$$
(\exists N)\{ (\f^* N =N) \wedge (x \in N) \wedge (\rho^* x = a^*)\}
$$
are exactly elements of the basis $\mathcal B.$
The proof of the case II can be now
completed as the proof of the previous
case. $\qed$

\section{Theorems on mutual interpretability} \label{3.3}

In this section we give a long list of pairwise mutually
syntactically interpretable  $V$-theories. At first
we make an effort to close the `chain' of $V$-theories,
begun in Sections \ref{3.1}--\ref{3.2} and then add to the
constructed `chain' new elements.

\begin{Th}
\begin{equation}\label{eq3.3.1}
\TH(\str{\vk,D},\LII(\vk^+)) \ge
\TH(\cVD,\LII(\vk^+)) \ge
\TH(\cVD,\Mon(\vk^+)) \ge
\TH(\cP).
\end{equation}
\end{Th}

{\it Proof.} The second $\ge$-statement in
(\ref{eq3.3.1}) is obvious.

\begin{Prop} \label{(vk,D)_II>VD_II}
$\TH(\str{\vk,D},\LII(\vk^+)) \ge \TH(\cVD,\LII(\vk^+)).$
\end{Prop}

\begin{proof} $D^{\vk}_{< \omega}$ is the standard
notation for the set of all finite partial functions
from $\vk$ to  $D.$ Consider the structure
$\str{\vk,D,D^{\vk}_{<\omega}}$ in the language
of $\str{\vk,D}$ expanded by an additional predicate
symbol to distinguish $D^{\vk}_{<\omega}$ and a ternary symbol
$R$ such that $R(\alpha,\mu,f)$ iff $\alpha  \in \vk,$
$\mu  \in D,$ $f \in \Dvkom$  and  $f(\alpha)=\mu.$

Taking into account $\cV \cong \bigoplus_{i < \vk}
D,$ it is easy to reconstruct  $\cVD$
in  $\str{\vk,D,D^{\vk}_{<\omega}}$ by means of
first order logic. Hence $\TH(\str{\vk,D,\Dvkom},\LII(\vk^+)) \ge
\TH(\cVD,\LII(\vk^+)).$

Let $\cM$  be a structure. $\cM_{\text{\rm II}}$
is (quite standard) notation for the structure, with
the domain $\bigcup_{n \in \omega} {\mathcal R}_n(M),$
where $M$ is the domain of $\cM,$ and ${\mathcal R}_n(M), n \in \omega,$ is
the set of all $n$-placed relations on $M;$ here 0-placed
relations represent the elements of $M.$ The unique
$n$-placed $(n \ge 1)$ basic relation on $\cM_{\text{\rm II}}$ says
whether $R(a_1,\ldots,a_{n-1})$ is true or false for any
tuple $a_1,\ldots,a_{n-1} \in M$
and an arbitrary element $R \in {\mathcal R}_{n-1}(M).$
When we require that our ${\mathcal R}_n$ are formed
from the relations of power $\le \vk,$ we obtain the
structure $\cM^{\vk^+}_{\text{\rm II}}.$ The
elementary theory of the structure
$\cM^{\vk^+}_{\text{\rm II}}$ and the theory of the
structure $\cM$ in the logic $\LII(\vk^+)$ are
obviously mutually syntactically interpretable.

We prove now that the $\LII(\vk^+)$-theory
of $\str{\vk,D,\Dvkom}$ is syntactically interpretable in
$\TH(\str{\vk,D}, \LII(\vk^+)).$
Let $\cM^*$ denote the structure
$\cM^{\vk^+}_{\text{\rm II}}.$

We first show that  $\str{\vk,D,\Dvkom}^*$ can be reconstructed
in $\str{\vk,D}^*.$ Since the conditions `$\cA$ is a
finite set'  and `every injection from $\cA$  into
itself is bijective' are equivalent, then it is
possible to reconstruct the set $\Dvkom$ in
$\str{\vk,D}^*.$  Every relation of power $\le  \vk$
on $\vk \cup D \cup  \Dvkom$ can be represented as the
image of a function with the domain in $\vk.$ Hence
the structure $\str{\vk,D,\Dvkom}^*$  is
mutually interpretable with the structure
$\str{\vk,D,\Dvkom,G(\vk)},$ where  $G(\vk)$ is the
set of all functions of the form
$$
g: \vk  \to A_1 \times \ldots  \times A_n,
$$
and $A_i$ is $\vk,$ or $D,$ or $\Dvkom.$ It can be shown
quite easily that the latter
structure is bi-interpretable with the structure
$\str{\vk,D,\Dvkom,G_1(\vk),G_2(\vk),G_3(\vk)},$ where
$G_1(\vk),$ $G_2(\vk),$ $G_3(\vk)$ are the sets of all
functions from $\vk$ to  $\vk,$ from  $\vk$  to  $D,$
and from  $\vk$ to $\Dvkom,$ respectively.

So we have only to interpret the set $G_3(\vk)$ in
the structure $\str{\vk,D}^*.$ This can be done as
follows.  To every function in $G_3(\vk)$  there
corresponds the set $Q$ in
$\vk~\times~\vk~\times~ D$ satisfying the definable condition
$$
\text{\rm `}\{(\beta,\mu) : (\alpha,\beta,\mu) \in Q\} \in
\Dvkom \text{\rm\ for all }  \alpha  \in \vk\text{\rm '}.
$$
in $\str{\vk,D}^*.$ Hence  $\TH(\str{\vk,D},\LII(\vk^+))
\ge \TH(\str{\vk,D,\Dvkom},\LII(\vk^+)).$
\end{proof}

\begin{Prop}
$\TH(\cVD,\Mon(\vk^+)) \ge  \TH(\cP).$
\end{Prop}

\begin{proof} Let this time  $\cVD^*$ be the structure
$\str{\cVD,Pow_{\le \vk}(\cVD)},$ where  $Pow_{\le
\vk}(X)$ is the family of all subsets of $X$ of power
$\le  \vk.$ Clearly, the theories
$\TH(\cVD,\Mon(\vk^+))$ and $\TH(\cVD^*)$ are mutually
syntactically interpretable. Let us reconstruct $\cP$
in the structure $\cVD^*.$

For every subset $\cA \subseteq  V$ of cardinality
$\le \vk,$ there naturally corresponds the subspace
$\str{\cA},$ the linear span of $\cA.$ Since  $\dim V
\le  \vk,$ all subspaces of $V$ can be constructed in
a such way.  Thus, we should prove that the relation
`$a \in \str\cA $' is definable in the
structure $\cVD^*.$

Fix a division subring $K$ of power $\le \vk$ in $D.$
One can consider $V$ as a vector space over $K$ and,
moreover, all the $K$-subspaces of $V$ of dimension
$\le \vk$ are definable in $\cVD^*$ with the parameter
$K.$ This therefore implies that the relation `$a \in
\str\cA{}_K$' is definable with the parameter $K.$
Hence the relation `$a \in \str\cA$'  is definable,
too, because
$$
a \in \str\cA \iff
(\exists K)(K \text{ is a division subring of } D \wedge
|K| \le  \vk \wedge a \in \str\cA_K).
$$
\end{proof}

Let $\frak V$ denote the structure $\cVD$ and
$\TH({\frak V},\text{\rm End}), \TH({\frak V},\text{\rm Sub})$
be the theories of this structure in the logics
with quantifier over endomorphisms and subspaces of
$V,$ respectively.

\begin{Th} \label{MainTheorem}   The following theories
are pairwise mutually syntactically interpretable:  $\TH(\cP),$
$\TH(\cE),$ $\TH(H(V)),$ $\TH({\frak V},\text{\rm Sub}),$
$\TH({\frak V},\text{\rm End}),$ $\TH({\frak V},\Mon(\vk^+)),$
$\TH({\frak V},\LII(\vk^+)),$ $\TH(\str{\vk,D},\LII(\vk^+)),$
where  $H =$ {\rm GL, PGL, End, PEnd.}
\end{Th}

\begin{proof}
One readily checks that
$$
\TH(\cP) \le  \TH({\frak V},\text{\rm Sub}) \le
\TH({\frak V},\text{\rm End}) \le
\TH(\cEPV).
$$
On the other hand, by Theorem \ref{P>PV>EPV}
$$
\TH(\cP) \ge \TH(\cEPV) \ge  \TH(\End V) \ge \TH(\pEnd V);
$$
using the fact that $\pgl V$ is the group of all invertible elements
of $\pEnd V,$ and applying then Theorem \ref{PrjSpaceInPrjGroup}
we have that
$$
\TH(\pEnd V) \ge  \TH(\pgl V) \ge \TH(\cP).
$$
Finally,
$$
\TH(\cP) \ge  \TH(\End V) \ge \TH(\gl V) \ge
\TH(\pgl V) \ge  \TH(\cP).
$$
Therefore each theory mentioned in the theorem
and the elementary theory of the projective
space $\cP$ are mutually syntactically interpretable
and the result follows.
\end{proof}

Consider the logic  $\cL(\vk^+)$ the only difference
of  which from the logic $\LII(\vk^+)$ is an
additional quantifier over arbitrary automorphisms of
the division ring $D.$

\begin{Th} \label{MnThForGls}  The theories
$\TH(\Gl V),$ $\TH(\pGl V),$ $\TH(\str{\vk,D},\cL(\vk^+))$
are pairwise mutually syntactically interpretable.
\end{Th}

\begin{proof} By Theorem \ref{PH<>H}  $\TH(\Gl V)$  and $\TH(\pGl V)$
are mutually interpretable. Let  $\cB$ be some basis of
$V,$ and $a$  a non-zero element of $V.$
By Theorems \ref{PrjSpaceInPrjGroup},
\ref{Gl>pGl>pgl>gl}, and \ref{BasisWeAreLookingFor}
the elementary theory of the structure
$\str{\Gl V,V,\str{a},\cB}$ (with natural
relations) is syntactically interpretable in
$\TH(\Gl V).$ The subgroup $\Phi$ of $\Gl V,$
consisting of all elements of $\Gl V,$ which satisfies
the definable condition
$$
(\forall b \in \cB)(\varphi b = b),
$$
is isomorphic to the group  $\aut D.$
As in Theorems \ref{PV>D} and \ref{EPV>(vk,D)} we once
more identify the set $\str{\cB,\str a}$ with $\str{\vk,D}$
and introduce on $\str a$ a structure
isomorphic to $\cD.$ The subgroup $\Phi$ acting on $\str{a},$
interprets the action of the group $\aut D$ on  $D.$ We
then use Theorem \ref{EPV>(vk,D)} to interpret all the
relations on the structure $\str{\vk,D}$ of power $\le \vk.$

Conversely, using \ref{(vk,D)_II>VD_II} we can interpret
in the theory $\TH(\str{\vk,D},\cL(\vk^+))$
the elementary theory of the structure
$\str{\cVD^{\vk^+}_{\text{II}}, \Phi},$ where
$\Phi  \cong  \aut D$ and the action of $\Phi $ on
$V$ is defined. Namely, we define an action of
$\aut D$ on  $\Dvkom,$ and the method of the proof of
Proposition \ref{(vk,D)_II>VD_II}
gives an action of $\aut D$ on  $\bigoplus_{i < \vk} D.$
Hence we get (a faithful) action of $\aut D$ on $V.$
Thus, the group $\aut D$ is now embedded into
$\Gl V$ and we denote the image under this embedding by
$\Phi.$ It is easy to see that $\Phi$ acts
trivially on some basis $\{e_i:i < \vk\}$  of  $V.$
The group $\gl V$ with its action on $V$ may be also
reconstructed in
$\str{\cVD^{\vk^+}_{\text{II}},\Phi},$  so we can work
with the structure $\str{\cVD^{\vk^+}_{\text{II}},\gl
V,\Phi}.$ One can build an arbitrary collineation as
follows: let $\{b_i: i < \vk\}$ be any basis of $V$
and $\sigma \in \aut D,$ then the transformation
$$
\avs(\sum \mu_ie_i) = \sum \mu_i^{\sigma} b_i
$$
is an element of  $\Gl V.$ Clearly,  $\avs$ is the
composition of some element of $\Phi$  and a
transformation of $\gl V,$ taking the basis
$\{e_i: i < \vk\}$ to the basis $\{b_i: i < \vk\}.$
\end{proof}

Let Mon denote the monadic logic (with quantification
over arbitrary subsets).

\begin{Cor} Let  $\vk \ge  |D|.$ Then
the following theories are pairwise mutually syntactically interpretable:
$\TH(\cP),$ $\TH(\cE),$ $\TH(H(V)),$
$\TH({\frak V},\text{\rm Sub}),$ $\TH({\frak V},\text{\rm End}),$
$\TH({\frak V},\text{\rm Mon}),$
$\TH({\frak V},\LII),\TH(\str{\vk,D},\LII),$ where
{\rm $H = \Gamma$L, P$\Gamma$L, GL, PGL, End, PEnd.}
\end{Cor}

\begin{proof} If  $\vk \ge  |D|,$ then the power of
$\str{\vk,D}$ is the cardinal $\vk.$
Hence
\begin{align*}
\TH(\str{\vk,D},\cL(\vk^+)) \ge
& \TH(\str{\vk,D},\LII(\vk^+)) =\\
& \TH(\str{\vk,D},\LII) \ge
\TH(\str{\vk,D},\cL(\vk^+)).
\end{align*}
\end{proof}

\begin{Cor} \label{3.3.7}  All first order
theories mentioned in Theorems  {\rm \ref{MainTheorem}} and
{\rm \ref{MnThForGls}} are unstable and undecidable.
\end{Cor}

Consider two infinite-dimensional vector spaces $V_1$
and $V_2$ over division rings $D_1$ and $D_2,$
respectively.  Assume also that $\vk_1=\dim V_1$ and
$\vk_2=\dim V_2.$

\begin{Th} \label{3.3.8}  \mbox{\rm (a)} Let  $H =$
{\rm GL, PGL, End}, or {\rm PEnd}. Then the following
conditions are equivalent:

{\rm (i)}  $H(V_1) \equiv  H(V_2);$

{\rm (ii)}  $\str{P(V_1);\subseteq} \equiv  \str{P(V_2);\subseteq};$

{\rm (iii)}  $\cE(V_1) \equiv  \cE(V_2);$

{\rm (iv)}  $\TH(\str{\vk_1,D_1},\LII(\vk^+_1))
=\TH(\str{\vk_2,D_2},\LII(\vk^+_2)).$

{\rm (b)}  Let  $H = \Gamma${\rm L, P$\Gamma$L.}
Then the condition  $H(V_1) \equiv  H(V_2)$ is equivalent to
$$
\TH(\str{\vk_1,D_1},\cL(\vk^+_1))
=\TH(\str{\vk_2,D_2},\cL(\vk^+_2));
$$
in particular, the condition $H(V_1) \equiv  H(V_2)$
implies
$$
\TH(\str{\vk_1,D_1},\LII(\vk^+_1))
=\TH(\str{\vk_2,D_2},\LII(\vk^+_2))
$$
\end{Th}

\begin{proof} Use Theorem \ref{MainTheorem} for (a) and
Theorem \ref{MnThForGls} for (b).  \end{proof}

\section{Examples} \label{3.4}

Throughout this section $\vk,\vk'$ are infinite cardinals,
and  $D,D'$ are division rings. Let $T(\vk,D)$ denote by
the theory $\TH(\str{\vk,D},\LII(\vk^+)).$
In this section we discuss a number of natural
conditions, necessary/sufficient for
\begin{equation}  \label{eqT=T'}
T(\vk,D) = T(\vk',D').
\end{equation}
We shall also investigate the logical strength of the
elementary theories of infinite-dimensional
semi-linear groups over algebraically closed fields. This will
enable us to prove that the condition  $T(\vk,D) =
T(\vk',D')$ -- necessary and sufficient for
the elementary equivalence of groups of types  GL and PGL --
is not sufficient for the elementary equivalence for
groups of types  $\Gamma$L  and P$\Gamma$L.

\begin{Clm} \label{T=T'=>...}  The following conditions
are necessary for $T(\vk,D) = T(\vk',D')$:

{\rm (a)}  $\vk = |D| \leftrightarrow \vk' = |D'|;$

{\rm (b)}  $\vk > |D| \leftrightarrow \vk' > |D'|;$

{\rm (c)}  $\vk < |D| \leftrightarrow \vk' < |D'|;$

{\rm (d)}  $\vk \equiv_{\LII} \vk';$

{\rm (e)}  $\TH(D,\LII(\vk^+)) = \TH(D',\LII(\vk'{}^+)).$
\end{Clm}

\begin{proof}
(a) There is a sentence in the logic $\LII(\vk^+),$
stating the existence of a bijection between $\vk$ and  $D.$
Similar arguments prove (b) and (c).
(d) It follows from  $T(\vk,D) = T(\vk',D')$ that
$\TH(\vk,\LII(\vk^+))=\TH(\vk',\LII(\vk'{}^+)),$ that is $\TH_2(\vk)
=\TH_2(\vk').$
\end{proof}

Under the condition  $\vk \ge  |D|$  the theory  $T(\vk,D)$  becomes
the theory $\TH_2(\str{\vk,D}).$ Hence if
$\vk \ge  |D|$  and $\vk' \ge  |D'|,$  then  (\ref{eqT=T'})
is equivalent to  $\str{\vk,D}
\equiv_{\LII} \str{\vk',D'},$ whence we obtain  $\vk
\equiv_{\LII} \vk'$ and $D \equiv_{\LII} D'.$
The converse is not true. Indeed,
consider a couple of distinct  $\LII$-equivalent
cardinals  $\vk,\vk'.$  Let  $D$  be a division
ring of power  $\vk.$ Then  $\str{\vk,D}
\not\equiv_{\LII} \str{\vk',D}$ by \ref{T=T'=>...}(a).

Some simplification can be also obtained in the case $\vk \le  |D|.$
We claim that
\begin{align*}
\TH(D,\LII(\vk^+)) \le & \TH(\str{\vk,D},\LII(\vk^+)) \le  \\
&\TH(\str{|D|,D},\LII(\vk^+)) \le \TH(D,\LII(\vk^+)).
\end{align*}
The first sort of the structure $\str{|D|,D}$  can be
identified in $D$  with the set  $D \times \{1\},$  and the
second one with the set $D \times \{0\}.$ Hence if
(\ref{eqT=T'}) is true and  $\vk =|D|,$ then $\vk' =|D'|$ and
$D \equiv_{\LII}D'.$ The condition (\ref{eqT=T'}), along
with the condition $\vk < |D|,$ is equivalent by Claim \ref{T=T'=>...}(c)
and the above arguments to the conditions  $\vk' < |D'|$  and
$\TH(D,\LII (\vk^+)) =\TH(D',\LII(\vk'{}^+)).$

Note one important particular case: if a division ring
$D$ is characterized up to isomorphism by a single
sentence of the full second order logic, then under
the condition $\vk \ge  |D|,$ (\ref{eqT=T'}) is
equivalent to conditions $\vk' \ge |D'|,$ $\vk
\equiv_{\LII} \vk',$  and $D \cong D'.$ Examples of
such $D$ are the fields $\Q,\R$  and  $\C,$
countable algebraically closed fields and finite
fields.

It can be deduced from \ref{T=T'=>...}(d) that if one
of the cardinals $\vk,\vk'$ is $\LII$-definable, then
$\eqref{eqT=T'} \Rightarrow \vk = \vk'.$ Obvious
examples are all cardinals $\aleph_n,$ where $n \in \omega.$

Since the cardinal $\aleph_0$ is $\LII$-definable,
and the field of reals $\R$ can be described up to isomorphism
by a single $\LII(\aleph_1)$-sentence (the field  $\Q$
is $\LII(\aleph_1)$-definable, and $\R$ can be reconstructed
from $\Q,$ using Dedekind cuts), then
$$
T(\aleph_0,\R) = T(\vk,D) \Leftrightarrow
 \vk = \aleph_0  \text{ and }  D \cong  \R.
$$
In contrast, $\C$ cannot be determined in a such way:
$T(\aleph_0,\C) = T(\vk,D)$ iff
$\vk = \aleph_0$  and  $D$ is an uncountable algebraically
closed field of  characteric zero. Necessity: as
$|\C| > \aleph_0,$ then  $|D| > \aleph_0$  by
Claim \ref{T=T'=>...}(b). It follows from  $T(\aleph_0,\C) = T(\vk,D),$
that $\TH(\C) = \TH(D);$ therefore
$D$ is an algebraically closed field of characteristic
zero. The sufficiency is an immediate consequence
of the following lemma from the joint paper
by Belegradek and the author \cite{BT}.

\begin{Lem} \label{LemFromBT} Let  $T$ be an uncountably
categorical first order theory and $\vk \ge \aleph_0.$
Then all models of $T$  of power $> \vk$ are
$\LII(\vk^+)$-equivalent.
\end{Lem}

Proof of \ref{LemFromBT} (using standard model-theoretic
techniques) is based on the fact that if
$\cM,\cN$ are two models of  $T$ of power  $> \vk$  and
$\cM \prec  \cN,$ then $\cM^{\vk^+}_{\text{II}}
\prec \cN^{\vk^+}_{\text{II}}.$

In particular, if $V, V'$ are vector spaces of
dimension $\aleph_0$  over uncountable algebraically
closed fields of the same characteristic, then $\gl V
\equiv  \gl{V'}.$ We shall see now that this result is
not true for groups of type $\Gamma$L.

Consider the vector space $\bigoplus_{i < \vk} D$ over
a division ring $D.$ The group of type  $H$ over this
space is denoted by  $H(\vk,D).$ It follows from
Theorem \ref{MnThForGls} that the logical power of the
theory $\TH(\Gl{\vk,D})$ grows with the growth of
$\vk.$ We shall now demonstrate that the growth of
logical power of $\TH(\Gl{\vk,D})$ can be also
achieved by exploiting the second `parameter', the
underlying division ring.

\begin{Prop} \label{GlsOverACFs}   Let  $K$ be
an algebraically closed field of infinite
transcendence degree over the prime field.
Assume that  $\vk$ is an infinite cardinal. Then
$\TH_2(K),$ the full second order
theory of the field  $K,$
is syntactically interpretable in $\TH(\Gl{\vk,K})$
{\rm(}in $\TH(\pGl{\vk,K}),$ uniformly in $\vk.$
\end{Prop}

\begin{proof}
According to Theorems 1.6 and 3.1 from the paper
\cite{MRRS}, if a field $K$ satisfies the conditions
of Theorem \ref{GlsOverACFs}, then the full second
order theory of the set $|K|$ can be syntactically
interpreted in the elementary theory of the lattice of
all algebraically closed subfields of $K.$

We need the following well-known fact.

\begin{Lem}
Every algebraically closed subfield $k$ of
$K$ is the fixed field of some automorphism of $K.$
\end{Lem}

By Theorem \ref{MnThForGls}  we can
interpret in $\TH(\Gl{\vk,K})$ the elementary theory of the structure
${\mathcal K}_{\vk}=\str{\,\str{\vk,K}^{\vk^+}_{\text{II}}, \aut K\,}.$
Let us build an interpretation of the lattice
of algebraically closed subfields of $K$ in ${\mathcal K}_{\vk}.$
The above remarks reduce our task to
finding a definable condition $\chi$ such that
$\models  \chi[\sigma]$ iff $\sigma \in
\aut K$  and the fixed field of $\sigma$ is
algebraically closed. For this purpose, it is
enough to model the situation `$\mu$ is a
root of a polynomial $f(x) = \lambda_0+ \lambda_1x
+\ldots + \lambda_n x^n$'.

The ordered tuple
$\str{\lambda_0,\lambda_1,\ldots,\lambda_n}$  can be
coded by the quadruple  $\str{\lambda_0,\lambda_n,A,g},$
where  $A = \{\lambda_0,\lambda_1,\ldots,\lambda_n\},$
and $g$ is a partial function such that $g(\lambda_i)=
\lambda_{i-1},i = 1,\ldots,n$  (the condition `a finite
set $A$ is an orbit of $g$' can be described by a
single sentence in the language of the structure
$\str{\vk,D}^{\vk^+}_{\text{II}}$).

An element  $\mu \in K$  is a root of  $f(x)$ iff
the following definable condition holds
\begin{align*}
(\exists h)& (h \text{ is a function from }  K  \text{ to }  K \wedge
h(\lambda_n) = \lambda_n \wedge \\
& (\forall \lambda',\lambda''\in A)
[\lambda'\ne \lambda_0 \wedge g(\lambda')=\lambda''  \rightarrow
h(\lambda'') =h(\lambda')\mu +\lambda''] \wedge
h(\lambda_0) =0).
\end{align*}
Indeed, if the latter condition is satisfied,
$h(\lambda_{n-1}) = \lambda_n\mu +\lambda_{n-1},
h(\lambda_{n-2}) = \lambda_n\mu ^2+\lambda_{n-1}\mu
+\lambda_{n-2},$ and so on. Hence $h(\lambda_0) =
f(\mu) =0.$

So we can work with the theory  $\TH_2(|K|).$
We now interpret $\TH_2(K)$ in this theory. Choose
subsets $X_1,X_2$ of $|K|$ with $X_1 \subseteq  X_2,$
so that  $X_2$ has the power $|K|.$
Choose further binary relations $R_i,S_i$ on
$X_i,$ where $i=1,2$ such that

(i) the structure $\str{X_1;R_1,S_1}$ is isomorphic
to the prime field of $K;$

(ii) the structure $\str{X_2;R_2,S_2}$ is an algebraically
closed field;

(iii)  $\str{X_1;R_1,S_1}$ is a substructure of
$\str{X_2;R_2,S_2};$

(iv)  the transcendence degree of $\str{X_2;R_2,S_2}$
over $\str{X_1;R_1,S_1}$ is $|K|.$

Clearly,  if  ($\LII$-definable) conditions
(i-iv) are true, then the structure  $\str{X_2;\ldots}$
is isomorphic to $K.$
\end{proof}

Thus, we see that the logical power of the theory
$\TH(\Gl V)$ can be significantly higher than the
logical power of $\TH(\gl V).$

We summarize some of our results.

\begin{Prop}

\mbox{\rm (a)} $\gl{\aleph_0,\R} \equiv  \gl{\vk,D}$
if and only if $\vk = \aleph_0$ and $D \cong  \R;$

{\rm (b)} $\gl{\aleph_0,\C} \equiv  \gl{\vk,D}$ if and only if
$\vk = \aleph_0$ and  $D$ is an uncountable algebraically
closed field of characteristic zero;

{\rm (c)} $\Gl{\aleph_0,\C} \equiv  \Gl{\vk,D}$ if and
only if $\vk = \aleph_0$ and  $D \cong  \C;$
\end{Prop}

\begin{proof}
By Lemma \ref{LemFromBT} and Proposition \ref{GlsOverACFs}.
\end{proof}

So the condition  $T(\vk,D) = T(\vk',D')$ is not
sufficient for the elementary equivalence of groups
$\Gl{\vk,D}$ and $\Gl{\vk',D'}.$

\section*{Acknowledgments}

This paper is a part of my Cand. Sci. thesis
\cite{To2}, written under the direction of Professor
Oleg Belegradek. I would like to express my deep
appreciation and thanks to him for his advice,
encouragement and fruitful co-operation. I would also
like to thank the referee for a detailed and
stimulating report.\\[2mm]

\end{document}